\documentclass[12pt]{article}

\usepackage{times}

\usepackage[shortlabels]{enumitem}
\usepackage{enumitem}
\usepackage[utf8]{inputenc}
\usepackage{commath}
\usepackage{amsthm, amscd}
\usepackage{amsmath}
\usepackage{amssymb}
\usepackage{cite}
\usepackage{setspace}
\usepackage{ stmaryrd }
\usepackage{tikz-cd}

\usepackage{tocloft}
\usepackage{sectsty}
\usepackage{xparse}

\newcommand*\tasklabelformat[1]{#1)}

\usepackage{titlesec}

\usepackage{tasks}[newest]
\settasks{
  label = \alph* ,
  label-format = \tasklabelformat ,
  label-width  = 12pt
}

\usepackage{bigints}
\usepackage{comment}
\usepackage{mathtools}
\usepackage{mathrsfs}
\usepackage{fancyhdr}

\allowdisplaybreaks[4]

\numberwithin{equation}{section}
\usepackage{etoolbox}
\patchcmd{\thebibliography}
  {\settowidth}
  {\setlength{\itemsep}{0pt plus -10pt}\settowidth}
  {}{}
\apptocmd{\thebibliography}
  {
  }
  {}{}
\makeatletter
\newtheorem*{rep@theorem}{\rep@title}
\newcommand{\newreptheorem}[2]{%
\newenvironment{rep#1}[1]{%
 \def\rep@title{#2 \ref{##1}}%
 \begin{rep@theorem}}%
 {\end{rep@theorem}}}
\makeatother

\theoremstyle{theorem}

\newreptheorem{theorem}{Theorem}
\newtheorem{thm}{Theorem}[section]
\newtheorem*{thm*}{Theorem}
\theoremstyle{definition}
\newtheorem{prop}[thm]{Proposition}
\newtheorem*{prop*}{Proposition}
\newtheorem{defn}[thm]{Definition}
\newtheorem{lem}[thm]{Lemma}
\newtheorem{cor}[thm]{Corollary}
\newtheorem*{cor*}{Corollary}
\theoremstyle{remark}

\newtheorem{rem}[thm]{Remark}

\title{\vspace*{-1.5cm} Geometric quantization on big line bundles}

\author
{Siarhei Finski
}

\date{}

\usepackage[%
    left=1in,%
    right=1in,%
    top=1.1in,%
    bottom=0.8in,%
    paperheight=11in,%
    paperwidth=8.5in%
]{geometry}

\newcommand{\imun} {\sqrt{-1}}
\newcommand{\vol}{v}

\newcommand{\res}{{\rm{Res}}}

\newcommand{\sym}{{\rm{Sym}}}

\newcommand{\comp}{\mathbb{C}}
\newcommand{\real}{\mathbb{R}}

\newcommand{\nat}{\mathbb{N}}

\newcommand{\enmr}[1]{\text{End}{(#1)}}

\newcommand{\ccal}{\mathscr{C}}

\newcommand{\dbar}{ \overline{\partial} }

\newcommand{\ddc}{\mathrm{d} \mathrm{d}^c}

\renewcommand{\Re}{\operatorname{Re}}
\renewcommand{\Im}{\operatorname{Im}}

\newcommand{\ban}{\operatorname{Ban}}
\newcommand{\hilb}{\operatorname{Hilb}}
\newcommand{\psh}{\operatorname{PSH}}

\makeatletter
\DeclareFontFamily{OMX}{MnSymbolE}{}
\DeclareSymbolFont{MnLargeSymbols}{OMX}{MnSymbolE}{m}{n}
\SetSymbolFont{MnLargeSymbols}{bold}{OMX}{MnSymbolE}{b}{n}
\DeclareFontShape{OMX}{MnSymbolE}{m}{n}{
    <-6>  MnSymbolE5
   <6-7>  MnSymbolE6
   <7-8>  MnSymbolE7
   <8-9>  MnSymbolE8
   <9-10> MnSymbolE9
  <10-12> MnSymbolE10
  <12->   MnSymbolE12
}{}
\DeclareFontShape{OMX}{MnSymbolE}{b}{n}{
    <-6>  MnSymbolE-Bold5
   <6-7>  MnSymbolE-Bold6
   <7-8>  MnSymbolE-Bold7
   <8-9>  MnSymbolE-Bold8
   <9-10> MnSymbolE-Bold9
  <10-12> MnSymbolE-Bold10
  <12->   MnSymbolE-Bold12
}{}

\let\llangle\@undefined
\let\rrangle\@undefined
\DeclareMathDelimiter{\llangle}{\mathopen}%
                     {MnLargeSymbols}{'164}{MnLargeSymbols}{'164}
\DeclareMathDelimiter{\rrangle}{\mathclose}%
                     {MnLargeSymbols}{'171}{MnLargeSymbols}{'171}
\makeatother

\newenvironment{sciabstract}{}

\cftsetindents{section}{0em}{1.7em}

\sectionfont{\large}

\titlespacing*{\section}
{0pt}{5pt}{5pt}

\begin{document}

\maketitle 

\vspace*{-0.7cm}

\vspace*{0.3cm}

\begin{sciabstract}
  \textbf{Abstract.} 
	We extend several geometric quantization results to the setting of big line bundles. 
	More precisely, we prove the asymptotic isometry property for the map that associates to a metric on a big line bundle the corresponding sup-norms on the spaces of holomorphic sections of its tensor powers.
	Building on this, we show that submultiplicative norms on section rings of big line bundles are asymptotically equivalent to sup-norms.
	As an application, we show that any bounded submultiplicative filtration on the section ring of a big line bundle naturally gives rise to a Mabuchi geodesic ray, and the speed of this ray encodes the statistical invariants of the filtration.
\end{sciabstract}

\pagestyle{fancy}
\lhead{}
\chead{Geometric quantization on big line bundles}
\rhead{\thepage}
\cfoot{}


\newcommand{\Addresses}{{
  \bigskip
  \footnotesize
  \noindent \textsc{Siarhei Finski, CNRS-CMLS, École Polytechnique F-91128 Palaiseau Cedex, France.}\par\nopagebreak
  \noindent  \textit{E-mail }: \texttt{finski.siarhei@gmail.com}.
}} 

\vspace*{0.25cm}

\par\noindent\rule{1.25em}{0.4pt} \textbf{Table of contents} \hrulefill

\vspace*{-1.5cm}

\tableofcontents

\vspace*{-0.2cm}

\noindent \hrulefill


\section{Introduction}\label{sect_intro}
	The main goal of this paper is to show that several geometric quantization results, which were previously established in the ample case, extend to the setting of big line bundles.
	For this, we fix a complex projective manifold $X$, $\dim_{\comp} X = n$, and a \textit{big line bundle} $L$ over $X$, i.e., so that for 
	\begin{equation}
		n_k := \dim H^0(X, L^{\otimes k}),
	\end{equation}
	for some $c > 0$, we have $n_k \geq c \cdot k^n$, for sufficiently large $k \in \nat$.
	\par 
	The results of this paper are closely connected to \emph{Mabuchi geometry}, which consists of a family of the so-called Mabuchi-Darvas distances $d_p$, $p \in [1, +\infty[$, defined on the space of metrics on $L$ whose potentials are plurisubharmonic (psh) with minimal singularities. 
	We refer the reader to Section \ref{sect_min_sing} for the definition of potentials with minimal singularities, and to Section \ref{sect_mab_big} for the definition of $d_p$.
	Here, we only mention that the resulting metric space is geodesically complete.
	In the ample case, this theory was established through the works of Mabuchi \cite{Mabuchi}, Chen \cite{ChenGeodMab}, Darvas \cite{DarvWeakGeod}, and others; the extension to the big setting is a more recent development, culminated in the work of Gupta \cite{GuptaPrakhar} and building on the contributions of Di Nezza-Lu \cite{DiNezzaLuBigNef}, Xia \cite{XiaBig}, Darvas-Di Nezza-Lu \cite{DDLL1}, Trusiani \cite{TrusianiL1}.
	\par 
	Now, for any bounded metric $h^L$ on $L$ and any $k \in \nat$, one can define the associated sup-norm on $H^0(X, L^{\otimes k})$, which we denote by $\ban^{\infty}_k(h^L)$; here, $\ban$ stands for Banach, following the established convention of denoting the $L^2$-norm associated with $h^L$ and a measure $\mu$ by $\mathrm{Hilb}_k(h^L, \mu)$, where $\mathrm{Hilb}$ stands for Hilbert, see the Notation paragraph below for the formal definition.
	\par 
	The space of Hermitian norms on a given finitely dimensional vector space $V$ admits natural distances $d_p$, $p \in [1, +\infty[$.
	These distances can be extended to quasi-distances on the space of all norms on $V$ that we shall denote by the same letter $d_p$, see Section \ref{sect_quas_metr}.
	\par 
	In \cite{TianBerg}, Tian demonstrated that the space of smooth positive metrics on $L$ can be approximated, via the $\mathrm{Hilb}_k$-maps, by the spaces of norms on $H^0(X, L^{\otimes k})$ as $k \to \infty$. 
	As both of these spaces possess natural metrics, it is essential to examine how these metrics are related through the aforementioned approximation.
	\par 
	For ample line bundles, Phong-Sturm in \cite{PhongSturm} proved that one can construct the Mabuchi geodesics between two smooth positive metrics through geometric quantization, i.e., from the geodesics between the respective $\mathrm{Hilb}$-images.
	Chen-Sun \cite{ChenSunQuant} and Berndtsson \cite{BerndtProb} proved that a similar relation holds for Mabuchi-Darvas distances.
	One of the main goals of this paper is to show that versions of these results continue to hold without the ampleness assumption.
	\par 
	To state our first result, for a given continuous metric $h^L$ on $L$, we denote by $P(h^L)$ the associated envelope, defined as $P(h^L) := h^L \cdot \exp(- 2 V_{\theta})$, where $\theta := c_1(L, h^L)$, and
	\begin{equation}\label{eq_v_theta}
		V_{\theta} := \sup \big\{ \psi \in \psh(X, \theta) : \psi \leq 0 \big\}.
	\end{equation}	 
	Above, $\psh(X, \theta)$ is the space of $\theta$-psh functions on $X$, see the Notation paragraph below for the formal definition.
	As we shall recall in Section \ref{sect_fs_oper}, $V_{\theta} \in \psh(X, \theta)$, and $V_{\theta}$ has minimal singularities.
	This allows us to state our first result.
	\begin{thm}\label{thm_isom}
		For any continuous metrics $h^L_0$, $h^L_1$ on $L$, and any $p \in [1, +\infty[$, we have
		\begin{equation}\label{eq_thm_isom}
			d_p(P(h^L_0), P(h^L_1))
			=
			\lim_{k \to \infty}
			\frac{d_p(\ban^{\infty}_k(h^L_0), \ban^{\infty}_k(h^L_1))}{k}.
		\end{equation}
	\end{thm}
	\begin{rem}\label{rem_thm_isom}
		a) For $p = 1$, Theorem \ref{thm_isom} is related with the previous work of Berman-Boucksom \cite{BermanBouckBalls}.
		To explain this, we denote by $B_{k, i} \subset H^0(X, L^{\otimes k})$, $i = 1, 2$, the unit balls of the norms $\ban^{\infty}_k(h^L_i)$.
		For every $k \in \nat$, we fix a Hermitian norm $H_k$ on $H^0(X, L^{\otimes k})$, which allows us to calculate the volumes $\vol(\cdot)$ of measurable subsets in $H^0(X, L^{\otimes k})$. 
		Note that while such volumes depend on the choice of $H_k$, their ratio does not.
		Then, as we shall see in Section \ref{sect_quant_mab}, when $p = 1$, (\ref{eq_thm_isom}) implies
		\begin{equation}\label{eq_rem_thm_isom}
			d_1(P(h^L_0), P(h^L_1))
			=
			\lim_{k \to \infty}
			\frac{\log(\vol(B_{k, 0})) + \log(\vol(B_{k, 1})) - 2 \log(\vol(B_{k, 0} \cap B_{k, 1}))}{k \cdot n_k}.
		\end{equation}
		Recall that for two given metrics $h^L_0$, $h^L_1$ on $L$ with potentials of minimal singularities, one can define the difference between their Monge-Ampère energies as 
		\begin{equation}
			\mathscr{E}(h^L_1) - \mathscr{E}(h^L_0)
			=
			\frac{1}{2(n + 1) \cdot {\rm{vol}}(L)} \sum_{i = 0}^{n} \int \log(h^L_0 / h^L_1) \cdot c_1(L, h^L_0)^i \wedge c_1(L, h^L_0)^{n - i},
		\end{equation}
		where the volume of a line bundle $L$ is defined as
		\begin{equation}\label{defn_vol}
			{\rm{vol}}(L) := \limsup_{k \to \infty} \frac{n!}{k^n} n_k,
		\end{equation}
		and the intersection products above are to be interpreted in the sense of the non-pluripolar product, see \cite{BEGZ}.
		Note that Fujita's theorem \cite{Fujita} states that the $\limsup$ in (\ref{defn_vol}) is actually a limit.
		\par 
		Recall, cf. \cite[Theorem 5.7]{GuptaPrakhar}, that for any metrics $h^L_0$, $h^L_1$ on $L$ with psh potentials with minimal singularities, verifying $h^L_0 \geq h^L_1$, the following relation holds
		\begin{equation}
			d_1(h^L_0, h^L_1)
			=
			\mathscr{E}(h^L_1) - \mathscr{E}(h^L_0).
		\end{equation}
		Due to the above fact, we see that for any continuous metrics $h^L_0, h^L_1$ on $L$ verifying $h^L_0 \geq h^L_1$, (\ref{eq_rem_thm_isom}) recovers the formula 
		\begin{equation}\label{rem_thm_isom2}
			\mathscr{E}(P(h^L_1)) - \mathscr{E}(P(h^L_0))
			=
			\lim_{k \to \infty}
			\frac{\log(\vol(B_{k, 1})) - \log(\vol(B_{k, 0}))}{k \cdot n_k}.
		\end{equation}
		We note that even if $h^L_0$ and $h^L_1$ are not ordered as above, (\ref{rem_thm_isom2}) still holds.
		Indeed, replacing $h^L_0$ by $C \cdot h^L_0$ for a sufficiently large $C > 0$, one reduces to the ordered case, as both the volume and the energy behave predictably under such a scaling.
		The formula (\ref{rem_thm_isom2}) is precisely the main result of \cite{BermanBouckBalls}.
		Our work therefore provides an alternative proof of \cite{BermanBouckBalls} and extends its main theorem from the $d_1$-distance to $d_p$, for any $p \in [1, +\infty[$.
		\par
		\begin{sloppypar} 
		b)
		By the Bernstein-Markov property, cf. Proposition \ref{prop_bm_volume}, the norms $\ban^{\infty}_k(h^L_0)$ and $\ban^{\infty}_k(h^L_1)$ on the right-hand side of (\ref{eq_thm_isom}) can just as well be replaced by $\hilb_k(h^L_0, dV_X)$ and $\hilb_k(h^L_1, dV_X)$ for an arbitrary volume form $dV_X$, see Section \ref{sect_ample_case} for details. 
		In the ample case, for positive and smooth $h^L_0$ and $h^L_1$, the above version of Theorem \ref{thm_isom} was established by Chen-Sun \cite{ChenSunQuant} for $p = 2$ and Berndtsson \cite{BerndtProb} for $p \in [1, +\infty[$.
		The regularity assumptions on $h^L_0$ and $h^L_1$ were further relaxed by Darvas-Lu-Rubinstein \cite{DarvLuRub}.
		\par 
		Moreover, the analogues of (\ref{rem_thm_isom2}) for $L^2$-norms on $H^0(X, L^{\otimes k} \otimes K_X)$ were established by Berman-Freixas \cite{BerFreix} for ample $L$ and metrics $h^L_i$ with psh potentials in the finite energy space, and more recently by Hou \cite{YuChiHou} for big and carrying a semi-positive metric $L$ and metrics $h^L_i$ with bounded psh potentials.
		\par 
		We choose to formulate our results in terms of sup-norms rather than $L^2$-norms, since this is the form required for later applications.
		Moreover, sup-norms are often more convenient to work with, as they behave better under the formation of psh envelopes, see (\ref{eq_max_prin}), under restriction to subvarieties (even in the singular setting), see Theorem \ref{thm_bost_ext}, and they are submultiplicative.
		\end{sloppypar} 
	\end{rem}
	\par 
	In \cite[Remark 1.4c)]{FinNarSim}, the author raised the question of whether the theory of submultiplicative norms developed in \cite{FinNarSim} for ample line bundles extends to the big case. 
	By relying on Theorem \ref{thm_isom}, we answer this question affirmatively below. 
	To state our result, we define the \textit{section ring}
	\begin{equation}
		R(X, L) := \oplus_{k = 0}^{\infty} H^0(X, L^{\otimes k}).
	\end{equation}
	A graded norm $N = (N_k)_{k = 0}^{\infty}$, $N_k := \| \cdot \|_k$, over $R(X, L)$ is called \textit{submultiplicative} if for any $k, l \in \nat^*$, $f \in H^0(X, L^{\otimes k})$, $g \in H^0(X, L^{\otimes l})$, we have
	\begin{equation}\label{eq_subm_s_ring}
		\| f \cdot g \|_{k + l} \leq 
		\| f \|_k \cdot
		\| g \|_l.
	\end{equation}
	\par 
	\begin{sloppypar}
	As a basic example, any bounded metric $h^L$ on $L$ induces the sup-norm $\ban^{\infty}(h^L) := (\ban^{\infty}_k(h^L))_{k = 0}^{\infty}$, which is clearly submultiplicative.
	We show below that, assuming a suitable boundedness condition on $N$, sup-norms are, up to a natural equivalence relation, the only submultiplicative norms.
The relevant equivalence relation is the following: two graded norms $N = (N_k)_{k = 0}^{\infty}$ and $N' = (N_k')_{k = 0}^{\infty}$ are $p$-\textit{equivalent} ($N \sim_p N'$), $p \in [1, +\infty[$, if
 	\begin{equation}\label{defn_equiv_relp}
		\frac{1}{k} d_p(N_k, N_k') \to 0, \qquad \text{as } k \to \infty.
	\end{equation}
	We say that a graded norm $N$ on $R(X, L)$ is \textit{bounded} if
	\begin{equation}\label{eq_bound_metr}
		\ban^{\infty}_k(h^L_0) \leq N_k \leq \ban^{\infty}_k(h^L_1),
	\end{equation}
	for certain bounded metrics $h^L_0$, $h^L_1$ on $L$.
	We underline that in the ample case, our definition of the bounded norm in \cite{FinNarSim} did not require the upper bound from (\ref{eq_bound_metr}).
	This is due to the fact that the upper bound in the ample case follows from the submultiplicativity assumption and the fact that $R(X, L)$ is finitely generated, see \cite[(3.16)]{FinNarSim}.
	As $R(X, L)$ is not finitely generated for arbitrary big line bundles, one can easily see that the upper bound of (\ref{eq_bound_metr}) in the big setting doesn't follow from  the submultiplicativity, see the end of Section \ref{sect_sm_sing} for details.
	\end{sloppypar}
	\par 
	To explain the construction of the metric corresponding to a submultiplicative norm, recall that any norm $N_k$ on $H^0(X, L^{\otimes k})$ induces the Fubini-Study metric $FS(N_k)$ on $L^{\otimes k}$, see Section \ref{sect_fs_oper}.
	For $L$ ample, the resulting metric is continuous, but for big $L$ it might have singularities.
	Nevertheless, for any bounded submultiplicative norm $N$, the sequence of metrics $FS(N_k)^{\frac{1}{k}}$ converges pointwise, as $k \to \infty$, to a metric on $L$, which we denote by $FS(N)$, cf. Lemma \ref{lem_fs_sm}. 
	In Section \ref{sect_fs_oper}, we show that the lower semicontinuous regularization $FS(N)_*$ of $FS(N)$ has a psh potential with minimal singularities.
	While we initially defined the sup-norms solely for bounded metrics, the definition makes sense for metrics with potentials of minimal singularities, as we explain in Section \ref{sect_fs_oper}.
	This allows us to state our second result.
	\begin{thm}\label{thm_char}
		Assume that a graded norm $N = (N_k)_{k = 0}^{\infty}$ over the section ring $R(X, L)$ of a big line bundle $L$ is submultiplicative and bounded.
		Then for any $p \in [1, + \infty[$, we have
		\begin{equation}\label{eq_char}
			N \sim_p \ban^{\infty}(FS(N)_*).
		\end{equation}
	\end{thm}
	\begin{rem}\label{rem_char}
		a) If a submultiplicative norm is replaced by a submultiplicative filtration, Rees \cite{ReesValI, ReesValII} and Boucksom-Jonsson \cite{BouckJohn21} proved the analogous result in the ample setting. 
		Reboulet \cite{Reboulet21} further extended this theory to arbitrary non-Archimedean norms.
		\par 
		b) For ample $L$, Theorem \ref{thm_char} was established by the author in \cite{FinNarSim}.
		In Theorem \ref{thm_char2}, we present a stronger version of Theorem \ref{thm_char} for the case where $L$ is ample on a singular variety $X$, thereby extending the result of \cite{FinNarSim}, which was proven only for the non-singular case.
	\end{rem}
	\par 
	It is natural to ask if $FS(N)_*$ is the unique metric which can be put on the right-hand side of (\ref{eq_char}).
	While, as in the case of ample line bundles, the statement does not hold when one allows arbitrary metrics on the line bundle, we shall explain in Section \ref{sect_reg_metr} that it becomes valid if we restrict our attention to the so-called regularizable from above psh metrics.
	Moreover, Proposition \ref{prop_fs_regul_blw} implies that $FS(N)_*$ is a regularizable from above psh metric.
	\par 
	We now describe two applications of Theorems \ref{thm_isom} and \ref{thm_char}. 
	Before doing so, we note that both applications rely on the uniqueness of metric geodesics associated with $d_p$, $p \in ]1, +\infty[$, which in the ample case is due to Darvas-Lu \cite{DarLuGeod}, and in the big setting to Gupta \cite{GuptaPrakhar}. 
	For the $d_1$-distance, the geodesics are not unique; consequently, the $d_1$-versions of Theorem \ref{thm_isom} mentioned in Remark \ref{rem_thm_isom} do not appear to imply our corollaries.
	\par 
	The first application concerns the relation between the geodesics in the space of metrics and norms.
	More specifically, we fix two continuous metrics $h^{L, 0}_0$, $h^{L, 0}_1$ on $L$.
	We denote by $h^L_t$ the Mabuchi geodesic between the associated envelopes $P(h^{L, 0}_0)$, $P(h^{L, 0}_1)$, which are defined as in (\ref{eq_v_theta}).
	For any $k \in \nat$, $t \in [0, 1]$, we consider the complex interpolation $N_{k, t}$ between $\ban^{\infty}(h^{L, 0}_0)$ and $\ban^{\infty}(h^{L, 0}_1)$, see (\ref{eq_defn_cx_inter}) for the definition of complex interpolation.
	\begin{thm}\label{cor_appr_geod}
		For any $t \in [0, 1]$, the sequence of metrics $FS(N_{k, t})^{\frac{1}{k}}$ on $L$ converges outside of a pluripolar subset towards $h^L_t$, as $k \to \infty$.
	\end{thm}
	\begin{rem}
		As we shall explain in Section \ref{sect_ample_case}, the norm $N_{k, t}$ above can just as well be replaced by the metric geodesic between $\hilb_k(h^{L, 0}_0, dV_X)$ and $\hilb_k(h^{L, 1}_1, dV_X)$ for an arbitrary fixed volume form $dV_X$. 
		In such formulation, for ample line bundles, Theorem \ref{cor_appr_geod} was established by Phong-Sturm \cite{PhongSturm} and Berndtsson \cite{BerndtProb}.
		Note, however, that, unlike in the ample case, one should not expect that the convergence of $FS(N_{k, t})^{\frac{1}{k}}$ is uniform, as already for the endpoints $t = 0, 1$ it is not necessarily the case, see Remark \ref{rem_non_unif_gq} for details. 
		It is, however, reasonable to expect that the convergence takes place outside of the augmented base locus of $L$, see (\ref{eq_aug_bl}) for the definition, but we do not pursue it here.
	\end{rem}
	\par 
	Another application of Theorems \ref{thm_isom}, \ref{thm_char} lies in the theory of submultiplicative filtrations.
	Recall that a \textit{decreasing $\real$-filtration} $\mathcal{F}$ of a vector space $V$ is a map from $\real$ to vector subspaces of $V$, $t \mapsto \mathcal{F}^t V$, verifying $\mathcal{F}^t V \subset \mathcal{F}^s V$ for $t > s$, and such that $\mathcal{F}^t V  = V$ for sufficiently small $t$ and $\mathcal{F}^t V = \{0\}$ for sufficiently large $t$.
	It is \textit{left-continuous} if for any $t \in \real$, there is $\epsilon_0 > 0$ such that $\mathcal{F}^t V = \mathcal{F}^{t - \epsilon} V $ for any $0 < \epsilon < \epsilon_0$.
	A filtration $\mathcal{F}$ on $R(X, L)$ is a collection $(\mathcal{F}_k)_{k = 0}^{\infty}$ of decreasing left-continuous filtrations $\mathcal{F}_k$ on $H^0(X, L^{\otimes k})$.
	We say that $\mathcal{F}$ is \textit{submultiplicative} if for any $t, s \in \real$, $k, l \in \nat$ we have 
	\begin{equation}\label{eq_sumb_filt}
		\mathcal{F}^t_k H^0(X, L^{\otimes k}) \cdot \mathcal{F}^s_l H^0(X, L^{\otimes l}) \subset \mathcal{F}^{t + s}_{k + l} H^0(X, L^{\otimes (k + l)}).
	\end{equation}
	In what follows, we sometimes omit the subscript from $\mathcal{F}_k$ when it is clear from the context.
	\par 
	We say that $\mathcal{F}$ is \textit{bounded} if there is $C > 0$ such that for any $k \in \nat^*$, we have
	\begin{equation}\label{eq_bnd_filt}
		\mathcal{F}^{ C k} H^0(X, L^{\otimes k}) = \{0\}, \qquad \mathcal{F}^{- C k} H^0(X, L^{\otimes k}) = H^0(X, L^{\otimes k}).
	\end{equation}
	Similarly to the remark we already made for norms, in the ample case, the usual definition of bounded filtrations does not require the upper bound from (\ref{eq_bnd_filt}), as it follows from submultiplicativity and the finite generation of $R(X, L)$, cf. \cite[Example 2.1.30]{LazarBookI}.
	\par 
	The theory of submultiplicative filtrations on rings was initiated by Samuel \cite{SamuelAsymptIdeals} and subsequently developed from an algebraic perspective by Nagata \cite{NagataSamuelConj}, Rees \cite{ReesBook}, and others.
	In complex geometry, interest in submultiplicative filtrations emerged from their connection to the theory of canonical metrics, as highlighted in the study of constant scalar curvature Kähler metrics in \cite{SzekeTestConf} and \cite{NystOkounTest}.
	Remarkably, the Mabuchi geometry -- originally developed to address the constant scalar curvature equation -- now provides new insights and results concerning submultiplicative filtrations.
	Moreover, submultiplicative filtrations have also found applications in other versions of canonical metrics, such as in the study of the Wess-Zumino-Witten equation \cite{FinHYM}.
	\par 
	Let us explain how Mabuchi geometry arises in the study of submultiplicative filtrations. 
	In Section \ref{sect_geod_ray}, for any continuous metric $h^L$ on $L$ and a bounded submultiplicative filtration $\mathcal{F}$, we construct a ray of metrics $h^L_t$, $t \in [0, +\infty[$, on $L$ departing from $P(h^L)$.
	We show that for any $t \in [0, +\infty[$, $h^L_t$ has a psh potential with minimal singularities, and, remarkably, the ray $h^L_t$ is a geodesic ray with respect to all distances $d_p$, $p \in [1, +\infty[$.
	In the ample setting, the above construction was carried out by Phong-Sturm \cite{PhongSturmTestGeodK} for finitely generated filtrations (i.e., those associated with a test configuration) and then by Ross-Witt Nystr{\"o}m \cite{RossNystAnalTConf} for arbitrary bounded submultiplicative filtrations.
	\par 
	\begin{sloppypar}
	Of course, the above construction is useful only if the resulting ray captures sufficient information about the filtration. 
	In our final result, we show that the speed of this ray encodes all of the statistical invariants of the filtration.
	To explain the latter point in detail, we define the sequence of \textit{jumping measures} $\mu_{\mathcal{F}, k}$, $k \in \nat^*$, on $\real$ of $\mathcal{F}$ as follows
	\begin{equation}\label{eq_jump_meas_d}
		\mu_{\mathcal{F}, k} := \frac{1}{n_k} \sum_{j = 1}^{n_k} \delta \Big[ \frac{e_{\mathcal{F}}(j, k)}{k} \Big], 
	\end{equation}
	where $\delta[x]$ is the Dirac mass at $x \in \real$ and $e_{\mathcal{F}}(j, k)$ are the \textit{jumping numbers}, defined as
	\begin{equation}\label{eq_defn_jump_numb}
		e_{\mathcal{F}}(j, k) := \sup \Big\{ t \in \real : \dim \mathcal{F}^t H^0(X, L^{\otimes k}) \geq j \Big\}.
	\end{equation}
	Now, for the above geodesic ray $h^L_t$, $t \in [0, +\infty[$, one can define its \textit{spectral measure} as 
	\begin{equation}\label{eq_muf_defn}
		\mu_{\mathcal{F}} := \Big(- \frac{\dot{h}^L_0}{2} \Big)_* \Big( \frac{c_1(L, h^L_0)^n}{ {\rm{vol}}(L) } \Big).
	\end{equation}
	The derivative $\dot{h}^L_0 := (h^L_t)^{-1} d h^L_t/d t|_{t = 0}$ from (\ref{eq_muf_defn}) is a bounded function defined using the convexity of the ray, see Section \ref{sect_mab_big}, despite the possible absence of regularity of $h^L_t$.
	When $L$ is ample, the metric $h^L_0$ is bounded, and the measure $c_1(L, h^L_0)^n$ above can be interpreted in the sense of Bedford-Taylor \cite{BedfordTaylor}.
	As we assume that $L$ is only big, $h^L_0$ is not necessarily bounded, and $c_1(L, h^L_0)^n$ is defined in the sense of non-pluripolar product of Boucksom-Eyssidieux-Guedj-Zeriahi \cite{BEGZ}.
	In either case, $\mu_{\mathcal{F}}$ is a probability measure: when $L$ is ample, this is a restatement of the asymptotic Riemann-Roch theorem, when $L$ is merely big, this follows from the result of Boucksom \cite{BouckVol}, which we recall in Theorem \ref{thm_bouck_vol}.
	\end{sloppypar}
	\begin{thm}\label{thm_filt}
		For any bounded submultiplicative filtration $\mathcal{F}$ on a section ring $R(X, L)$ of a big line bundle $L$, the jumping measures $\mu_{\mathcal{F}, k}$, $k \in \nat^*$, converge weakly, as $k \to \infty$, to $\mu_{\mathcal{F}}$.
	\end{thm}
	\begin{rem}
		The existence of the weak limit was proved by Boucksom-Chen \cite{BouckChen}, refining an earlier work of Chen \cite{ChenHNolyg}.
		The limiting measure can then be interpreted through Okounkov bodies \cite{BouckChen} and \cite{XiaPartialOkounkov}, but the main point of Theorem \ref{thm_filt} concerns a different expression using the associated geodesic ray.
		When the filtration is induced by an ample test configuration and $h^L$ is smooth and positive, the relation between the limiting measure and the geodesic ray from Theorem \ref{thm_filt} was established for product test configurations by Witt Nystr{\"o}m \cite[Theorems 1.1 and 1.4]{NystOkounTest}, and for general test configurations by Hisamoto in \cite[Theorem 1.1]{HisamSpecMeas}, proving a conjecture \cite[after Theorem 1.4]{NystOkounTest}.
		For arbitrary bounded submultiplicative filtrations on the section ring of ample line bundles, Theorem \ref{thm_filt} was established by the author in \cite{FinNarSim}.
		\par 
		In this article, we follow the approach from  \cite{FinNarSim}.
		In \cite{FinSubmToepl} the author also established a local analogue of Theorem \ref{thm_filt} in the ample case, relying on the analysis of the corresponding weighted Bergman kernel. 
		Extending this to the big setting is a natural direction for future work.
	\end{rem}
	\par 
	In conclusion, we briefly outline the main technical ideas underlying the proofs of Theorems \ref{thm_isom} and \ref{thm_char}.
	The key strategy in both cases is to reduce the problem to the ample setting. 
	This is achieved by constructing sufficiently large subrings of section rings of ample line bundles inside the section ring of the given big line bundle, and then showing that the distance between two norms on a vector space can be approximated by the distances between their restrictions to sufficiently large subspaces.
	For Theorem \ref{thm_char}, one must moreover extend our previous characterization of submultiplicative norms in the ample case to singular varieties and, in order to analyze the sup‐norms associated with a possibly non-continuous metric $FS(N)_*$, introduce a refined notion of singularity class for plurisubharmonic functions that sharpens the classical notion of analytic singularities.
	\par 
	This article is organized as follows.
	In Section \ref{sect_2}, we recall the necessary preliminaries.
	Section \ref{sect_3} contains the proof of Theorem \ref{thm_isom} as well as several technical results required in later sections.
	Section \ref{sect_4} is devoted to the proof of Theorem \ref{thm_char}.
	In Section \ref{sect_5}, we then establish a number of applications, among them Theorems \ref{cor_appr_geod} and  \ref{thm_filt}.
	\par
	\textbf{Notation.}
	We denote by $d = \partial + \dbar$ the usual decomposition of the exterior derivative in terms of its $(1, 0)$ and $(0, 1)$ parts, and we set 
	\begin{equation}
		d^c := \frac{\partial - \dbar}{2 \pi \imun}.	
	\end{equation}
	The Poincaré-Lelong formula gives us the following: for any smooth metric $h^L$ on a line bundle $L$, any $s \in H^0(X, L)$, $s \neq 0$, for the divisor $E := [s = 0]$, we have
	\begin{equation}\label{eq_poinc_lll}
		c_1(L, h^L) + \ddc \log |s(x)|_{h^L} = [E],
	\end{equation}
	where $c_1(L, h^L)$ is the first Chern form of $(L, h^L)$, defined as $\frac{\imun}{2 \pi} R^L$, where $R^L$ is the curvature of the Chern connection and $[E]$ is the current of integration along $E$.
	Note also that for an arbitrary smooth function $\phi: X \to \real$, we have
	\begin{equation}\label{eq_c1_bc}
		c_1(L, h^L \cdot \exp(-2 \phi))
		=
		c_1(L, h^L)
		+
		\ddc \phi.
	\end{equation}
	\par 
	We say that $\phi$ is a potential of a metric $h^L$ on $L$ with respect to another metric $h^L_0$ if 
	\begin{equation}
		h^L = h^L_0 \cdot \exp(-2 \phi).
	\end{equation}
	Occasionally, we will refer to a local potential of $h^L$, which is a potential with respect to a metric that trivializes a particular local holomorphic frame.
	By a continuous (resp. singular or bounded) metric on $L$ we mean a metric with a continuous (resp. $L^1_{loc}$ or bounded) potential.
	Through (\ref{eq_c1_bc}), for singular metrics $h^L$, we extend the definition of $c_1(L, h^L)$ as a $(1, 1)$-current.
	\par 
	Throughout the article, by a filtration on a finite-dimensional vector space $V$, we always mean a decreasing left-continuous filtration on $V$. 
	A norm on $V$ is called Hermitian if it is associated with a Hermitian product. 
	To better distinguish between norms (on vector spaces) and metrics (on line bundles), we prefer to use the terminology “metric" instead of “Hermitian metric", even though all the metrics in this article are indeed Hermitian.
	\par 
	By a metric $h^L$ on a $\mathbb{Q}$-line bundle, we shall always mean an equivalence class of metrics $h^{L^{\otimes kN}}$ on $L^{\otimes k N}$, where $k \in \mathbb{N}^*$ and $N \in \mathbb{N}^*$ is the minimal integer such that $L^{\otimes N}$ is a genuine line bundle.  
	Two metrics are considered equivalent if some of their powers coincide.  
	We then define $c_1(L, h^L)$ as $\frac{1}{N} c_1(L^{\otimes N}, h^{L^{\otimes N}})$, and one can freely speak of potentials, and so on.
	\par 
	Throughout the article, for a function $f: X \to [-\infty, +\infty]$, we denote by $f^*$ its upper semicontinuous regularization, defined for any $x \in X$ as 
	\begin{equation}
		f^*(x) := \lim_{\epsilon \to 0} \sup_{B(x, \epsilon)} f,
	\end{equation}
	where $B(x, \epsilon) \subset X$ is a ball of radius $\epsilon$ around $x$ (with respect to some fixed metric). 
	The lower semicontinuous regularization of a metric $h^L$, which we denote by $h^L_*$, is obtained by replacing the potential of $h^L$ with its upper semicontinuous regularization.
	\par 
	By a canonical singular metric $h^E_{{\rm{sing}}}$ on a line bundle $\mathscr{O}(E)$ associated with a divisor $E$ we mean a singular metric, defined such that for any $x \in X \setminus E$, we have $|s_E(x)|_{h^E_{{\rm{sing}}}} = 1$, where $s_E$ is the \textit{canonical holomorphic section} of $\mathscr{O}(E)$, verifying ${\rm{div}}(s_E) = E$.
	\par 	
	A function $\phi: X \to [-\infty, +\infty[$ is called quasi-plurisubharmonic (qpsh) if it can be locally written as the sum of a plurisubharmonic function and a smooth function. 
	For a smooth closed $(1, 1)$-form $\theta$, we say $\phi$ is $\theta$-plurisubharmonic ($\theta$-psh) if it is qpsh and $\theta + \ddc \phi > 0$ in the sense of currents.
	We let $\psh(X, \theta)$ denote the set of $\theta$-psh functions that are not identically $-\infty$.
	We say that $\phi \in \psh(X, \theta)$ is strictly $\theta$-psh if $\theta + \ddc \phi$ is a Kähler current, i.e., such that there is a Kähler form $\omega$ on $X$ so that $\theta + \ddc \phi \geq \omega$.	
	\par 
	When the class $[\alpha] \in H^2(X, \real) \cap H^{1, 1}(X)$ contains a Kähler current, such a class is called \textit{big}. When it contains a Kähler form, it is called \textit{Kähler}.
	We call $[\alpha]$ \textit{nef} if for a Kähler class $[\omega]$ and any $\epsilon > 0$, the class $[\alpha] + \epsilon [\omega]$ is Kähler.
	\par 
	For a bounded metric $h^L$ on $L$ and a positive Borel measure $\mu$, we denote by ${\textrm{Hilb}}_k(h^L, \mu)$ the positive semi-definite form on $H^0(X, L^{\otimes k})$ defined for arbitrary $s_1, s_2 \in H^0(X, L^{\otimes k})$ as
	\begin{equation}\label{eq_defn_l2}
			\langle s_1, s_2 \rangle_{{\textrm{Hilb}}_k(h^L, \mu)} = \int_X \langle s_1(x), s_2(x) \rangle_{(h^L)^{k}} \cdot d \mu(x).
	\end{equation}
	\par 
	For many objects in this article, if we wish to explicitly state the manifold to which the object is associated, we add it in the notation.
	In this way, for a holomorphic line bundle $L$ on a manifold $X$, we sometimes denote ${\rm{vol}}(L)$ by ${\rm{vol}}_X(L)$, and for a bounded metric $h^L$ on $L$, we sometimes write $\ban^{\infty}(X, h^L)$ instead of $\ban^{\infty}(h^L)$.
	\par
	\textbf{Acknowledgement.}
	 This work was supported by the CNRS, École Polytechnique, and in part by the ANR projects QCM (ANR-23-CE40-0021-01), AdAnAr (ANR-24-CE40-6184) and STENTOR (ANR-24-CE40-5905-01).
	 The author is grateful to the members and staff of Westlake University, China, for their warm hospitality during the preparation of a part of this work, and especially to Huayi Chen for his kind invitation.
	 The author also thanks Sébastien Boucksom and Tamás Darvas for pointing out the subtlety discussed in Remark \ref{rem_dem_appr}, and Yu-Chi Hou for informing us about his recent preprint \cite{YuChiHou} and pointing out several misprints in the current paper.
	 
	\section{Preliminaries from quantization, pluripotential theory and linear algebra}\label{sect_2}
	
	This section collects preliminary results from linear algebra, pluripotential theory, and geometric quantization.  
	Section \ref{sect_min_sing} covers potentials with minimal, analytic, and algebraic singularities.  
	Section \ref{sect_rel_pp} reviews relative pluripotential theory.  
	Section \ref{sect_mab_big} presents Mabuchi geometry.  
	Section \ref{sect_fs_oper} discusses the Fubini–Study operator and the link between norms and metrics.  
	Section \ref{sect_quas_metr} introduces a family of distances on the spaces of Hermitian norms on finite-dimensional vector spaces and their quasi-distance extensions to all norms.  
	Section \ref{sect_ample_case} summarizes geometric quantization results for ample line bundles and their relation with the main results.

	\subsection{Potentials with minimal and algebraic singularities}\label{sect_min_sing}
	
	Positive metrics on ample line bundles form a bridge between complex and algebraic geometry.
	When the line bundle is merely big but not ample, Kodaira's theorem implies that the line bundle admits no positive metrics.
	In this case, the role of positive metrics is taken over by Kähler currents, which will be the main focus of this section.
	\par 
	Similarly to Kodaira characterization of ample line bundles, it is known that a line bundle $L$ is big if and only if there is a closed positive $(1, 1)$-current in $c_1(L)$, moreover, a manifold carries a big line bundle if and only if it is Moishezon, see Bonavero \cite{Bonavero}, Ji-Shiffman \cite{JiShiffman}, Demailly \cite{DemIntroHodge}, cf. also \cite[Theorem 2.3.28]{MaHol}.
	While an arbitrary Moishezon manifold is bimeromorphic to a projective manifold, if the Moishezon manifold carries a Kähler metric, it is automatically projective, cf. \cite[Theorem 2.2.26]{MaHol}.
	\par 
	General Kähler currents are often difficult to analyze, but those with potentials having analytic singularities can be effectively studied using tools from algebraic geometry.
	We will refer to such currents as \textit{currents with analytic singularities}.
	Recall that a function $\psi: X \to [-\infty, +\infty[$ is said to have \textit{analytic singularities} if, locally, it can be written as
	\begin{equation}\label{eq_phi_alg_sing}
		\psi = \frac{c}{2} \cdot \log(\sum |f_i|^2) + g,
	\end{equation}
	where $c > 0$, $g$ is a locally bounded function, and the $f_i$ are local holomorphic functions.  
	We say that $\psi$ has \emph{neat analytic singularities} if the function $g$ above can be chosen to be smooth. 
	It is often more convenient to work with functions having \textit{(neat) algebraic singularities}: these are functions with (neat) analytic singularities for which the constant $c$ in (\ref{eq_phi_alg_sing}) can be chosen rational. 
	\par 
	We shall need the following basic result on these functions.
	\begin{prop}\label{prop_max_alg_sing}
		Assume that $\psi_i : X \to [-\infty, +\infty[$, $i = 0, 1$, have algebraic singularities. 
		Then $\max\{\psi_0, \psi_1\}$ also has algebraic singularities.
	\end{prop}
	\begin{rem}
		It is also immediate that a sum of two potentials with algebraic singularities has algebraic singularities.
	\end{rem}
	\begin{proof}
		While the proof is classical, cf. \cite[Proposition 4.1.8]{DemaillyOnCohPSEF}, we reproduce it for the reader's convenience and also because its arguments are employed in proving Proposition \ref{prop_max_alg_sing_weak_neat}, which refines Proposition \ref{prop_max_alg_sing}.
		Since $\phi_i$, $i = 0, 1$, have algebraic singularities, for some $p_i, q_i \in \nat^*$, local holomorphic functions $f_{i, j}$ and bounded $g_i$, we can write $\psi_i = \frac{p_i}{q_i} \cdot \log(\sum |f_{i, j}|^2) + g_i$.
		Note that for any $a, b > 0$, we have $\log(a+b) - \log(2) \leq \max\{ \log(a), \log(b) \} \leq \log(a + b)$. 
		Hence for a certain bounded function $g_2$, we can write
		\begin{equation}\label{eq_max_g3}
		\max \{ \psi_0, \psi_1 \}
		=
		\frac{1}{q_0 q_1} \log \Big(
			\big( \sum |f_{0, j}|^2 \big)^{p_0 q_1} + \big( \sum |f_{1, j}|^2 \big)^{p_1 q_0}
		\Big)
		+
		g_2.
		\end{equation}
		Once we expand the sum under the exponent, the proof concludes. 
	\end{proof}
	One of the central results in the theory of positive currents is Demailly's regularization theorem \cite{DemRegul}, which quite often reduces the study of arbitrary positive currents to those with algebraic singularities. 
	Since we will repeatedly use this theorem in what follows, and because several formulations exist in the literature, we explicitly state the precise version that will be needed here.
	\begin{thm}[{Demailly \cite{DemRegul} }]\label{thm_dem_appr}
		We fix a closed positive $(1,1)$-current $T$ on $X$, so that $T \geq \omega$ for some Kähler form $\omega$.
		We write $T = \theta + \ddc \psi$ for some smooth $(1,1)$-form $\theta$ and $\theta - \omega$-psh function $\psi$.
		Then there is a sequence $\epsilon_k > 0$, which decreases to $0$, and a decreasing sequence $\psi_k : X \to [-\infty, +\infty[$ of $\theta - \omega + \epsilon_k \omega$-psh functions with neat algebraic singularities on some birational model $\hat{X}_k$ of $X$ such that $\psi_k$ converges pointwise towards $\psi$.
		In particular, the currents $T_k = \theta + \ddc \psi_k$ are such that $T_k \geq (1 - \epsilon_k) \omega$, and $T_k$ converge weakly towards $T$, as $k \to \infty$.
	\end{thm}
	\begin{rem}\label{rem_dem_appr}
		The author was unable to find a proof of the more natural statement requiring $\psi_k$ to have neat algebraic singularities on $X$ itself rather than on a birational model.  
		\par 
		We note, however, that in this article, Theorem \ref{thm_dem_appr} is used exclusively for currents $T$ representing rational multiples of $c_1(L)$ of a holomorphic line bundle $L$ on $X$.
		For such currents, the aforementioned stronger form of the regularization theorem follows by restricting the approximating sequence $\phi_m + \frac{C}{m}$ from \cite[Theorem 14.21]{DemBookAnMet} to the subsequence $m = 2^k$, $k \in \nat$, for a suitable constant $C > 0$ corresponding to $C_8$ from \cite[p.18]{DPSPseudoeff}.
		\par 
		For the general results of Sections \ref{sect_rel_pp} and \ref{sect_mab_big}, whenever we say that a potential has (neat) algebraic (or analytic) singularities on $X$, we always mean that this property holds on some birational model of $X$.
		Note that any two birational models are dominated by a third one, and since our arguments involve only finitely many potentials, this poses no problem.
	\end{rem}
	\par 
	By blowing up the coherent ideal sheaf generated by the holomorphic functions $f_i$ as in (\ref{eq_phi_alg_sing}), and resolving the singularities of the blow-up, we see that for any function $\psi$ with analytic singularities on $X$, we can find a complex manifold $\hat{X}$, a birational model $\pi : \hat{X} \to X$ such that $\pi^* \psi$ has singularities along the divisor with normal crossings.
	More specifically, let $E = \sum_{i = 1}^{n} a_i E_i$, $a_i \in \mathbb{Q}$, $a_i \geq 0$, be a divisor on $\hat{X}$ such that for $c > 0$, we can write
	\begin{equation}\label{eq_resol_psi_sing}
		\psi \circ \pi(x)
		=
		c \cdot \log |s(x)|_{h^E} + g, 
	\end{equation}	 
	where $s = \prod_{i = 1}^{n} s_i^{a_i}$, $s_i$ are canonical holomorphic sections of $\mathscr{O}(E_i)$, $h^E$ is a smooth metric on the $\mathbb{Q}$-line bundle $\mathscr{O}(E) := \otimes_{i = 1}^{n} \mathscr{O}(a_i E_i)$, and $g$ is a bounded function (resp. smooth if $\psi$ has neat analytic singularities) on $\hat{X}$.
	Then, by the Poincaré-Lelong formula, see (\ref{eq_poinc_lll}), applied for $(\mathscr{O}(E), h^E)$, the following version of the Siu decomposition, cf. \cite[(2.18)]{DemBookAnMet}, takes place
	\begin{equation}\label{eq_siu_rel}
		\pi^* (\theta + \ddc \psi)
		=
		c \cdot [E] + \hat{\theta} + \ddc g,
	\end{equation}
	where $\hat{\theta} := \pi^* \theta - c \cdot c_1(\mathscr{O}(E), h^E)$, and $\hat{\theta} + \ddc g$ is a positive $(1, 1)$-current.
	Note that if $\psi$ has algebraic singularities, we can even take $c = 1$ upon rescaling $a_i \in \mathbb{Q}$ above.
	\begin{prop}\label{prop_b_i_beta}
		The cohomology class $[\hat{\theta}]$ of $\hat{\theta}$ is nef.
		Moreover, for any Kähler class $[\omega]$ on $X$, there are some rational $b_i > 0$, $i = 1, \ldots, n$, such that for any $\epsilon \in ]0, 1]$, the cohomology class $[\hat{\theta}] - \epsilon \sum b_i [E_i] + \epsilon \pi^* [\omega]$ is Kähler.
		\par 
		Finally, if $\psi$ has neat analytic singularities and the current $\theta + \ddc \psi$ is Kähler, then the $(1, 1)$-form $\hat{\theta} + \ddc g$ is semi-positive, and there is a closed $(1, 1)$-form $\beta$ in the class $-\sum b_i [E_i]$ such that for any $\epsilon \in ]0, 1]$, the form $\hat{\theta} + \ddc g + \epsilon \beta$ is strictly positive.
	\end{prop}	
	\begin{proof}
		We begin by proving the first part of the proposition, concerning general functions with analytic singularities.
		Since the current $\hat{\theta} + \ddc g$ is positive and has a bounded potential, the first claim follows directly from Theorem \ref{thm_dem_appr}.
		The second claim then follows from the first together with the well-known fact, see \cite[Proposition 2.1.8]{MaHol}, that there exist rational numbers $b_i > 0$ such that the class $- \sum b_i [E_i] + \pi^* [\omega]$ is Kähler.
		\par
		We now prove the second part of the proposition, which concerns Kähler currents and potentials with neat analytic singularities.
		From (\ref{eq_siu_rel}) and the fact the current $\theta + \ddc \psi$ is Kähler, we deduce that for any Kähler form $\omega$ on $X$, there exists $\epsilon > 0$ such that $\hat{\theta} + \ddc g \geq \epsilon \pi^* \omega$.
		This establishes the first claim.
		The second claim follows from the first and from the second statement in the first part of the proposition.
	\end{proof}
	\par 
	Potentials with analytic singularities are often convenient to work with, but there are many geometrically significant potentials that do not possess analytic singularities.
	Most notably, the potential $V_\theta$ defined in (\ref{eq_v_theta}) does not always have analytic singularities, as illustrated in \cite[Example 5.4]{BEGZ}.
	Nevertheless, $V_\theta$ has \textit{minimal singularities}, as explained below.
	\par 
	To define the potentials with minimal singularities, we first introduce a partial order on the space of $\theta$-psh functions on $X$. 
	We say that a $\theta$-psh function $\phi_0$ is \emph{more singular} than $\phi_1$ (and denote it by $\phi_0 \preceq \phi_1$) if there exists a constant $C > 0$ such that $\phi_0 \leq \phi_1 + C$.  
	Furthermore, we define the equivalence class of $\theta$-psh functions associated with $\phi_0$ with respect to this order as 
	\begin{equation}
		[\phi_0] = \Big\{ \psi \in \psh(X, \theta) : \phi_0 \preceq \psi \text{ and } \psi \preceq \phi_0 \Big\},
	\end{equation}
	and call it the \textit{singularity type} of $\phi_0$.
	\par 
	As has been observed by Demailly, cf. \cite{DPSPseudoeff}, in a given cohomological class $[\alpha] \in H^{1, 1}(X) \cap H^2(X, \real)$ carrying a positive current, there is always a closed positive current with a potential of minimal singularities.
	To construct such a current, consider an arbitrary current $T$ from $[\alpha]$ with a bounded potential (for example a smooth closed $(1, 1)$-form from $[\alpha]$), and consider $P(T) := T + \ddc V_T^*$, where $V_T$ is defined analogously to (\ref{eq_v_theta}).
	The fact that the resulting current has minimal singularities follows immediately from the definitions, and we see that the difference between two potentials with minimal singularities is bounded.
	\par 
	We extend the definition of minimal singularities to non-necessarily psh functions by saying that a function $\phi: X \to [-\infty, +\infty[$ has minimal singularities if $|\psi - V_T^*|$ is bounded.
	\par 
	Many invariants of big line bundles admit expressions in terms of the intersection theory of currents.
	Most importantly for this paper, such description exists for the volume of the line bundle.
	Recall that the intersection of positive currents with locally bounded potentials has been developed by Bedford-Taylor in \cite{BedfordTaylor}, and later extended for general positive currents by Boucksom-Eyssidieux-Guedj-Zeriahi in \cite{BEGZ}.
	All in all their construction produces a Radon measure $T^n$ on $X$ for every positive $(1, 1)$-current $T$.
	Let us recall some basic results from this theory; for this, we fix a smooth closed $(1, 1)$-form $\theta$ in a big class $[\theta]$. 
	\par 
	Immediately from the definition of the non-pluripolar product, we see that it is invariant under the pullback of a birational map.
	More specifically, if the map $\pi : Y \to X$ is birational, then for any $\phi \in \psh(X, \theta)$, we have
	\begin{equation}\label{eq_bir_nnpp}
		\pi_* (\pi^* \theta + \ddc \pi^* \phi)^n =  (\theta + \ddc \phi)^n.
	\end{equation}
	\par 
	\begin{thm}[{Witt Nystr{\"o}m \cite{WittNystrMonotonicity}, cf. also \cite{BEGZ}, \cite{LuChinhCompMA}, \cite{LuNguen}, \cite{VuDucRelNonpp} }]\label{thm_wn_monot}
		We fix $\phi_0, \phi_1 \in \psh(X, \theta)$, so that $\phi_0$ is less singular than $\phi_1$.
		Then, for $T_i := \theta + \ddc \phi_i$, $i = 0, 1$, we have
		\begin{equation}
			\int T_1^n \leq \int T_0^n.
		\end{equation}
	\end{thm}
	The next result makes a relation between the non-pluripolar product and the volume.
	\begin{thm}[{Boucksom \cite{BouckVol}}]\label{thm_bouck_vol}
		For any positive closed $(1, 1)$-current $T_{\min}$ with potential of minimal singularities in a class $[c_1(L)]$, where $L$ is a big line bundle, we have
		\begin{equation}
			\int T_{\min}^n = {\rm{vol}}(L).
		\end{equation}
	\end{thm}
	Following \cite{BouckVol}, we then extend the definition of the volume from line bundles to general big classes $[\theta] \in H^{1, 1}(X) \cap H^2(X, \real)$ as 
	\begin{equation}\label{eq_vol_new_defn}
		{\rm{vol}}([\theta]) := \int T_{\min}^n,
	\end{equation}
	where $T_{\min}$ is a positive closed $(1, 1)$-current with a potential of minimal singularities in $[\theta]$.
	From Theorem \ref{thm_wn_monot}, ${\rm{vol}}([\theta])$ is well-defined, and by Theorem \ref{thm_bouck_vol}, such a definition is compatible with the previously given one of the volume of a line bundle in the sense that ${\rm{vol}}(L) = {\rm{vol}}(c_1(L))$.
	One can extended the definition of the volume from (\ref{defn_vol}) to $\mathbb{Q}$-line bundles in a natural way, cf. \cite[Remark 2.2.39]{LazarBookI}, and such a definition is again compatible with (\ref{eq_vol_new_defn}).
	\begin{thm}[{Boucksom-Eyssidieux-Guedj-Zeriahi \cite[Theorem 2.17]{BEGZ} and Darvas-Di Nezza-Lu \cite[Theorem 2.3 and Remark 2.5]{DDLMonoton}}]\label{thm_cont_begz}
		We fix $\phi \in \psh(X, \theta)$, and $\phi_i \in \psh(X, \theta)$, $i \in \nat$.
		Then for $T_i := \theta + \ddc \phi_i$, and $T := \theta + \ddc \phi$, the sequence of measures $T_i^n$ converges weakly towards $T^n$ as long as one of the following two conditions are satisfied
		\begin{itemize}
			\item $\phi_i$ increase towards $\phi$ almost everywhere, as $i \to \infty$.
			\item $\phi$ has minimal singularities and $\phi_i$ decrease towards $\phi$, as $i \to \infty$.
		\end{itemize}
	\end{thm}
	\begin{rem}\label{rem_conv_out_pp}
		Both of the above conditions are equivalent to the corresponding statements in which convergence is required only outside a pluripolar set, see \cite[Proposition 8.4]{GuedjZeriahBook}.
	\end{rem}
	\par	
	In this article, superadditive families of functions appear more naturally than increasing ones.
	We state below a weaker version of Theorem \ref{thm_cont_begz} that is tailored to this context.
	\begin{prop}\label{prop_sm_weak_conv}
		Let $\phi_i \in \psh(X,\theta)$, $i \in \nat$, be a sequence which is uniformly bounded from above. 
		Assume that $i \phi_i$ is superadditive, i.e., $i \phi_i + j \phi_j \leq (i+j) \phi_{i+j}$ for all $i,j \in \nat$.
		By Fekete's lemma, the limit $\phi^0 := \lim_{i \to \infty} \phi_i$ exists pointwise. 
		Let $\phi = (\phi^0)^*$. 
		Then $\phi \in \psh(X,\theta)$, and in the notations of Theorem \ref{thm_cont_begz}, the sequence of measures $T_i^n$ converges weakly towards $T^n$.
	\end{prop}
	\begin{proof}
		Remark first that by Fekete's lemma, 
		\begin{equation}\label{eq_fek_sup}
			\phi^0 = \sup_{i \in \nat} \phi_i.
		\end{equation}
		Hence, by the uniform boundedness of $\phi_i$, it is standard that $\phi \in \psh(X, \theta)$, cf. \cite[Proposition I.4.24]{DemCompl}. 
		By \cite[Theorem 2.3]{DDLMonoton}, it suffices to show that $\int_X T^{\,n} \geq \int_X T_i^{\,n}$ and that $\phi_i$ converges to $\phi$ in capacity, see \cite[Definition 4.23]{GuedjZeriahBook} for the definition of the convergence in capacity. 
		The first condition follows from Theorem \ref{thm_wn_monot} and (\ref{eq_fek_sup}).
		\par 
		The proof that $\phi_i$ converges to $\phi$ in capacity is a minor modification of the usual argument showing that a uniformly bounded from above increasing sequence of psh functions converges to the upper semicontinuous regularization of its supremum in capacity, see \cite[Proposition 4.25]{GuedjZeriahBook}. 
		Indeed, the only step of \cite[Proposition 4.25]{GuedjZeriahBook} where monotonicity is used is the application of Dini's lemma, but the latter remains valid for supermultiplicative sequences, cf. \cite[\S A.2]{FinSecRing}.
	\end{proof}
		
	\subsection{Pluripotential theory with prescribed singularities}\label{sect_rel_pp}
	The main goal of this section is to recall some basic facts from the relative pluripotential theory developed by Ross-Witt Nystr{\"o}m in \cite{RossWNEnvelop}, \cite{RossNystAnalTConf} and Darvas-Di Nezza-Lu in \cite{DarvDiNezLuSingType}, \cite{DDLL1}, \cite{DDLMonoton}.
	\par 
	We fix a smooth closed $(1, 1)$-form $\theta$ from a big class $[\theta]$.
	Recall that we already defined the partial order $\preceq$ on the space of $\theta$-psh functions in Section \ref{sect_min_sing}.
	Instead of considering the space of all $\theta$-psh functions, the relative pluripotential theory concerns solely $\theta$-psh functions which have singularities which are not better than those of a fixed $\psi \in \psh(X, \theta)$.
	More precisely, it deals with the following subset of $\psh(X, \theta)$:
	\begin{equation}
		\psh(X, \theta, \psi)
		:=
		\{	\phi \in \psh(X, \theta) : \phi \preceq \psi
		\}.
	\end{equation}
	Remark that this space depends solely on the singularity type $[\psi]$.
	\par 
	Let us now describe the analogue of the envelope construction in the relative setting.
	First, given a non-empty subset $V \subset X$ and a bounded function $f : V \to \mathbb{R}$, we define 
	\begin{equation}\label{eq_p_theta_env}
		P_{\theta, V}(f) := 
		\sup \big\{
			\phi \in \psh(X, \theta) : \phi \leq f \text{ on } V 
		\big\}.
	\end{equation}
	For brevity, we also note $P_{\theta}(f) := P_{\theta, X}(f)$.
	Then it is standard that if $V$ is non-pluripolar, then $P_{\theta, V}(f)$ is uniformly bounded from above, and hence $P_{\theta, V}(f)^*$ is $\theta$-psh, cf. \cite[Theorem 9.17]{GuedjZeriahBook}.
	In particular, we can extend the definition $P_{\theta}(f)$ for an arbitrary bounded from below function $f : X \to \mathbb{R} \cup \{ +\infty \}$, which is bounded from above on a non-pluripolar subset.
	Indeed, if $V$ is such a subset, then one can easily verify that for $M := \sup P_{\theta, V}(f)$, we have
	\begin{equation}\label{eq_env_min}
		P_{\theta}(f) = P_{\theta}(\min\{ f, M \}).
	\end{equation}
	\par 
	It is also immediate that the envelope behaves well under homotheties: for any $f$ as above and any $t \in ]0, +\infty[$, we have
	\begin{equation}\label{eq_env_homot}
		P_{t \cdot \theta}(t \cdot f) = t \cdot P_{\theta}(f).
	\end{equation}
	This property allows us to extend the definition of envelopes to metrics on $\mathbb{Q}$-line bundles.
	\par 
	Moreover, the envelope behaves well with respect to birational maps, cf. \cite[Proposition 2.9]{BermanBouckBalls}.
	More specifically, if the map $\pi : Y \to X$ is birational, then for any $f$ as in (\ref{eq_env_min}), we have
	\begin{equation}\label{eq_bir_inv_env}
		P_{\pi^* \theta}(\pi^* f) = \pi^* P_{\theta}(f).
	\end{equation}
	\par 
	Note that we also immediately have $P_{\theta}(f)^* \leq f^*$, cf. \cite[Theorem 9.17]{GuedjZeriahBook}, and so, if $f$ is upper semicontinuous, then $P_{\theta}(f)$ is also upper semicontinuous, and $P_{\theta}(f) \in \psh(X, \theta)$.
	\par 
	Now, for any fixed $\psi \in \psh(X, \theta)$, we define the envelope
	\begin{equation}
		P_{\theta}[\psi](f) := 
		\lim_{C \to + \infty}
		P_{\theta}(\min\{ \psi + C, f \}).
	\end{equation}
	If $f$ is bounded from above on some non-pluripolar subset, then by a similar argument as in (\ref{eq_env_min}), $P_{\theta}(\min\{ \psi + C, f \})^*$ is an increasing sequence of $\theta$-psh functions, which is uniformly bounded from above.
	Hence, we have $P_{\theta}[\psi](f)^* \in \psh(X, \theta)$.
	Moreover, as for $M := - \sup_{x \in X} \psi(x) + \inf_{x \in X} f(x)$, we have $\psi + M \leq f$, we obtain $\psi + M \leq P_{\theta}[\psi](f)^*$.
	In particular, for the singularity types, we have $[\psi] \preceq [P_{\theta}[\psi](f)^*]$. 
	\par 
	Directly from (\ref{eq_env_min}), (\ref{eq_bir_inv_env}) and \cite[Remark 4.6]{RossNystAnalTConf}, \cite{RossWNEnvelop}, we obtain the following result.
	\begin{prop}\label{prop_model_pot}
		If $f$ is bounded from above on some non-pluripolar subset and from below everywhere, then $[\psi] = [P_{\theta}[\psi](f)^*]$ for every $\psi$ with analytic singularities.
	\end{prop}
	\par 
	Let us now point out the following simple property of the above envelope construction.
	Assume that $\psi_1, \psi_2 \in \psh(X, \theta)$ are such that $\psi_2 \preceq \psi_1$.
	Then for any $f$ as in Proposition \ref{prop_model_pot},
	\begin{equation}\label{eq_comp_law_env}
		P_{\theta}[\psi_2](f)
		=
		P_{\theta}[\psi_2](P_{\theta}[\psi_1](f)).
	\end{equation}
	\par 
	\begin{sloppypar}
	We will now fix $\psi \in \psh(X, \theta)$ with analytic singularities.
	We use the notation $\hat{X}$, $\hat{\theta}$, $g$ for the objects introduced in (\ref{eq_resol_psi_sing}).
	We define the order-preserving bijection between $\psh(X, \theta, \psi)$ and $\psh(\hat{X}, \hat{\theta})$ as
	\begin{equation}\label{eq_isom_hat}
		\psh(X, \theta, \psi) \ni u \mapsto \hat{u} := (u - \psi) \circ \pi  + g \in \psh(\hat{X}, \hat{\theta}).
	\end{equation}
	For the proof of the fact that it is indeed a bijection, see \cite[Theorem 3.1]{GuptaPrakhar}. 
	Note that it implies that $u \in \psh(X, \theta, \psi)$ has the same singularity type as $\psi$ if and only if $\hat{u}$ has minimal singularities.
	\par 
	From the fact that (\ref{eq_isom_hat}) is an order-preserving bijection, it is immediate to see that for any $f$ as in Proposition \ref{prop_model_pot}, we have
	\begin{equation}
		(P_{\theta}[\psi](f) - \psi) \circ \pi  + g = P_{\hat{\theta}}((f - \psi) \circ \pi + g).
	\end{equation}
	By re-arranging the terms, taking the upper semicontinuous regularization and using the fact that $\psi \circ \pi - g$ is upper semicontinuous, see (\ref{eq_resol_psi_sing}), we deduce that under (\ref{eq_isom_hat}), we have
	\begin{equation}\label{eq_env_corresp_one_an}
		P_{\theta}[\psi](f)^* \mapsto P_{\hat{\theta}}((f - \psi) \circ \pi + g)^*.
	\end{equation}
	\end{sloppypar}
	\par 
	The class $[\hat{\theta}]$ is nef by Proposition \ref{prop_b_i_beta}.
	If we moreover assume that $\int (\theta + \ddc \psi)^n > 0$, then by (\ref{eq_siu_rel}), $\int (\hat{\theta} + \ddc g)^n > 0$, and the class $[\hat{\theta}]$ is also big.
	In particular, the relative pluripotential theory in a big class $[\theta]$ associated with a potential $\phi$ with analytic singularities verifying $\int (\theta + \ddc \psi)^n > 0$, reduces to the pluripotential theory associated with a big and nef class $[\hat{\theta}]$.
	This observation is crucial for the extension of Mabuchi geometry in the big setting, as we recall in the next section.
	\par 
	When $\psi$ has neat analytic singularities, the correspondence described above can be refined so that its image consists of potentials associated with a semi-positive form. 
	Indeed, by Proposition \ref{prop_b_i_beta}, the form $\hat{\theta}(g) := \hat{\theta} + \ddc g$ is semi-positive. 
	Moreover, in complete analogy with~ (\ref{eq_isom_hat}), one obtains an order-preserving bijection between $\psh(X, \theta, \psi)$ and $\psh(\hat{X}, \hat{\theta}(g))$, defined as
	\begin{equation}\label{eq_isom_hat2}
		\psh(X, \theta, \psi) \ni u \mapsto \hat{u} := (u - \psi) \circ \pi  \in \psh(\hat{X}, \hat{\theta}(g)).
	\end{equation}
	Although (\ref{eq_isom_hat}) and (\ref{eq_isom_hat2}) introduce an apparent notational conflict, we will always specify which bijection is being used, so no ambiguity will arise.
	 
	 \subsection{Mabuchi geometry on ample and big classes}\label{sect_mab_big}
	The main goal of this section is to recall the definition of the Mabuchi-Darvas distance on the space of Kähler forms within a given Kähler class $[\omega]$. 
	We then extend this definition to the space of bounded $\omega$-psh potentials, where $\omega$ is a Kähler form representing $[\omega]$. 
	In this part, we follow Mabuchi \cite{Mabuchi}, Chen \cite{ChenGeodMab}, and Darvas \cite{DarvasFinEnerg}. 
	Building on these results, and following Di Nezza-Lu \cite{DiNezzaLuBigNef}, we explain how to extend this construction to big and nef classes $[\beta]$. 
	Finally, following Gupta \cite{GuptaPrakhar}, we further extend the construction to the general big classes $[\theta]$.
 	\par 
	Now, let us fix a Kähler form $\omega$ on $X$ in the class $[\omega]$. 
	On the space of Kähler potentials $\mathcal{H}_{\omega} := \{ \phi \in \mathscr{C}^{\infty}(X) : \omega + \ddc \phi > 0 \}$, we define a collection of $L^p$-type Finsler metrics, $p \in [1, +\infty[$, as follows.
	If $u \in \mathcal{H}_{\omega}$ and $\xi \in T_u \mathcal{H}_{\omega} \simeq \ccal^{\infty}(X, \real)$, then the $L^p$-length of $\xi$ is given by the following expression
	\begin{equation}\label{eq_finsl_dist_fir}
		\| \xi \|_p^u
		:=
		\sqrt[p]{
		\frac{1}{\int [\omega]^n}
		 \int_X |\xi(x)|^p \cdot (\omega + \ddc u)^n(x)}.
	\end{equation}
	For $p = 2$, this was introduced by Mabuchi \cite{Mabuchi}, and for $p \in [1, +\infty[$ by Darvas \cite{DarvasFinEnerg}.
	\par 
	Extending the analysis of Chen \cite{ChenGeodMab}, Darvas in \cite{DarvasFinEnerg} studied the metric completions $(\mathcal{E}^p_{\omega}, d_p)$, $p \in [1, +\infty[$, of the path length metric structures $(\mathcal{H}_{\omega}, d_p)$ associated with (\ref{eq_finsl_dist_fir}).
	He proved that these completions are geodesic metric spaces and showed that they coincide with the space of finite energy classes previously introduced by Guedj-Zeriahi \cite{GuedjZeriahiMAEnergy}.
	The phrase path length metric structure above means that the distance between $\phi_0 \in \mathcal{H}_{\omega}$ and $\phi_1 \in \mathcal{H}_{\omega}$ is defined as the infimum of the length $l(\gamma) := \int_0^1 \|\gamma'(t)\|^{\gamma(t)}_p dt$, where $\gamma$ is a piecewise smooth path in $\mathcal{H}_{\omega}$ joining $\phi_0$ and $\phi_1$.
	What will be important for us is that the above distance is well-defined on ${\rm{PSH}}(X, \omega) \cap L^{\infty}(X)$, as ${\rm{PSH}}(X, \omega) \cap L^{\infty}(X) \subset \mathcal{E}^p_{\omega}$ for any $p \in [1, +\infty[$, cf. \cite[Exercise 10.2]{GuedjZeriahBook}.
	\par 
	\begin{sloppypar}
	Di Nezza-Lu \cite{DiNezzaLuBigNef} further extended this theory to the setting of big and nef classes $[\beta]$. 
	To explain how this was done, we fix a smooth closed $(1, 1)$-form $\beta$ in $[\beta]$.
	Then for any $\epsilon > 0$, while the smooth closed $(1, 1)$-form $\beta + \epsilon \omega$ is not necessarily positive, it lies inside of the class $[\beta + \epsilon \omega]$, which is Kähler.
	For $f, g \in \mathscr{C}^0(X, \real)$, Di Nezza-Lu defined the distance between the envelopes $P_{\beta}(f), P_{\beta}(g) \in \psh(X, \beta)$ as follows
	\begin{equation}\label{eq_dl_defn}
		d_p(P_{\beta}(f), P_{\beta}(g))
		:=
		\lim_{\epsilon \to 0} d_p(P_{\beta + \epsilon \omega}(f), P_{\beta + \epsilon \omega}(g)).
	\end{equation}
	Note that the right-hand side of (\ref{eq_dl_defn}) is well-defined, as $[\beta + \epsilon \omega]$ is Kähler, and so $P_{\beta + \epsilon \omega}(f), P_{\beta + \epsilon \omega}(g)$ are bounded.
	For arbitrary $u, v \in \psh(X, \beta)$ with minimal singularities, the distance can then be defined through an additional approximation.
	More specifically, we consider decreasing sequences $f_i, g_i \in \mathscr{C}^0(X, \real)$, $i \in \nat$, so that $\lim_{i \to \infty} f_i = u$ and $\lim_{i \to \infty} g_i = v$ (such sequences exist as $u$ and $v$ are upper semicontinuous). 
	Then we define 
	\begin{equation}\label{eq_mab_dl_co}
		d_p(u, v)
		:=
		\lim_{i \to \infty} d_p(P_{\beta}(f_i), P_{\beta}(g_i)).
	\end{equation}
	Di Nezza-Lu proved that this defines a distance function on the space of $\beta$-psh potentials with minimal singularities, independent of all auxiliary choices. 
	\end{sloppypar}
	\par 
	Gupta \cite{GuptaPrakhar} further extended this theory to the setting of big classes $[\theta]$. 
	To do so, he proceeded by approximation of the space $\psh(X, \theta)$ with $\psh(X, \theta, \psi_k)$, for a certain sequence of $\psi_k \in \psh(X, \theta)$ with analytic singularities. 
	Within each $\psh(X, \theta, \psi_k)$, he defined the metrics using the above definition of Di Nezza-Lu.
	\par 
	Let us first explain how the second step is done and then proceed to the description of the first one.
	For this, we fix $\psi \in \psh(X, \theta)$ with analytic singularities and so that $\int (\theta + \ddc \psi)^n > 0$.
	We describe how to endow with the distance the subspace of $\psh(X, \theta, \psi)$ consisting of $\theta$-psh potentials of the singularity type as $[\psi]$.
	This is done through the use of the bijection (\ref{eq_isom_hat}), from which we use the notations below.
	As we explained in Section \ref{sect_rel_pp}, $\theta$-psh potentials of the singularity type as $[\psi]$ are in one-to-one correspondence with $\hat{\theta}$-psh potentials with minimal singularities.
	Moreover, as we explained, the class $[\hat{\theta}]$ is big and nef.
	Hence, the distance can be imported from (\ref{eq_mab_dl_co}).
	Gupta in \cite[Theorem 3.14]{GuptaPrakhar} proved that the resulting distance does not depend on the choice of the resolution $\pi : \hat{X} \to X$.
	\par 
	Now, we describe the approximation of the space $\psh(X, \theta)$ by $\psh(X, \theta, \psi_k)$.
	We fix an arbitrary $\psi \in \psh(X, \theta)$ with $\psi \leq -1$ and analytic singularities, which is strictly $\theta$-psh (such $\psi$ exists by Theorem \ref{thm_dem_appr}). 
	Gupta showed in \cite[\S 4]{GuptaPrakhar} that there exists an increasing sequence $\psi_k \in \psh(X, \theta)$, $k \in \nat$, with analytic singularities, satisfying $\psi_k \geq \psi$, and such that $\psi_k \to V_{\theta}$, as $k \to \infty$, outside a pluripolar set, where $V_{\theta}$ is defined in (\ref{eq_v_theta}).
	We next show that the sequence $\psi_k$ may be chosen to have algebraic singularities, and so that $\psi_k$ is strictly $\theta$-psh.
	To this end, we adapt Gupta's construction with only minor modifications.
	\par 
	Note first that since $\psi \leq -1$, we have $\psi \leq V_\theta - 1$.
	Define $\varphi_j = \frac{1}{j} \psi + \frac{j-1}{j} V_\theta$.
	Then $\varphi_j \leq V_\theta$ and $\varphi_j \nearrow V_\theta$ outside a pluripolar set. 
	Although $\varphi_j$ do not necessarily have analytic singularities, each $\varphi_j$ is strictly $\theta$-psh (as $\psi$ is) and $\varphi_j \geq \psi$. 
	By Theorem \ref{thm_dem_appr} and a version of Dini's theorem, we may choose a strictly $\theta$-psh potential with algebraic singularities $\phi_j$, verifying $V_\theta \geq \phi_j \geq \varphi_j$. 
	The problem now is that the sequence $\phi_j$, $j \in \nat$, is not necessarily monotone.
	\par 
	Consider the envelope $P_\theta[\phi_j]$.
	Since $\varphi_j \le \phi_j$ and $\phi_j \le V_\theta$, we have $\varphi_j \leq P_\theta[\phi_j]^* \leq V_\theta$.
	Because $\varphi_j \nearrow V_\theta$ outside a pluripolar set, it follows that $P_\theta[\phi_j]^* \to V_\theta$ pointwise outside a pluripolar set. 
	Moreover, since $\phi_j$ has algebraic singularities and $[P_\theta[\phi_j]^*] = [\phi_j]$ by Proposition \ref{prop_model_pot}, the envelopes $P_\theta[\phi_j]^*$ also have algebraic singularities.
	Define
	\begin{equation}\label{eq_gupt_approx_const}
		\psi_k := \max\{P_\theta[\phi_1]^*, \dots, P_\theta[\phi_k]^*\}.
	\end{equation}
	Then $\psi_k$ has algebraic singularities by Proposition \ref{prop_max_alg_sing}, and $\psi_k \nearrow V_\theta$ outside a pluripolar set.
	Moreover, as $P_\theta[\phi_j]^* \geq \varphi_j \geq \psi$, we see that $\psi_k \geq \psi$.
	Note that by considering $\frac{k - 1}{k} \psi_k + \frac{1}{k} \psi$ instead of $\psi_k$, without loosing the generality, we may further assume that $\psi_k$ is strictly $\theta$-psh. 
	\par 
	Note that by Theorem \ref{thm_wn_monot}, we have $\int (\theta + \ddc \psi_k)^n \geq \int (\theta + \ddc \psi)^n > 0$, and so each space $\psh(X, \theta, \psi_k)$ falls into the mold of the considerations after (\ref{eq_isom_hat}).
	Then, for any $f, g \in \mathscr{C}^2(X)$, Gupta in \cite[(7)]{GuptaPrakhar} defined the distance between $P_{\theta}(f)$ and $P_{\theta}(g)$ as follows
	\begin{equation}\label{eq_def_mab_psi_k}
		d_p(P_{\theta}(f)^*, P_{\theta}(g)^*)
		:=
		\lim_{k \to \infty}
		d_p(P_{\theta}[\psi_k](f)^*, P_{\theta}[\psi_k](g)^*).
	\end{equation}
	Remark that the distances on the right-hand side of (\ref{eq_def_mab_psi_k}) are well-defined.
	This is because $\psi_k$ has analytic singularities, and so $[P_{\theta}[\psi_k](f)^*] = [P_{\theta}[\psi_k](g)^*] = [\psi_k]$ by Proposition \ref{prop_model_pot}.
	\par 
	Gupta verified in \cite[Theorem 4.7]{GuptaPrakhar} that the distance (\ref{eq_def_mab_psi_k}) does not depend on the choice of $\psi_k$ as long as this sequence verifies the two properties above: $\psi_k \in \psh(X, \theta)$, is an increasing sequence with analytic singularities, and $\psi_k \to V_{\theta}$, as $k \to \infty$, outside a pluripolar set.
	Further, Gupta in \cite[Definition 4.10]{GuptaPrakhar} defined $d_p(u, v)$ for arbitrary $u, v \in \psh(X, \theta)$ with minimal singularities in the same way as in (\ref{eq_mab_dl_co}) and verified that this is well-defined.
	In what follows, we will occasionally write $d^{\theta}_p(u, v)$ in place of $d_p(u, v)$ in order to emphasize the dependence on $\theta$; however, in most instances this will be clear from the context.
	\par 
	Following previous work of Trusiani \cite{TrusianiL1}, done for $p = 1$, Gupta in \cite[Theorem 7.3]{GuptaPrakhar} verified that for any $u, v \in \psh(X, \theta)$ with minimal singularities, and any $\psi \in \psh(X, \theta)$ with analytic singularities, the following contraction property is satisfied
	\begin{equation}\label{eq_contraction_prop}
		d_p(P_\theta[\psi](u)^*, P_\theta[\psi](v)^*) \leq d_p(u, v).
	\end{equation}
	\par 
	We now proceed to establish the following statement using the results discussed above.
	\begin{lem}\label{lem_choice_psi_cl}
		For any $f, g \in \mathscr{C}^2(X)$, $\epsilon > 0$, there is a strictly $\theta$-psh potential $\psi$ with neat algebraic singularities, so that 
		\begin{equation}\label{eq_choice_psi_cl}
			d_p(P_{\theta}(f)^*, P_{\theta}(g)^*) - \epsilon \leq d_p(P_{\theta}[\psi](f)^*, P_{\theta}[\psi](g)^*) \leq d_p(P_{\theta}(f)^*, P_{\theta}(g)^*).
		\end{equation}
		Moreover, one can choose $\psi$ so that 
		\begin{equation}\label{eq_choice_psi_cl2}
			\int (\theta + \ddc \psi)^n \geq {\rm{vol}}(\theta) - \epsilon.
		\end{equation}
	\end{lem}
	\begin{proof}
		We fix a sequence $\psi_k$ of strictly $\theta$-psh potentials described after (\ref{eq_def_mab_psi_k}).
		By (\ref{eq_def_mab_psi_k}), there is $k \in \nat$ such that 
		\begin{equation}\label{eq_choice_psi_cl3}
			d_p(P_{\theta}(f)^*, P_{\theta}(g)^*) - \epsilon \leq d_p(P_{\theta}[\psi_k](f)^*, P_{\theta}[\psi_k](g)^*).
		\end{equation}
		Moreover, upon increasing $k$, one may further assume by Theorems \ref{thm_bouck_vol} and \ref{thm_cont_begz} that 
		\begin{equation}\label{eq_choice_psi_cl4}
			\int (\theta + \ddc \psi_k)^n \geq {\rm{vol}}(\theta) - \epsilon,
		\end{equation}
		and upon replacing $\psi_k$ by $\psi_k - 1$, we may further assume that $\psi_k \leq -1$.
		Note that while both (\ref{eq_choice_psi_cl}) and (\ref{eq_choice_psi_cl2}) are satisfied for $\psi := \psi_k$, the resulting $\psi$ is not with neat algebraic singularities.
		\par 
		To remedy this, we consider a sequence of strictly $\theta$-psh potentials $\phi_l$, $l \in \nat$, with algebraic singularities decreasing towards $\psi_k$ as in Theorem \ref{thm_dem_appr}.
		Then by our assumption $\psi_k \leq -1$, a version of Dini's theorem implies that we have $\phi_l \leq -1/2$, for sufficiently large $l$.
		We fix such $l$ from now on, and to simplify further exposition, denote $\phi := \phi_l$.
		In particular, by the construction (\ref{eq_gupt_approx_const}), we can consider a sequence $\psi_r^0 \in \psh(X, \theta)$, $r \in \nat$, with algebraic singularities satisfying $\psi_r^0 \geq \phi$, and such that $\psi_r^0 \to V_{\theta}$, as $r \to \infty$, outside a pluripolar set.
		Then immediately from the fact that the definition (\ref{eq_def_mab_psi_k}) doesn't depend on the choice of the approximation sequence, we conclude that 
		\begin{equation}
			d_p(P_{\theta}(f)^*, P_{\theta}(g)^*)
			=
			\lim_{r \to \infty}
			d_p(P_{\theta}[\psi_r^0](f)^*, P_{\theta}[\psi_r^0](g)^*).
		\end{equation}
		From this, (\ref{eq_contraction_prop}) and (\ref{eq_choice_psi_cl3}), we conclude that 
		\begin{multline}
			d_p(P_{\theta}(f)^*, P_{\theta}(g)^*)
			-
			\epsilon
			\leq
			d_p(P_{\theta}[\psi_k](f)^*, P_{\theta}[\psi_k](g)^*)
			\\
			\leq
			d_p(P_{\theta}[\phi](f)^*, P_{\theta}[\phi](g)^*) 
			\leq 
			d_p(P_{\theta}[\psi_1^0](f)^*, P_{\theta}[\psi_1^0](g)^*) 
			\leq 
			d_p(P_{\theta}(f)^*, P_{\theta}(g)^*)
		\end{multline}
		In particular, we deduce that (\ref{eq_choice_psi_cl}) is verified for $\psi := \phi$.
		By Theorem \ref{thm_wn_monot} and (\ref{eq_choice_psi_cl4}), we deduce that (\ref{eq_choice_psi_cl2}) is also verified for such a choice of $\psi$, finishing the proof.
	\end{proof}
	\par 
	The extension of the Mabuchi-Darvas distance for arbitrary $u, v \in \psh(X, \theta)$ with minimal singularities is then done as follows.
	As before, we consider decreasing sequences $f_i, g_i \in \mathscr{C}^2(X, \real)$, $i \in \nat$, so that $\lim_{i \to \infty} f_i = u$ and $\lim_{i \to \infty} g_i = v$. 
	Then we define 
	\begin{equation}\label{eq_ext_dp_gennn}
		d_p(u, v)
		:=
		\lim_{i \to \infty} d_p(P_{\theta}(f_i)^*, P_{\theta}(g_i)^*).
	\end{equation}
	Gupta in \cite[Theorems 4.11, 4.13]{GuptaPrakhar} verified that it gives a well-defined distance on $\theta$-psh potentials with minimal singularities.
	In \cite[Theorem 5.6]{GuptaPrakhar}, he also verified that his construction coincides with the one of Di Nezza-Lu if the class $[\theta]$ is big and nef.
	\par 
	Now, let us explain that if $u, v \in \psh(X, \theta)$ are such that for some $C > 0$, we have $|u - v| \leq C$, 
	\begin{equation}\label{eq_dp_sup_nm_bnd}
		d_p(u, v) \leq C.
	\end{equation}
	Indeed, when then class $[\theta]$ is Kähler, and $u, v \in \ccal^{\infty}(X)$ so that both $\theta + \ddc u$ and $\theta + \ddc v$ are Kähler forms, the conclusion follows immediately from (\ref{eq_finsl_dist_fir}).
	Since the general definition reduces to this case, the result holds without any additional hypotheses.
	\par 
	By approximating $\mathscr{C}^0$-functions with $\mathscr{C}^2$-functions, (\ref{eq_dp_sup_nm_bnd}) implies that (\ref{eq_def_mab_psi_k}) remains valid for $f, g \in \mathscr{C}^0(X)$.
	\par 
	In our further analysis, the following lemma will be important.
	\begin{lem}\label{lem_incr}
		Consider $\psi \in \psh(X, \theta)$ with analytic singularities and an increasing uniformly bounded from above sequence $u_k \in \psh(X, \theta)$, $k \in \nat$, so that $\psi \leq u_k$, for any $k \in \nat$, and for a certain $u \in \psh(X, \theta)$ with minimal singularities, $u_k \to u$, as $k \to \infty$, outside a pluripolar set.
		Consider an arbitrary $v \in \psh(X, \theta)$ with minimal singularities, verifying $\psi \leq v$.
		Then 
		\begin{equation}\label{eq_lem_incr}
			\lim_{k \to \infty}
			d_p(P[\psi](u_k)^*, P[\psi](v)^*) 
			=
			d_p(P[\psi](u)^*, P[\psi](v)^*).
		\end{equation}
	\end{lem}
	\begin{proof}
	\begin{sloppypar}
		Note first that since $\psi$ has analytic singularities, by Proposition \ref{prop_model_pot}, we deduce that $[P_{\theta}[\psi](u_k)^*] = [P_{\theta}[\psi](u)^*] = [P_{\theta}[\psi](v)^*] = [\psi]$.
		In particular, the left-hand side of (\ref{eq_lem_incr}) is well-defined.
		By the triangle inequality, to establish (\ref{eq_lem_incr}), it suffices to verify
		\begin{equation}\label{eq_lem_incr222}
			\lim_{k \to \infty}
			d_p(P[\psi](u_k)^*, P[\psi](u)^*) 
			=
			0.
		\end{equation}
		For this, note that since $u_k \to u$, as $k \to \infty$, outside a pluripolar set, we deduce by \cite[Proposition 2.2]{GuedjLuZeriahEnv} that $P[\psi](\psi_k)^* \to P[\psi](u)^*$, as $k \to \infty$, outside a pluripolar set.
		By \cite[Proposition 4.25]{GuedjZeriahBook}, we conclude that $P[\psi](\psi_k)^*$ converges in capacity to $P[\psi](u)^*$, as $k \to \infty$.
		By \cite[Proposition 1.9]{GuedjLuZeriahEnv}, it implies that for $u, v \in \psh(X, \theta)$ and $I_p(u, v) \geq 0$ defined as
		\begin{equation}
			I_p(u, v)
			:=
			\int |u - v|^p \cdot \big( (\theta + \ddc u)^n + (\theta + \ddc v)^n \big),
		\end{equation}
		and extended to $u, v \in \psh(X, \theta, \psi)$ using the bijection (\ref{eq_isom_hat}), we have
		\begin{equation}
			\lim_{k \to \infty}
			I_p(P[\psi](u_k)^*, P[\psi](u)^*) 
			=
			0.
		\end{equation}
		\end{sloppypar}
		\par 
		However, by \cite[Lemma 4.12]{GuptaPrakhar}, there is $C > 0$, so that for any $u, v \in \psh(X, \theta)$ with minimal singularities, we have
		\begin{equation}
			\frac{1}{C}
			I_p(u, v)
			\leq
			d_p(u, v)
			\leq
			C \cdot
			I_p(u, v)
		\end{equation}
		A combination of the above three statements imply (\ref{eq_lem_incr222}).
	\end{proof}
	\par 
	We emphasize that the Mabuchi-Darvas distance between potentials depends only on the associated currents.
	Namely, if $g: X \to \real$ is any smooth function, then for any $u, v \in \psh(X, \theta)$ with minimal singularities, the Mabuchi-Darvas distance between $u$ and $v$ coincides with the distance between the shifted potentials $u - g, v - g$, considered with respect to $\psh(X, \theta(g))$, where $\theta(g) := \theta + \ddc g$, see for instance \cite[Proposition 3.2]{DiNezzaLuBigNef}.
	We summarize this as
	\begin{equation}\label{eq_transl}
		d_p^{\theta}(u, v)
		=
		d_p^{\theta(g)}(u - g, v - g).
	\end{equation}
	\par 
	Note also that with our normalization of the Mabuchi distance, for any $u, v \in \psh(X, \theta)$ with minimal singularities and any $t \in ]0, +\infty[$, the potentials $t \cdot u, t \cdot v \in \psh(X, t \cdot \theta)$ have minimal singularities, and we have
	\begin{equation}\label{eq_homoth}
		d_p^{t \cdot \theta}(t \cdot u, t \cdot v)
		=
		d_p^{\theta}(u, v).
	\end{equation}
	\par 
	Following previous work of Darvas \cite{DarvasFinEnerg}, done in the Kähler setting, Gupta in \cite[Theorem 5.5]{GuptaPrakhar} verified the so-called Pythagorean identity, which states that for any $u, v \in \psh(X, \theta)$ with minimal singularities, for $P_\theta(u, v) := P_\theta(\min \{u, v \})$, we have
	\begin{equation}
		d_p(u, v)^p = d_p(u, P_\theta(u, v))^p + d_p(P_\theta(u, v), v)^p.
	\end{equation}
	\par 
	Let us now proceed to the description of geodesics associated with the above distances.
	For this, we consider the strip $S = \{ z \in \mathbb{C} : 0 < \Re z < 1 \}$, and let $p : X \times S \to X$ be the natural projection map.
	For any $u_0, u_1 \in \psh(X, \theta)$, we say that $v_t \in \psh(X, \theta)$, $t \in ]0, 1[$, is a subgeodesic between $u_0, u_1$ if the map $X \times S \ni (x, z) \mapsto v_{\Re(z)}(x)$ is a $p^* \theta$-psh function so that $\limsup_{t \to 0+} v_t \leq u_0$ and $\limsup_{t \to 1-} v_t \leq u_1$.
	\par 
	It is immediate that if both $u_0$ and $u_1$ have minimal singularities, then there are subgeodesics joining them.
	We define the \textit{weak geodesic} $u_t$, $t \in ]0, 1[$, between $u_0, u_1 \in \psh(X, \theta)$ with minimal singularities as the supremum over all subgeodesics joining them.
	\par 
	It is then immediate that for any $t \in ]0, 1[$, $u_t$ has minimal singularities.
	Moreover, for any $x \in X$, the function $]0, 1[ \ni t \mapsto u_t(x)$ extends continuously to $[0, 1]$ by the values $u_0(x), u_1(x)$.
	In the ample case, this was observed in \cite{BernBrunnMink}, and the same proof works in general, as observed in \cite{DDLL1}.
	Indeed, we assume that $C > 0$ is so that $u_0 - C \leq u_1 \leq u_0 + C$.
	Then both $v^0_t := u_0 - C t$ and $v^1_t := u_1 - C (1-t)$ are subgeodesics.
	Since they verify that $\liminf_{t \to 0+} v^0_t = u_0$ and $\liminf_{t \to 1-} v^1_t = u_1$, the same would hold for the weak geodesic.
	\par 
	Any subgeodesic is convex in $t$-variable.
	Moreover, the above also shows that for any $x \in X$, the function $[0, 1] \ni t \mapsto u_t(x)$ is a $C$-Lipshitz function.
	In particular, the partial derivative at $t = 0$ of the weak geodesic is well-defined. 
	We denote it by $\dot{u}_0$.
	\par 
	The relevance of the weak geodesics to the Mabuchi geometry comes from the following result, which can be found in \cite[Lemma 5.3, Theorems 5.4 and Corollary 8.5]{GuptaPrakhar}.
	\begin{thm}\label{thm_mab_geod}
		For any $u_0, u_1 \in \psh(X, \theta)$ with minimal singularities, the weak geodesic $u_t$, $t \in [0, 1]$, between them is a metric geodesic, i.e., for any $t, s \in [0, 1]$, we have $d_p(u_t, u_s) = |t - s| \cdot d_p(u_0, u_1)$.
		Moreover, if $p \in ]1, +\infty[$, this is the unique metric geodesic between $u_0$ and $u_1$ for $d_p$ .
		Finally, if $u_0 = P_{\theta}(f)$ for some $f \in \mathscr{C}^2(X)$, then 
		\begin{equation}\label{eq_mab_geod_sp}
			d_p(u_0, u_1)^p
			=
			\frac{1}{{\rm{vol}}(\theta)}
			\int |\dot{u}_0|^p \cdot (\theta + \ddc u_0)^n,
		\end{equation}
		where $(\theta + \ddc u_0)^n$ is defined in the sense of non-pluripolar product.
	\end{thm}
	\par 
	Finally, we note that the Mabuchi-Darvas distance is invariant under birational maps $\pi : Y \to X$.
	Note first that $\pi^* [\theta]$ is a big class on $Y$, see Theorem \ref{thm_bouck_vol}.
	We fix $u, v \in \psh(X, \theta)$ with minimal singularities.
	From (\ref{eq_bir_inv_env}), it follows that $\pi^* u, \pi^* v \in \psh(Y, \pi^* \theta)$ have minimal singularities, and the distance $d_p(\pi^* u, \pi^* v)$ is well-defined.
	We then claim that 
	\begin{equation}\label{eq_bir_eq_mab}
		d_p(\pi^* u, \pi^* v)
		=
		d_p(u, v).
	\end{equation}
	First of all, by (\ref{eq_bir_inv_env}) and (\ref{eq_ext_dp_gennn}), it suffices to establish (\ref{eq_bir_eq_mab}) for $u = P(f)$, $v = P(g)$ for some $\mathscr{C}^2$-functions $f, g : X \to \real$.
	For such $u, v$, the result then follows from Theorem \ref{thm_mab_geod}. 
	Indeed, by using again the relation between the envelopes on $X$ and $Y$ as in (\ref{eq_bir_inv_env}), we see that Mabuchi geodesic between $\pi^* u$ and $\pi^* v$ is the pullback of the Mabuchi geodesic between $u$ and $v$.
	Then (\ref{eq_bir_eq_mab}) follows from (\ref{eq_mab_geod_sp}) and the fact that the non-pluripolar product is invariant under pullback.
	\par 
	Now, note that the Mabuchi geometry equips the space of metrics on a big line bundle $L$ -- with plurisubharmonic potentials with minimal singularities -- with a natural metric structure.
	Indeed, fixing a smooth metric $h^L_0$ on $L$ and setting  $\theta = c_1(L, h^L_0)$, any metric $h^L$ on $L$ with psh potentials with minimal singularities can be written as $h^L = h^L_0 \cdot \exp(-2 \phi)$, where $\phi \in \psh(X, \theta)$ has minimal singularities. 
	The distance between two such metrics is then defined as the distance between their potentials. 
	By (\ref{eq_transl}), such a distance does not depend on the choice of $h^L_0$, and by (\ref{eq_homoth}), such a definition extends naturally to $\mathbb{Q}$-line bundles.
	\par 
	Assume a big line bundle $L$ decomposes as $L = F \otimes \mathscr{O}(E)$, where $E$ is an effective divisor and $F$ is a big line bundle. 
	We fix a continuous metric $h^L$ on $L$.
	Let $h^E_{{\rm{sing}}}$ be the canonical singular metric on $\mathscr{O}(E)$.
	We define the singular metric $h^F := h^L / h^E_{{\rm{sing}}}$.
	\begin{lem}\label{lem_cont_sing_m}
		There is a continuous metric $h^F_1$ on $F$ such that $P(h^F) = P(h^F_1)$.
	\end{lem}
	\begin{proof}
		It follows immediately from (\ref{eq_env_min}) and the fact that the minimum of two continuous functions is continuous.
	\end{proof}
	\par 
	Let $h^L_i := h^L_0 \cdot \exp(- 2\phi_i)$, $i = 1, 2$, be two metrics with psh potentials with minimal singularities $\phi_i \in \psh(X, \theta)$.
	Let $\psi \in \psh(X, \theta)$ be a potential with algebraic singularities.
	We borrow the notations for $\pi: \hat{X} \to X$, $E$, $h^E$, from (\ref{eq_resol_psi_sing}).
	We define the $\mathbb{Q}$-metric $h^F_i := \pi^* h^L_i / h^E_{{\rm{sing}}}$ on $F := \pi^* L \otimes \mathscr{O}(-E)$.
	Let us consider the envelopes $P(h^F_i)$ defined as in (\ref{eq_v_theta}) and (\ref{eq_env_homot}).
	It follows immediately from the way we defined our metrics and from (\ref{eq_env_corresp_one_an}) that for any $p \in [1, +\infty[$, the following identity holds
	\begin{equation}\label{eq_metr_pot_dist_corresp}
		d_p(P(h^F_0)_*, P(h^F_1)_*)
		=
		d_p(P_{\theta}[\psi](\phi_0)^*, P_{\theta}[\psi](\phi_1)^*).
	\end{equation}

	\subsection{Norms, metrics and envelopes}\label{sect_fs_oper}
	In geometric quantization, the correspondence between metrics on line bundles and norms on their spaces of holomorphic sections plays a central role.
	The main goal of this section is to recall the basics of this correspondence.
	\par 
	We fix a holomorphic line bundle $L$ over a compact complex manifold $X$.
	We assume that the space $H^0(X, L^{\otimes k})$ is non-empty for some $k \in \nat$.
	For any norm $N_k = \| \cdot \|_k$ on $H^0(X, L^{\otimes k})$, we associate a singular metric $FS(N_k)$ on $L^{\otimes k}$ as follows: for any $x \in X$, $l \in L^{\otimes k}_x$, we let
	\begin{equation}\label{eq_fs_norm}
		|l|_{FS(N_k), x}
		=
		\inf_{\substack{s \in H^0(X, L^{\otimes k}) \\ s(x) = l}}
		\| s \|_k,
	\end{equation}
	with the convention that the infimum equals $+ \infty$ if there is no $s \in H^0(X, L^{\otimes k})$ such that $s(x) = l$.
	Note that immediately from the definitions, for any $k \in \nat$, we have 
	\begin{equation}\label{eq_nk_fs_lw_bnd}
		N_k \geq \ban^{\infty}_k(FS(N_k)^{\frac{1}{k}}).
	\end{equation}
	Moreover, $N_k \mapsto FS(N_k)$ is a monotone procedure in the sense that for any norms $N_{k, 0}$, $N_{k, 1}$ verifying  $N_{k, 0} \geq N_{k, 1}$, we have
	\begin{equation}\label{eq_fs_mono}
		FS(N_{k, 0}) \geq FS(N_{k, 1}).
	\end{equation}
	\par 
	When the line bundle $L^{\otimes k}$ is very ample, the above metric can be alternatively described through the Kodaira embeddings
	\begin{equation}\label{eq_kod}
		{\rm{Kod}}_k : X \hookrightarrow \mathbb{P}(H^0(X, L^{\otimes k})^*).
	\end{equation}
	which embeds $X$ in the space of hyperplanes in $H^0(X, L^{\otimes k})$.
	The evaluation maps provide the isomorphism 
	$
		L^{\otimes (-k)} \to {\rm{Kod}}_k^* \mathscr{O}(-1),
	$
	where $\mathscr{O}(-1)$ is the tautological bundle over $\mathbb{P}(H^0(X, L^{\otimes k})^*)$.
	We endow $H^0(X, L^{\otimes k})^*$ with the dual norm $N_k^*$ and induce from it a metric $h^{FS}(N_k)$ on $\mathscr{O}(-1)$ over $\mathbb{P}(H^0(X, L^{\otimes k})^*)$. 
	The reader will check that the resulting metric $FS(N_k)$ on $L^{\otimes k}$ is the only metric verifying under the dual of the above isomorphism the identity $FS(N_k) = {\rm{Kod}}_k^* ( h^{FS}(N_k)^* )$.
	A version of the above identity also exists in the big setting, and we recall it in Section \ref{sect_subr}.
	\par 
	For the latter use, we introduce the \textit{base locus} of a holomorphic line bundle $L'$ on $X$ as
	\begin{equation}
		Bs(L') := \Big\{ x \in X : \text{there is } s \in H^0(X, L'), \text{ such that } s(x) \neq 0 \Big\}.
	\end{equation}
	The \textit{stable base locus} is defined as
	\begin{equation}
		\mathbb{B}(L') := \cap_{k = 1}^{+\infty} Bs((L')^{\otimes k}).
	\end{equation}
	The notion of the stable base locus extends naturally to $\mathbb{Q}$-line bundles, a fact we use to define the \textit{augmented base locus} as
	\begin{equation}\label{eq_aug_bl}
		\mathbb{B}_+(L') := \mathbb{B}(L' - \epsilon A),
	\end{equation}
	where $A$ is an arbitrary ample line bundle and $\epsilon$ is a sufficiently small rational number.
	From Noether's finiteness result, the above definition doesn't depend on $\epsilon > 0$ as long as it is very small.
	From this, it is immediate to see that it doesn't depend on the choice of the ample line bundle $A$.
	\par 
	Let us now recall that $FS(N_k)$ has a psh potential if $H^0(X, L^{\otimes k}) \neq \{0\}$.
	To see this, we fix an arbitrary smooth metric $h^L$ on $L$, then the function $\phi_k: X \to [-\infty, +\infty[$, defined so that $FS(N_k) = (h^L)^k \cdot \exp(- 2k \cdot \phi_k)$, can be described as
	\begin{equation}\label{eq_phi_k_fs_k}
		\phi_k(x) :=  \sup \Big\{
			\frac{1}{k} \log |s(x)|_{(h^L)^k} : s \in H^0(X, L^{\otimes k}), \| s \|_{N_k} \leq 1
		\Big\}.
	\end{equation}
	Note that by the Poincaré-Lelong formula, the functions under the supremum in (\ref{eq_phi_k_fs_k}) are in $\psh(X, \theta)$, where $\theta := c_1(L, h^L)$.
	Since the norm $N_k$ is bounded from below by a sup-norm of some metric on $L$, we see that the functions under the supremum in (\ref{eq_phi_k_fs_k}) also form a uniformly bounded from above subset of functions, and hence $\phi_k^* \in \psh(X, \theta)$.
	However, it is immediate to see that $\phi_k$ is continuous outside of the base-locus, where it takes $-\infty$ values, and so $\phi_k^* = \phi_k$; see also (\ref{eq_fs_n_k_a_k_rel}) for an alternative proof of the last fact.
	For Hermitian norms $N_k$, the above formula can be further simplified
	\begin{equation}\label{eq_phi_k_fs_k_herm}
		\phi_k(x) := \frac{1}{2k} \log \Big( \sum_{i = 1}^{n_k} |s_i(x)|^2_{(h^L)^k} \Big).
	\end{equation}
	Hence, if $N_k$ is a Hermitian norm, then $\phi_k$ has neat algebraic singularities. 
	Moreover, since general norms $N_k$ can be bounded above and below by Hermitian norms, it follows that $\phi_k$ always has algebraic singularities, see also the proof of Proposition \ref{prop_fs_regul_blw} for a refinement of this result.
	\par
	We can now study the convergence properties of Fubini-Study metrics.
	Recall that in (\ref{eq_v_theta}), for a continuous metric $h^L$ on $L$, we defined an envelope $P(h^L)$.
	It was established by Berman \cite[Theorem 1.1]{BermanEnvProj} that $P(h^L)$ is $\mathscr{C}^{1, 1}$ on $X \setminus \mathbb{B}_+(L)$ if $h^L$ is $\mathscr{C}^2$ on $X$.
	By approximation and the fact that taking the envelope $P$ is a monotone operation, we deduce that $P(h^L)$ is continuous on $X \setminus \mathbb{B}_+(L)$ if $h^L$ is continuous on $X$.
	It is also standard, cf. \cite[Theorem 9.17]{GuedjLuZeriahEnv}, that $P(h^L)_*$ has a psh potential.
	Moreover, by the continuity of $h^L$, by the reasoning as after (\ref{eq_p_theta_env}), we have $P(h^L)_* = P(h^L)$.
	We now establish a version of a result due to Berman \cite{BermanEnvProj}.
	\begin{sloppypar}
	\begin{thm}\label{thm_conv_fs}
		Let $h^L$ be a continuous metric on $L$, then the sequence of singular metrics $FS(\ban^{\infty}_k(h^L))^{\frac{1}{k}}$ on $L$ converges towards $P(h^L)$ uniformly over any compact subset of $X \setminus \mathbb{B}_+(L)$.
	\end{thm}
	\end{sloppypar}
	\begin{rem}\label{rem_non_unif_gq}
		When $L$ is ample and $h^L$ is smooth and positive, the rate of convergence can be analyzed more precisely, following the results of Tian \cite{TianBerg}, Zelditch \cite{ZeldBerg}, and others. 
		However, such uniform behavior should not be expected in the big case.
		This is for the following reason.
		We already explained that $FS(\ban^{\infty}_k(h^L))^{\frac{1}{k}}$ has a potential with algebraic singularities. 
		But $P(h^L)$ has a potential with minimal singularities.
		As singularity class cannot change under the uniform convergence, the uniform convergence of potentials of $FS(\ban^{\infty}_k(h^L))^{\frac{1}{k}}$ doesn't hold if potentials with minimal singularities do not have algebraic singularities.
	\end{rem}
	We now fix an arbitrary smooth volume form $dV_X$ on $X$.
	\begin{prop}\label{prop_bm_volume}
		For any $\rho \in L^{\infty}(X)$, $\rho \geq 0$ such that ${\rm{ess\,supp}}(\rho) = X$, the following holds.
		For any $\epsilon > 0$ and any continuous metric $h^L$ on $L$, there is $k_0 > 0$ such that, for any $k \geq k_0$,
		\begin{equation}
			\hilb_k(h^L, \rho \cdot dV_X)
			\geq
			\ban^{\infty}_k(h^L) \cdot \exp(- \epsilon k).
		\end{equation}
	\end{prop}
	\begin{proof}
		The result follows immediately from \cite[Theorem 7.1]{FinEigToepl}, but we provide below a slightly simpler proof adapted to our setting.
		Using the standard terminology, Proposition \ref{prop_bm_volume} simply asserts that the measure $\mu := \rho \cdot dV_X$ satisfies the Bernstein-Markov property, cf. \cite[\S 0.3]{BerBoucNys}. For the special case $\rho = 1$, this is a classical result, cf. \cite[Lemma 3.2]{BermanBouckBalls}.
		\par To extend this to a general density $\rho$, observe that by our assumption ${\rm{ess\,supp}}(\rho) = X$, the measure $\mu$ has full support $X$. Moreover, $X$ is regular in the sense of \cite[Definition 1.5]{BerBoucNys}, cf. \cite[Proposition 1.5]{BerBoucNys}. 
		Hence, by \cite[Proposition 1.12]{BerBoucNys}, $\mu$ satisfies the Bernstein-Markov property if and only if it is determining in the sense of \cite[Proposition 1.10]{BerBoucNys}.
		\par 
		From the definition, the determining property depends only on the absolute continuity class of the measure. 
		Since ${\rm ess\,supp}(\rho) = X$, the measure $\mu$ belongs to the same absolute continuity class as $dV_X$. 
		As $dV_X$ is Bernstein-Markov, it follows that $\mu$ is Bernstein-Markov as well.
	\end{proof}
	\begin{proof}[Proof of Theorem \ref{thm_conv_fs}]
		From Proposition \ref{prop_bm_volume}, it suffices to establish the result for $\hilb_k(h^L, dV_X)$ instead of $\ban^{\infty}_k(h^L)$.
		When $h^L$ is $\mathscr{C}^2$, a more refined statement has been proven by Berman in \cite[Theorem 1.4]{BermanEnvProj}.
		The general case of continuous metrics follows easily by the approximation.
	\end{proof}
	The following result was crucial in the formulation of Theorem \ref{thm_char}.
	\begin{lem}\label{lem_fs_sm}
		The sequence of Fubini-Study metrics $FS(N_k)$, $k \in \nat^*$, is submultiplicative for any submultiplicative graded norm $N = (N_k)_{k = 0}^{\infty}$.
		If $N$ is bounded, the sequence of metrics $FS(N_k)^{\frac{1}{k}}$ on $L$ converges, as $k \to \infty$, to a singular metric, $FS(N)$, verifying $FS(N) = \inf_{k \in \nat^*} FS(N_k)^{\frac{1}{k}}$, and the metric $FS(N)_*$ has a psh potential with minimal singularities
	\end{lem}
	\begin{rem}
		By submultiplicativity of $FS(N_k)$, the sequence of metrics $FS(N_k)^{\frac{1}{k}}$ doesn't increase over multiplicative subsequences in $\nat$.
	\end{rem}
	\begin{proof}
		The statement follows directly from Fekete’s lemma and (\ref{eq_fs_norm}), and was explained in \cite[Lemma 3.3]{FinNarSim}, except for the final part, which we describe in detail below.
		From the boundedness assumption and the obvious monotonicity in the definition of the Fubini-Study metric, we see that there are bounded metrics $h^L_0$ and $h^L_1$ such that
		\begin{equation}
			FS(\ban^{\infty}_k(h^L_0)) \leq FS(N_k) \leq FS(\ban^{\infty}_k(h^L_1)).
		\end{equation}
		From this and Theorem \ref{thm_conv_fs}, $P(h^L_0) \leq FS(N) \leq P(h^L_1)$.
		Since both $P(h^L_0)_*$ and $P(h^L_1)_*$ have potentials with minimal singularities, we conclude that the same is true for $FS(N)_*$.
	\end{proof}
	\par 
	Now, even though we initially defined the sup-norms $\ban^{\infty}_k(h^L)$ for bounded metrics $h^L$, the definition actually makes sense for any metric with a potential bounded from below by a potential with minimal singularities.
	To see this, we consider such a metric $h^L$, and write $h^L = h^L_0 \cdot \exp(- 2 \psi)$ for some bounded metric $h^L_0$.
	Then for $s \in H^0(X, L^{\otimes k})$, we define
	\begin{equation}\label{eq_sup_potent}
		\| s \|_{\ban^{\infty}_k(h^L)} := \exp( k \cdot \sup_{x \in X} (\phi - \psi)),
	\end{equation}
	where $\phi := \frac{1}{k} \log |s|_{(h^L_0)^k}$.
	As $\phi \in \psh(X, \theta)$ by Poincaré-Lelong, we deduce that the supremum on the right-hand side of (\ref{eq_sup_potent}) is finite if $\psi$ is bounded from below by a potential with minimal singularities.
	Hence, by the above, the right-hand side of (\ref{eq_char}) is well-defined.
	\par 
	From the above description it is also immediate, cf. \cite[Lemma 2.8]{BermanBouckBalls}, that for any metric $h^L$ as in (\ref{eq_sup_potent}), we have the following identity
	\begin{equation}\label{eq_max_prin}
		\ban^{\infty}_k(h^L)
		=
		\ban^{\infty}_k(P(h^L)),
	\end{equation}
	where $P(h^L)$ is the envelope defined as in (\ref{eq_v_theta}), see the discussion after (\ref{eq_ind_metr_eq}) for a brief proof and Proposition \ref{thm_regul_blw} for a further refinement.
	\par 
	From the above, we see that for any $h^L$ with a potential of minimal singularities, the norm $\ban^{\infty}(h^L)$ is bounded.
	Indeed, since $h^L$ has a potential with minimal singularities and for any continuous $h^L_0$, $P(h^L_0)$ has a potential with minimal singularities, we deduce that there is $C > 0$, such that $P(h^L_0) \cdot \exp(-C) \leq h^L \leq P(h^L_0) \cdot \exp(C)$.
	It then follows from (\ref{eq_max_prin}) that 
	\begin{equation}\label{eq_bound_norm_min_sing}
		\exp(-Ck) \cdot \ban^{\infty}_k(h^L_0)
		\leq
		\ban^{\infty}_k(h^L)
		\leq
		\exp(Ck) \cdot \ban^{\infty}_k(h^L_0).
	\end{equation}
	\par 
	Consider now the following situation.
	Assume that a big line bundle $L$ can be written as $L = F \otimes \mathscr{O}(E)$, where $E$ is an effective divisor and $F$ is an arbitrary line bundle. 
	We fix a singular metric $h^L$ on $L$, and introduce the following metric on $F$: $h^F := h^L / h^E_{{\rm{sing}}}$, where $h^E_{{\rm{sing}}}$ is the canonical singular metric on $\mathscr{O}(E)$.
	It is immediate to verify that if the line bundle $F$ is psh, and the potential of the metric $h^L$ is bounded from below by a potential with minimal singularities, then the same holds true for the induced metric on $F$.
	Moreover, the embedding
	\begin{equation}\label{eq_can_emb}
		\iota_k : H^0(X, F^{\otimes k}) \to H^0(X, L^{\otimes k}), \qquad s \mapsto s \cdot s_E^k,
	\end{equation}
	is isometric, when $H^0(X, F^{\otimes k})$ is endowed with the norm $\ban^{\infty}_k(h^F)$ and $H^0(X, L^{\otimes k})$ is endowed with the norm $\ban^{\infty}_k(h^L)$.
	\par 
	Let us establish the following result, generalizing (\ref{eq_max_prin}).
	\begin{sloppypar}
	\begin{prop}\label{prop_pull_back}
		If the potential of the metric $h^L$ is bounded from below by a potential with minimal singularities, then for any $k \in \nat$, $\iota_k^* \ban^{\infty}_k(h^L) = \ban^{\infty}_k(P(h^F))$.
	\end{prop} 
	\end{sloppypar}
	\begin{proof}
		First of all, for any $s \in H^0(X, F^{\otimes k})$, we have
		\begin{equation}\label{eq_ind_metr_eq}
			\| s \|_{\ban^{\infty}_k(h^F)} = \exp( k \cdot \sup_{x \in X} (\phi - \psi)),
		\end{equation}
		where $\phi := \frac{1}{k} \log|s|_{(h^F_0)^k}$.
		On the other hand, we have
		\begin{equation}
			\| s \|_{\ban^{\infty}_k(P(h^F))} = \exp( k \cdot \sup_{x \in X} (\phi - P_{\theta}(\psi))).
		\end{equation}
		Clearly, we have $P_{\theta}(\psi) \leq \psi$, giving us the inequality $\ban^{\infty}_k(P(h^F)) \geq \ban^{\infty}_k(h^F)$.
		However, since $\phi \in \psh(X, \theta)$ by Poincaré-Lelong, we deduce 
		\begin{equation}
			\phi - \sup_{x \in X} (\phi - \psi) \leq P_{\theta}(\psi).
		\end{equation}	 
		By rearranging terms and taking the supremum, we obtain the reverse inequality, which finishes the proof of Proposition \ref{prop_pull_back}.
	\end{proof}

	\subsection{Quasi-distances on the space of norms and respective geodesics}\label{sect_quas_metr}
	In this section, we recall the definition of a family of distances on the space of Hermitian norms on finite-dimensional vector spaces, and explain how these distances extend to quasi-distances on the space of all norms. 
	We then discuss the construction of geodesics for these distances.
	\par 
	To begin, we fix two Hermitian norms $N_i$, $i = 0, 1$, associated with the scalar products $\langle \cdot, \cdot \rangle_i$ on a vector space $V$ of dimension $v$.
	Define the transfer map $A \in \enmr{V}$ between $N_0$ and $N_1$ as the Hermitian operator, verifying $\langle A \cdot, \cdot \rangle_0 = \langle \cdot, \cdot \rangle_1$.
	For any $p \in [1, +\infty[$, we define
	\begin{equation}\label{dist_norm_fins_expl}
		d_p(N_0, N_1)
		=
		\frac{1}{2}
		\sqrt[p]{\frac{{\rm{Tr}}[|\log A|^p]}{v}}.
	\end{equation}
	By \cite[Theorem 3.1]{BouckErik21}, the above functions $d_p$, $p \in [1, +\infty[$, satisfy the triangle inequality.
	Hence, they define distances on the space of all Hermitian norms.
	It is then immediate to verify that the Hermitian norms $N_t$, associated with the scalar products $\langle A^t \cdot, \cdot \rangle_0$ form a metric geodesic connecting $N_0$ with $N_1$ for any $d_p$, $p \in [1, +\infty[$.
	\par 
	While there are many motivations for the above definition of distances, we only mention that one can show that the space of Hermitian norms on $V$, endowed with the distance $d_2$, is isometric to the space $SL(V) / SU(V)$, equipped with the distance induced by the standard $SL(V)$-invariant metric, see \cite[Theorem 1.1]{DarvLuRub}.
	\par
	Let $N_0, N_1, N_2$ be Hermitian norms on $V$, ordered as $N_0 \leq N_1 \leq N_2$. 
	Lidskii inequality from \cite[Theorem 5.1]{DarvLuRub} states that for any $p \in [1, +\infty[$, we have
	\begin{equation}\label{eq_lidski_herm}
		d_p(N_1, N_2)^p \leq d_p(N_0, N_2)^p - d_p(N_0, N_1)^p.
	\end{equation}
	\par 
	Next, we extend the construction of distances to non-Hermitian norms.
	Let $N_i = \| \cdot \|_i$, $i = 1, 2$, be two norms on $V$.
	We define the \textit{logarithmic relative spectrum} of $N_1$ with respect to $N_2$ as a non-increasing sequence $\lambda_j := \lambda_j(N_1, N_2)$, $j = 1, \cdots, v$, defined as
	\begin{equation}\label{eq_log_rel_spec}
		\lambda_j
		:=
		\sup_{\substack{W \subset V \\ \dim W = j}} 
		\inf_{w \in W \setminus \{0\}} \log \frac{\| w \|_2}{\| w \|_1}.
	\end{equation}
	Directly from (\ref{eq_log_rel_spec}), for any norms $N_1, N_2, N_3$, $j = 1, \cdots, v$, we have
	\begin{equation}\label{eq_log_rel_spec02}
		\lambda_j(N_1, N_2) + \lambda_{\dim V}(N_2, N_3) \leq \lambda_j(N_1, N_3) \leq \lambda_j(N_1, N_2) + \lambda_1(N_2, N_3),
	\end{equation}
	and whenever $N_1 \leq N_2 \leq N_3$, we have
	\begin{equation}\label{eq_log_rel_spec2}
		\lambda_j(N_1, N_2) \leq \lambda_j(N_1, N_3).
	\end{equation}
	We then define the following quantity
	\begin{equation}\label{dist_norm_fins_expl_gen}
		d_p(N_1, N_2)
		=
		\sqrt[p]{\frac{\sum_{i = 1}^{v} |\lambda_i(N_1, N_2)|^p}{v}}.
	\end{equation}
	It is then an immediate consequence of the min-max characterization of the eigenvalues that the definition (\ref{dist_norm_fins_expl_gen}) extends that of (\ref{dist_norm_fins_expl}).
	\par 
	While the function from (\ref{dist_norm_fins_expl_gen}) separates points, it is not clear that it satisfies the triangle inequality, which is why we refer to them as quasi-distances and not distances. 
	We shall show below that it satisfies a relaxed triangle inequality.
	\par 
	By the John ellipsoid theorem, cf. \cite[p. 27]{PisierBook}, for any normed vector space $(V, N_V)$, there is a Hermitian norm $N_V^H$ on $V$, verifying 
	\begin{equation}\label{eq_john_ellips}
		N_V^H \leq N_V \leq \sqrt{v} \cdot N_V^H.
	\end{equation}
	Note that it is immediate from (\ref{eq_log_rel_spec02}) and  Minkowski inequality that if the norms $N_1$ and $N_2$ are such that for $C > 0$, we have $N_1 \cdot \exp(-C) \leq N_2 \leq N_1 \cdot \exp(C)$, then for any norm $N_0$ on $V$, $p \in [1, +\infty[$, we have
	\begin{equation}\label{eq_dp_dinf_compp}
		\big|
			d_p(N_0, N_1) - d_p(N_0, N_2)
		\big| 
		\leq
		C.
	\end{equation}
	From this and the fact that the triangle inequality holds for metrics, it follows, cf. \cite[proof of Lemma 2.9]{FinNarSim} for details, that for any norms $N_1, N_2, N_3$ on $V$,
	\begin{equation}\label{eq_gen_triangle}
		d_p(N_1, N_3) \leq d_p(N_1, N_2) + d_p(N_2, N_3) + 6 \log v.
	\end{equation}
	Similarly to (\ref{eq_gen_triangle}), Lidskii inequality (\ref{eq_lidski_herm}) holds in the following weaker form: for any norms $N_0, N_1, N_2$ on $V$, ordered as $N_0 \leq N_1 \leq N_2$, we have
	\begin{equation}\label{eq_lidski_gen}
		d_1(N_1, N_2) \leq d_1(N_0, N_2) - d_1(N_0, N_1) + 6 \log v.
	\end{equation}
	\par 
	Let us now describe a relation between the $d_1$-distance and the volumes of the unit balls.
	We fix a Hermitian norm $H$ on $V$, which allows us to calculate the volumes $\vol(\cdot)$ of measurable subsets in $V$. 
	Note that while such volumes depend on the choice of $V$, their ratio does not.
	\begin{lem}\label{lem_vol_balls_d1}
		For any two Hermitian norms $N_0$, $N_1$ on $V$, verifying $N_0 \leq N_1$, we have
		\begin{equation}\label{eq_vol_balls_d1}
			d_1(N_0, N_1)
			=
			\frac{1}{v} \log
			\Big(
			\frac{\vol(B_0)}{\vol(B_1)}
			\Big),
		\end{equation}
		where $B_i \subset V$, $i = 0, 1$, are the unit balls of the norms $N_i$.
		If the norms $N_0$, $N_1$ are not necessarily Hermitian, then we have
		\begin{equation}
			\Big|
			d_1(N_0, N_1)
			-
			\frac{1}{v} \log
			\Big(
			\frac{\vol(B_0)}{\vol(B_1)}
			\Big)
			\Big|
			\leq
			6 \log v.
		\end{equation}
	\end{lem}
	\begin{proof}
		For Hermitian norms, the result is immediate: it simply says that the volume of an ellipse is proportional to the product of its principal axes. 
		For not necessarily Hermitian norms, the result follows from (\ref{eq_john_ellips}), (\ref{eq_gen_triangle}) and (\ref{eq_vol_balls_d1}).
	\end{proof}
	\par 
	When working with general norms rather than Hermitian ones, it becomes essential to replace the geodesic construction with an alternative path between norms. 
	Complex interpolation appears to be the most suitable for this purpose.
	\par 
	Let us recall the definition of complex interpolation.
	We fix two norms $N_0 = \| \cdot \|_0$, $N_1 = \| \cdot \|_1$ on $V$.
	Consider the strip $S = \{ z \in \mathbb{C} : 0 < \Re z < 1 \}$.
	Define $\mathcal{F}(V)$ to be the set of all functions $f : S \to V$ such that: $f$ extend continuously on $\overline{S}$, holomorphic on $S$, and the boundary functions $t \mapsto f(it)$ and $t \mapsto f(1+it)$ are bounded.
	\par 
	For $t \in ]0, 1[$, the \emph{complex interpolation norm} $N_t = \| \cdot \|_t$ on $V$ is defined by
	\begin{equation}\label{eq_defn_cx_inter}
	\| v \|_t
	= \inf_{\substack{f \in \mathcal{F}(V) \\ f(t) = v}} \Big\{ 
	    \max \big( 
	        \sup_{y \in \mathbb{R}} \| f(iy) \|_0, \;
	        \sup_{y \in \mathbb{R}} \| f(1+iy) \|_1
	    \big)
	\Big\}.
	\end{equation}
	\par 
	To make a connection between the construction of geodesics and complex interpolation, we note that by the interpolation theorem of Stein-Weiss, see \cite[Theorem 5.4.1]{InterpSp}, a complex interpolation between two Hermitian norms is given by the geodesic connecting them.
	\par 
	From the above definition, it is immediate to see that complex interpolation is a monotone operation in the sense that if the complex interpolation between $N_0$ and $N_1$ is bigger than the complex interpolation between $N'_0$ and $N'_1$ for any norms verifying $N_0 \geq N'_0$ and $N_1 \geq N'_1$.

	\subsection{Geometric quantization on ample line bundles}\label{sect_ample_case}
	The main goal of this section is to recall some geometric quantization results in the ample setting.
	\par 
	We fix an ample line bundle $A$ over $X$. Let $dV_X$ be a smooth volume form on $X$. The following result was established by Chen-Sun \cite{ChenSunQuant} for $p = 2$ and by Berndtsson \cite{BerndtProb} for $p \in [1, +\infty[$, both under additional regularity assumptions on the metrics. These regularity assumptions were later relaxed by Darvas-Lu-Rubinstein \cite{DarvLuRub}.
	See also Darvas-Xia \cite[Lemma 2.10, Theorem 2.11]{DarvXiaTest} for further developments.
	\begin{thm}\label{thm_isom_ample}
		For any bounded metrics $h^A_0$, $h^A_1$ on $A$ with psh potentials, and any $p \in [1, +\infty[$,
		\begin{equation}\label{eq_thm_isom_ample}
			d_p(h^A_0, h^A_1)
			=
			\lim_{k \to \infty}
			\frac{d_p(\hilb_k(h^A_0, dV_X), \hilb_k(h^A_1, dV_X))}{k}.
		\end{equation}
	\end{thm}
	\par 
	Let us explain that Theorem \ref{thm_isom_ample} implies the following version of Theorem \ref{thm_isom}. 
	\begin{prop}\label{thm_isom_ample_ban}
		For any continuous metrics $h^A_0$, $h^A_1$ on $A$ with psh potentials, and any $p \in [1, +\infty[$, we have
		\begin{equation}\label{eq_thm_isom_ample1}
			d_p(h^A_0, h^A_1)
			=
			\lim_{k \to \infty}
			\frac{d_p(\ban_k(h^A_0), \ban_k(h^A_1))}{k}.
		\end{equation}
	\end{prop} 
	\begin{proof}
		By Proposition \ref{prop_bm_volume}, together with (\ref{eq_max_prin}) and (\ref{eq_dp_dinf_compp}), we see that
		\begin{equation}
			\lim_{k \to \infty}
			\frac{d_p(\ban_k(h^A_0), \ban_k(h^A_1))}{k}
			=
			\lim_{k \to \infty}
			\frac{d_p(\hilb_k(h^A_0, dV_X), \hilb_k(h^A_1, dV_X))}{k}.
		\end{equation}
		We conclude by applying Theorem \ref{thm_isom_ample}.
	\end{proof}
	\par 
	Since the envelope $P(h^A)$ is continuous for any continuous metric $h^A$ on $A$ by the discussion before Theorem \ref{thm_conv_fs}, we see that Proposition \ref{thm_isom_ample_ban} and (\ref{eq_max_prin}) imply that Theorem \ref{thm_isom} holds for ample line bundles.
	\par 
	We also note that in \cite{DarvLuRub}, in the setting of Theorem \ref{thm_isom_ample} the authors considered metrics with more singular potentials arising from the corresponding finite energy spaces. 
	It is reasonable to expect that analogous results should hold in the big setting as well.
	\par 
	Now, for any bounded metrics $h^A_0$, $h^A_1$ on $A$ with psh potentials, we denote by $h^A_t$, $t \in [0, 1]$, the Mabuchi geodesic between them.
	For any $k \in \nat$, $t \in [0, 1]$, we consider the geodesic $H_{k, t}$ between $\hilb_k(h^A_0, dV_X)$ and $\hilb_k(h^A_1, dV_X)$, as described after (\ref{dist_norm_fins_expl}).
	For regular metrics $h^A_0$ and $h^A_1$, the following result was established by Berndtsson \cite{BerndtProb}, following the previous work of Phong-Sturm \cite{PhongSturm}, establishing a weaker form of convergence.
	These regularity assumptions were later relaxed in \cite{DarvLuRub}.
	\begin{thm}\label{thm_appr_geod_ample}
		For any $t \in [0, 1]$, the sequence of metrics $FS(H_{k, t})^{\frac{1}{k}}$ on $A$ converges to $h^A_t$ as $k \to \infty$, with respect to any $d_p$-distance, for $p \in [1, +\infty[$.  Moreover, if $h^A_0$ and $h^A_1$ are continuous, then the convergence is uniform.
	\end{thm}
	\par 
	Let us explain that Theorem \ref{thm_appr_geod_ample} implies the following version of Theorem \ref{cor_appr_geod}. 
	\begin{prop}\label{thm_appr_geod_ample_ban}
		For any continuous metrics $h^A_0$, $h^A_1$ on $A$ with psh potentials, in the notations of Theorem \ref{cor_appr_geod}, the sequence of metrics $FS(N_{k, t})^{\frac{1}{k}}$ converges uniformly towards $h^A_t$, as $k \to \infty$. 
	\end{prop} 
	\begin{proof}
		By the monotonicity properties of complex interpolation explained after (\ref{eq_defn_cx_inter}), the fact that the metric geodesic $H_{k, t}$ can be interpreted as complex interpolation between $\hilb_k(h^A_0, dV_X)$, $\hilb_k(h^A_1, dV_X)$, and Proposition \ref{prop_bm_volume}, we see that for any $\epsilon > 0$, there is $k_0 \in \nat$, so that for any $k \geq k_0$, $t \in [0, 1]$, we have
		\begin{equation}
			N_{k, t} \cdot \exp(-\epsilon k) \leq H_{k, t} \leq N_{k, t} \cdot \exp(\epsilon k).
		\end{equation}
		The result now follows from Theorem \ref{thm_appr_geod_ample} and (\ref{eq_fs_mono}).
	\end{proof}
	\par 
	Finally, let us consider a graded submultiplicative norm $N = (N_k)_{k = 0}^{\infty}$, $N_k := \| \cdot \|_k$, on the section ring $R(X, A)$.
	We assume that there is a bounded metric $h^A_0$ on $A$ such that $N_k \geq \ban^{\infty}_k(h^A_0)$ for any $k \in \nat$.
	Then by the ample version of Lemma \ref{lem_fs_sm} from \cite[Lemma 3.3]{FinNarSim}, the sequence of metrics $FS(N_k)^{\frac{1}{k}}$ converges pointwise, as $k \to \infty$, to a bounded metric $FS(N)$ on $A$, and the lower semicontinuous regularization $FS(N)_*$ of $FS(N)$ has a psh potential.
	The following statement is the main result of the previous work of the author \cite{FinNarSim}, which obviously refines Theorem \ref{thm_char} for ample line bundles.
	\begin{thm}\label{thm_char_ample}
		For any $p \in [1, + \infty[$, we have
		\begin{equation}\label{eq_char_ample}
			N \sim_p \ban^{\infty}(FS(N)_*).
		\end{equation}
		If, moreover, $FS(N)$ is continuous, then (\ref{eq_char_ample}) can be strengthened as follows: for any $\epsilon > 0$, there is $k_0 \in \nat$ such that for any $k \geq k_0$, we have
		\begin{equation}\label{eq_char_ample121}
			\ban_k^{\infty}(FS(N)) \leq N_k \leq \ban_k^{\infty}(FS(N)) \cdot \exp(\epsilon k).
		\end{equation}
	\end{thm}
	\begin{rem}
		See Theorem \ref{thm_char2} for the generalization of (\ref{eq_char_ample121}) to singular spaces.
	\end{rem}
	
	\section{Approximations by section rings of ample line bundles}\label{sect_3}
	This section is devoted to one of the principal tools of this article: an approximation procedure that reduces the study of section rings of big line bundles to those of ample line bundles.
	More precisely, in Section \ref{sect_subr}, we describe several subring constructions in a section ring of a big line bundle.
	In Section \ref{sect_comp_max_norm}, we describe how to relate distances between norms on a vector space and a subspace.
	Finally, in Section \ref{sect_quant_mab}, we apply both of these techniques to prove Theorem \ref{thm_isom}.
	
	\subsection{Subrings of section rings of big line bundles}\label{sect_subr}
	A key technique in this paper is to reduce statements concerning big line bundles to ample line bundles.
	To achieve this, we construct “sufficiently large" subrings inside of the section ring of a big line bundle, which can be realized as section rings of ample line bundles.
	\par 
	In this section we review two different constructions of such subrings. 
	The first construction is purely algebraic, whereas the second is more analytic.
	Throughout this section we fix a complex projective manifold $X$, and a big line bundle $L$ on $X$.
	\par 
	\begin{sloppypar}
	\textbf{Construction 1}: \textit{the subring generated by sections of a fixed degree}.
	\par 
	Fix $m \in \mathbb{N}$ and consider the subring $[\sym H^0(X, L^{\otimes m})] \subset R(X, L^{\otimes m})$ generated by $H^0(X, L^{\otimes m})$. 
	Let $[\sym^k H^0(X, L^{\otimes m})]$ be the $k$-th graded component of $[\sym H^0(X, L^{\otimes m})]$. 
	\begin{thm}[{Lazarsfeld-Mustață \cite[Theorem D]{LazMus} }]\label{thm_sym_subalgebra}
		For any $\epsilon > 0$, there is $m_0 \in \nat$ such that for any $m \geq m_0$, there is $k_0 \in \nat$ such that for $k \geq k_0$, we have
		\begin{equation}
			\frac{\dim [\sym^k H^0(X, L^{\otimes m})]}{\dim H^0(X, L^{\otimes km})}
			\geq 
			1 - \epsilon.
		\end{equation}
	\end{thm}
	\end{sloppypar}
	\par 
	Below we give a proof of Theorem \ref{thm_sym_subalgebra}, distinct from that in \cite{LazMus}. This is not only for completeness, but also because it relies on the observation that the ring $\oplus_{k = 0}^{\infty} [\sym^k H^0(X, L^{\otimes m})]$ coincides in high degrees with the section ring of an \textit{ample} line bundle over a \textit{singular space}. 
	This will play an important role in our proof of Theorem \ref{thm_char}.
	\par 
	\begin{proof}[Proof of Theorem \ref{thm_sym_subalgebra}.]
	\begin{sloppypar}
	We denote by $Y_m$ the Zariski closure of the image of the Kodaira map $X \dashrightarrow \mathbb{P}(H^0(X, L^{\otimes m})^*)$.
	As $L$ is big, for $m \in \nat$ sufficiently large, $Y_m$ is a complex analytic space of the same dimension as $X$, i.e., $n$, cf. \cite[p. 139]{LazarBookI}.
	We consider the resolution of indeterminacies of the Kodaira map 
	\begin{equation}\label{eq_blow_up_kod}
		\begin{tikzcd}
		\hat{X}_m \arrow[d, "\pi_m"] \arrow[dr, "f_m"] & & \\
X \arrow[r, dashed, "{\rm{Kod}}_m"] & Y_m \arrow[r, hook, "\iota_m"] & \mathbb{P}(H^0(X, L^{\otimes m})^*).
		\end{tikzcd}
	\end{equation}
	The map $f_m : \hat{X}_m \to Y_m$ is a birational modification, cf. \cite[p. 139]{LazarBookI}.
\end{sloppypar}
	\par 
	If we denote by $E_i$, $i \in I$, the exceptional divisors on $\hat{X}_m$, then there are $b_i \in \nat$ such that for the divisor $E_m := \sum b_i \cdot E_i$, the evaluation map provides the isomorphism 
	\begin{equation}\label{eq_iso_a_k_l_k}
		\pi_m^* L^{\otimes (-m)} \otimes \mathscr{O}(E_m) \to f_m^* \mathscr{O}(-1).
	\end{equation}
	\par 
	\begin{sloppypar}
	Over $\hat{X}_m$, we define the following line bundle 
	\begin{equation}
		L_m := \pi_m^* L^{\otimes m} \otimes \mathscr{O}(-E_m).
	\end{equation}
	We also denote by $A_m$ the restriction of the hyperplane line bundle $\mathscr{O}(1)$ on $Y_m$.
	The isomorphism (\ref{eq_iso_a_k_l_k}) induces the isomorphism 
	\begin{equation}\label{eq_lk_ak}
		L_m \simeq f_m^* A_m,
	\end{equation}
	and birational invariance of the volume, cf. \cite[Proposition 2.2.43]{LazarBookI}, gives us
	\begin{equation}\label{eq_vol_x_k_y_k}
		{\rm{vol}}_{\hat{X}_m}(L_m) = {\rm{vol}}_{Y_m}(A_m).
	\end{equation}
	Also, for sufficiently large $k \in \nat$, there is an isomorphism 
	\begin{equation}\label{eq_sym_yk_rel}
		[\sym^k H^0(X, L^{\otimes m})]
		\overset{\sim}{\to}
		H^0(Y_k, A_m^{\otimes k}).
	\end{equation}
	To see this, first note that the restriction morphism from $\mathbb{P}(H^0(X, L^{\otimes m})^*)$ to $Y_m$ induces a map $H^0(\mathbb{P}(H^0(X, L^{\otimes m})^*), \mathscr{O}(k)) \to H^0(Y_m, A_m^{\otimes k})$. 
	Using the standard identification  $H^0(\mathbb{P}(H^0(X, L^{\otimes m})^*), \mathscr{O}(k)) \simeq \sym^k H^0(X, L^{\otimes m})$ and the definition of $Y_m$, this map factors through the evaluation map on $\sym^k H^0(X, L^{\otimes m})$, which gives precisely the map (\ref{eq_sym_yk_rel}). 
	That this map is an isomorphism for $k$ sufficiently large follows from the fact that the restriction morphism is surjective for $k$ sufficiently large, cf. the explanation after (\ref{eq_res_y_mor}).
	\end{sloppypar}
	\par 
	By (\ref{eq_vol_x_k_y_k}) and (\ref{eq_sym_yk_rel}), in order to prove (\ref{thm_sym_subalgebra}), it is enough to show that there is $m_0 \in \nat$ such that for any $m \geq m_0$, we have
	\begin{equation}\label{eq_thm_sym_subalgebra_0}
		\frac{{\rm{vol}}_{\hat{X}_m}(L_m)}{{\rm{vol}}_X(L)}
		\geq
		m^n \cdot (1 - \epsilon).
	\end{equation}	 
	We will now show that this is a consequence of Theorems \ref{thm_bouck_vol}, \ref{thm_cont_begz} and \ref{thm_conv_fs}.
	\par 
	Indeed, let us fix an arbitrary continuous metric $h^L$ on $L$.
	We denote by $h^{A_m}$ the metric on $A_m$ induced by the norm $N_m := \ban_m^{\infty}(h^L)$ on $H^0(X, L^{\otimes m})^*$ and the isomorphism (\ref{eq_iso_a_k_l_k}).
	Note that this metric has a continuous psh potential.
	In particular, by Theorem \ref{thm_bouck_vol}, we have
	\begin{equation}\label{eq_thm_sym_subalgebra_1}
		{\rm{vol}}_{\hat{X}_m}(f_m^* A_m)
		=
		\int_{\hat{X}_m} c_1(f_m^* A_m, f_m^* h^{A_m})^n.
	\end{equation}
	We denote by $h^{E_m}_{{\rm{sing}}}$ the canonical singular metric on $\mathscr{O}(E_m)$.
	Analogously to the identity described after (\ref{eq_kod}), we have
	\begin{equation}\label{eq_fs_n_k_a_k_rel}
		\pi_m^* FS(N_m)
		=
		f_m^* h^{A_m}
		\cdot
		h^{E_m}_{{\rm{sing}}}.
	\end{equation}
	By the definition of the non-pluripolar product, the fact that $\pi_m$ is a birational modification and (\ref{eq_fs_n_k_a_k_rel}), we then have
	\begin{equation}\label{eq_thm_sym_subalgebra_2}
		\int_{\hat{X}_m} c_1(f_m^* A_m, \pi_m^* FS(N_m) / h^{E_m}_{{\rm{sing}}})^n
		=
		\int_{X} c_1(L^{\otimes m}, FS(N_m))^n.
	\end{equation}	
	On the other hand, by a combination of Theorems \ref{thm_bouck_vol}, \ref{thm_cont_begz}, \ref{thm_conv_fs} and Proposition \ref{prop_sm_weak_conv}, we obtain
	\begin{equation}\label{eq_thm_sym_subalgebra_3}
		\lim_{m \to \infty}
		\int_X c_1(L, FS(N_m)^{\frac{1}{m}})^n
		=
		{\rm{vol}}_X(L).
	\end{equation}	 
	We see now that (\ref{eq_thm_sym_subalgebra_0}) follows from (\ref{eq_thm_sym_subalgebra_1}), (\ref{eq_thm_sym_subalgebra_2}) and (\ref{eq_thm_sym_subalgebra_3}).
	\end{proof}
	\par 
	\textbf{Construction 2}: \textit{the subring associated with a psh function with algebraic singularities}.
	\par 
	We now fix a smooth $(1,1)$-form $\theta$ in $c_1(L)$ and a $\theta$-psh function $\psi$ with algebraic singularities.
	We use the same notations for the birational model $\pi: \hat{X} \to X$, the divisors $E$, $E_i$, $b_i \in \mathbb{Q}$, $b_i \geq 0$, as in Proposition \ref{prop_b_i_beta} and (\ref{eq_resol_psi_sing}).
	Over $\hat{X}$, consider the $\mathbb{Q}$-line bundle $A_{\epsilon} := \pi^* L \otimes \mathscr{O}(-E - \epsilon \sum b_i E_i)$, where $\epsilon \in \mathbb{Q} \cap ]0, 1]$.
	\begin{prop}
		For any $\epsilon \in \mathbb{Q} \cap ]0, 1]$, the $\mathbb{Q}$-line bundle $A_{\epsilon}$ is ample.
		Moreover, for any $\epsilon' > 0$, there is $\epsilon_0 > 0$ such that for any $\epsilon \in \mathbb{Q} \cap ]0, \epsilon_0]$, we have
		\begin{equation}
			{\rm{vol}}_{\hat{X}}(A_{\epsilon})
			\geq
			\int_X (\theta + \ddc \psi)^n
			-
			\epsilon'.
		\end{equation}
	\end{prop}
	\begin{proof}
		The fact that $A_{\epsilon}$ is ample follows immediately by Proposition \ref{prop_b_i_beta} and Kodaira's theorem.
		Now, by the birational invariance of the non-pluripolar product, we have
		\begin{equation}
			\int_X (\theta + \ddc \psi)^n
			=
			\int_{\hat{X}} (\pi^* \theta + \ddc \psi \circ \pi )^n
		\end{equation}
		Immediately from the definition of non-pluripolar product and (\ref{eq_siu_rel}), we have
		\begin{equation}
			\int_{\hat{X}} (\pi^* \theta + \ddc \psi \circ \pi )^n
			=
			\int_{\hat{X}} (\hat{\theta} + \ddc g )^n
		\end{equation}
		Note that as $g$ is bounded, the current $\hat{\theta} + \ddc g$ has minimal singularities.
		It has also been explained after (\ref{eq_siu_rel}) that it is positive.
		Hence, by Theorem \ref{thm_bouck_vol}, we have
		\begin{equation}
			\int_{\hat{X}} (\hat{\theta} + \ddc g )^n
			=
			{\rm{vol}}_{\hat{X}}(\pi^* L \otimes \mathscr{O}(-E)).
		\end{equation}
		The statement now follows from the continuity of the volume functional, cf. \cite[Theorem 2.2.44]{LazarBookI}, which implies that for any $\epsilon' > 0$, there is $\epsilon_0 > 0$ such that for any $\epsilon \in \mathbb{Q} \cap ]0, \epsilon_0]$, we have ${\rm{vol}}_{\hat{X}}(A_{\epsilon}) \geq {\rm{vol}}_{\hat{X}}(\pi^* L \otimes \mathscr{O}(-E)) - \epsilon'$.
	\end{proof}
	Now, for any $\epsilon \in \mathbb{Q} \cap ]0, 1]$, we fix $N \in \nat^*$, so that $A_{\epsilon}^{\otimes N}$ is a line bundle (and not only a $\mathbb{Q}$-line bundle). 
	We shall now explain that for any $k \in \nat$, there is a natural monomorphism
	\begin{equation}\label{eq_incl_a_eps_coh}
		H^0(\hat{X}, A_{\epsilon}^{\otimes kN}) \hookrightarrow H^0(X, L^{\otimes kN}).
	\end{equation}
	To see this, remark first that since $X$ is smooth and $\pi: \hat{X} \to X$ is a birational modification, for any $k \in \nat$, we have the isomorphism $H^0(X, L^{\otimes k}) \to H^0(\hat{X}, \pi^* L^{\otimes k})$.
	The monomorphism (\ref{eq_incl_a_eps_coh}) is then simply a composition of the inverse of this isomorphism and the monomorphism $H^0(\hat{X}, A_{\epsilon}^{\otimes kN}) \hookrightarrow H^0(\hat{X}, \pi^* L^{\otimes kN})$, given by a multiplication by canonical holomorphic sections of $\mathscr{O}(kN E + \epsilon kN \sum b_i E_i)$.
	
	\subsection{Relating distances on a vector space and a subspace}\label{sect_comp_max_norm}
	In several parts of this article, we will need to compare the distances between two norms on a vector space and the distances between their restrictions to a subspace. 
	The following result, which constitutes the main result of this section, provides precisely this comparison.
	\begin{lem}\label{lem_comp_rest}
		On a finitely dimensional vector space $V$, for any $p \in [1, +\infty[$, the $d_p$-distances between two norms $N_0$, $N_1$, are related with the respective $d_p$-distances between their restrictions $N_0|_E$, $N_1|_E$ to a subspace $E \subset V$ as
		\begin{equation}\label{eq_lem_comp_rest111}
			\Big| v \cdot d_p(N_0, N_1)^p -  e \cdot d_p(N_0|_E, N_1|_E)^p \Big|
			\leq
			2 ( v - e) \cdot C^p + 20(1 + \log v) v \cdot p \cdot C^{p - 1}.
		\end{equation}
		where $v := \dim V$, $e := \dim E$, and $C \geq 0$ is such that $N_0 \cdot \exp(-C) \leq N_1 \leq N_0 \cdot \exp(C)$.
		In particular, if for some $\epsilon > 0$, we have $e/v \geq 1 - \epsilon$, then
		\begin{equation}\label{eq_comp_dist_two}
			\Big| d_p(N_0, N_1)^p -  d_p(N_0|_E, N_1|_E)^p \Big|
			\leq
			3 \epsilon \cdot C^p + 20 (1 + \log v) p \cdot C^{p - 1}.
		\end{equation}
	\end{lem}
	\par 
	In order to prove Lemma \ref{lem_comp_rest}, we develop some notions from finite-dimensional functional analysis.
	For two Hermitian norms $N_0$, $N_1$ on $V$, we define the “rooftop" Hermitian norm $N_0 \vee N_1$ as follows: we choose a basis $e_1, \ldots, e_v$ of $V$, $v := \dim V$, which is orthogonal for both $N_0$ and $N_1$, and define $N_0 \vee N_1$ such that $e_1, \ldots, e_v$ is orthogonal for $N_0 \vee N_1$, and normalized such that $\| e_i \|_{N_0 \vee N_1} := \max \{ \| e_i \|_{N_0}, \| e_i \|_{N_1} \}$.
	\par 
	We stress out while we obviously have $N_0 \leq N_0 \vee N_1$ and  $N_1 \leq N_0 \vee N_1$, the space of Hermitian norms on $V$ is not a lattice (unless $\dim V = 1$), i.e., there is in general no minimal Hermitian norm which majorizes both $N_0$ and $N_1$.
	The reason for considering the norm $N_0 \vee N_1$ comes from the following Pythagorean formula
	\begin{equation}\label{eq_wedge}
		d_p(N_0, N_1)^p = d_p(N_0, N_0 \vee N_1)^p + d_p(N_0 \vee N_1, N_1)^p.
	\end{equation}
	\par 
	Unlike the space of Hermitian norms, the space of all norms does form a lattice.
	Indeed, for any two norms $N_0$, $N_1$, the maximum $\max\{N_0, N_1\}$ gives a norm, which is also the least upper bound.
	We claim that for Hermitian $N_0$, $N_1$, we have 
	\begin{equation}\label{eq_max_vee_comp}
		\frac{1}{\sqrt{2}} N_0 \vee N_1 \leq \max\{N_0, N_1\} \leq N_0 \vee N_1.
	\end{equation}
	The second inequality in (\ref{eq_max_vee_comp}) is immediate, since $\max\{N_0, N_1\}$ is the smallest majorant of $N_0$ and $N_1$ among all possible norms, and $N_0 \vee N_1$ is a majorant.
	The first inequality in (\ref{eq_max_vee_comp}) is a reformulation of the following elementary estimate: for any $a_i, b_i > 0$, $x_i \in \comp$, $i = 1, \ldots, v$,
	\begin{equation}
		\frac{1}{2}
		\sum \max\{ a_i, b_i \} \cdot |x_i|^2
		\leq
		\max\Big\{
			\sum a_i \cdot |x_i|^2, \sum b_i \cdot |x_i|^2, 
		\Big\}.
	\end{equation}
	\par 
	Note that by the mean value theorem, for any $d > 0$, $x, y \in \real$, $|x - y| \leq d$ such that for some $C > 0$, $|x|, |y| \leq C$, we have 
	\begin{equation}\label{eq_mv_easy}
		||x|^p - |y|^p| \leq d p \cdot C^{p-1}.
	\end{equation}
	By (\ref{eq_log_rel_spec02}), (\ref{eq_wedge}), (\ref{eq_max_vee_comp}) and (\ref{eq_mv_easy}), we deduce that for arbitrary Hermitian norms $N_0, N_1$ on $V$, the following weak form of Pythagorean identity holds
	\begin{equation}\label{eq_weak_vee}
		\Big| d_p(N_0, N_1)^p - d_p(N_0, \max\{N_0, N_1\})^p - d_p(\max\{N_0, N_1\}, N_1)^p \Big|
		\leq
		10 p \cdot C^{p - 1}.
	\end{equation}
	By (\ref{eq_john_ellips}), (\ref{eq_mv_easy}) and (\ref{eq_weak_vee}), we further see that for not necessarily Hermitian norms $N_0, N_1$,
	\begin{equation}\label{eq_weak_vee_nh}
		\Big| d_p(N_0, N_1)^p - d_p(N_0, \max\{N_0, N_1\})^p - d_p(\max\{N_0, N_1\}, N_1)^p \Big|
		\leq
		10 (1 + \log v) p \cdot C^{p - 1}.
	\end{equation}
	\begin{proof}[Proof of Lemma \ref{lem_comp_rest}]
		Remark first that if we additionally assume that $N_0 \leq N_1$, then we even have
		\begin{equation}\label{eq_lem_comp_rest_or}
			\Big| v \cdot d_p(N_0, N_1)^p -  e \cdot d_p(N_0|_E, N_1|_E)^p \Big|
			\leq
			 ( v - e) \cdot C^p,
		\end{equation}
		Indeed, from the definition of the logarithmic relative spectrum in (\ref{eq_log_rel_spec}), we see that for any $i = 1, \cdots, e$, we have
		\begin{equation}\label{eq_bnd_eig}
			\lambda_i(N_0, N_1) \geq \lambda_i(N_0|_E, N_1|_E) \geq \lambda_{i + v - e}(N_0, N_1).
		\end{equation}
		Since under the assumption $N_0 \leq N_1 \leq N_0 \cdot \exp(C)$, we have $0 \leq \lambda_{i}(N_0, N_1) \leq C$, for any $i = 1, \cdots, v$, (\ref{eq_lem_comp_rest_or}) follows immediately from (\ref{dist_norm_fins_expl_gen}) and (\ref{eq_bnd_eig}).
		\par 
		By (\ref{eq_lem_comp_rest_or}), using the fact that $\max\{N_0, N_1\}|_E = \max\{N_0|_E, N_1|_E\}$, for $i = 0, 1$, we obtain
		\begin{equation}
			\Big| v \cdot d_p(N_i, \max\{N_0, N_1\})^p -  e \cdot d_p(N_i|_E, \max\{N_0|_E, N_1|_E\})^p \Big|
			\leq
			(v - e) \cdot C^p.
		\end{equation}
		We obtain (\ref{eq_lem_comp_rest111}) by summing up the above equation for $i = 0, 1$, using the triangle inequality and (\ref{eq_weak_vee_nh}) twice: for the pair $(N_0, N_1)$ and for $(N_0|_E, N_1|_E)$. 
	\end{proof}	
	
	\subsection{Quantization of distances}\label{sect_quant_mab}
	The main goal of this section is to prove Theorem \ref{thm_isom}.
	Briefly, our strategy will be to reduce Theorem \ref{thm_isom} to the ample case, which is known by Proposition \ref{thm_isom_ample_ban}.
	For this, we shall apply the results from the two previous sections.
	\begin{proof}[Proof of Theorem \ref{thm_isom}]
		By approximation, it is immediate that it suffices to establish the result for smooth $h^L_0$, $h^L_1$.
		Henceforth, we assume that $h^L_0$ and $h^L_1$ are smooth.
		We fix a smooth reference metric $h^L$ on $L$ and set  $\theta = c_1(L, h^L)$.
		We write $h^L_i := h^L \cdot \exp(- 2\phi_i)$, $i = 0, 1$.
		Then in the notations of (\ref{eq_p_theta_env}), we have $P(h^L_i) = h^L \cdot \exp(- 2P_\theta(\phi_i))$, $i = 0, 1$.
		\par 
		For a given $\epsilon > 0$ small enough, by Lemma \ref{lem_choice_psi_cl}, there is a strictly $\theta$-psh potential $\psi$ with neat algebraic singularities, verifying (\ref{eq_choice_psi_cl2}) and 
		\begin{equation}\label{eq_vol_k_ass00}
			\Big| 
			d_p(P_{\theta}(\phi_0)^*, P_{\theta}(\phi_1)^*)^p
			-
			d_p(P_{\theta}[\psi](\phi_0)^*, P_{\theta}[\psi](\phi_1)^*)^p
			\Big|
			\leq
			\epsilon.
		\end{equation}
		\par 
		As in (\ref{eq_siu_rel}), we consider a complex manifold $\hat{X}$, a birational model $\pi : \hat{X} \to X$, a divisor $E$, a function $g: X \to \real$, and a closed semi-positive smooth $(1, 1)$-form $\hat{\theta}(g)$ as in (\ref{eq_siu_rel}) and (\ref{eq_isom_hat2}).
		As explained in (\ref{eq_isom_hat2}), the class $[\hat{\theta}(g)]$ of $\hat{\theta}(g)$ is nef and -- due to our assumption (\ref{eq_choice_psi_cl2}) -- big. 
		By definition of the Mabuchi-Darvas distance on $\psh(X, \theta, \psi)$, (\ref{eq_env_corresp_one_an}) and (\ref{eq_transl}), for $\hat{\phi_i} := (\phi_i - \psi) \circ \pi$, $i = 0, 1$, we have
		\begin{equation}\label{eq_d_p_reduc_pf}
			d_p(P_{\theta}[\psi](\phi_0)^*, P_{\theta}[\psi](\phi_1)^*)
			=
			d_p(P_{\hat{\theta}(g)}(\hat{\phi_0})^*, P_{\hat{\theta}(g)}(\hat{\phi_1})^*).
		\end{equation}
		As described in the proof of Lemma \ref{lem_cont_sing_m}, $P_{\hat{\theta}(g)}(\hat{\phi_i})$ can be interpreted as the envelope of a continuous potential $\min\{ \hat{\phi_i}, M \}$, $i = 0, 1$, for a sufficiently large $M > 0$, and so on the right hand side of (\ref{eq_d_p_reduc_pf}), we have a well-defined Mabuchi-Darvas distance associated with a big and nef class.
		From now on, we shall make this interpretation implicit, to lighten the notations.
		\par 
		Now, instead of working with big and nef classes, we shall reduce directly to an ample class.
		For this, since $\hat{\theta}(g)$ is semipositive, we have $V_{\hat{\theta}(g)} = 0$, where we used the notations analogous to (\ref{eq_v_theta}).
		Now, let $b_i \in \mathbb{Q}$, $b_i > 0$, and a closed $(1, 1)$-form $\beta$ be as in Proposition \ref{prop_b_i_beta}.
		We define the effective $\mathbb{Q}$-divisor $E(b): = \sum b_i \cdot E_i$.
		By the $\partial \dbar$-lemma, we can construct a metric $h^{E(b)}$ on the $\mathbb{Q}$-line bundle $\mathscr{O}_{\hat{X}}(E(b))$, so that $c_1(\mathscr{O}_{\hat{X}}(E(b)), h^{E(b)}) = \beta$.
		We consider $\psi_{E(b)}(x) := \sum \log |s|_{h^{E(b)}}$, where $s$ is the canonical holomorphic section of $\mathscr{O}_{\hat{X}}(E(b))$.
		\par 
		Then, by the Poincaré-Lelong formula, $\psi_{E(b)} \in \psh(\hat{X}, \beta)$.
		By multiplying $h^{E(b)}$ by a constant, we may further arrange that $\psi_E \leq 0$.
		Let $b_0 := \epsilon > 0$ be as in Proposition \ref{prop_b_i_beta}. 
		Then for any $b$ such that $b_0 > b > 0$, we have $\hat{\theta}(g) + \ddc (b \cdot \psi_E) \geq \hat{\theta}(g) - b \beta$.
		Note that the form $\hat{\theta}(g) - b \beta$ is Kähler by Proposition \ref{prop_b_i_beta}.
		Hence, we have $\psi + b \psi_E \in \psh(\hat{X}, \pi^* \theta)$.
		\par 
		Moreover, $\psi_E$ has analytic singularities, and for any sequence $b_k \in \mathbb{Q}$, $b_k > 0$, $k \in \nat^*$, decreasing to $0$, and such that $b_1 < b_0$, the sequence of potentials $b_k \cdot \psi_E$ increases, as $k \to \infty$, towards $V_{\hat{\theta}(g)} = 0$ outside of the pluripolar subset.
		Since the construction of the Mabuchi-Darvas distance on big classes is compatible with the construction on big and nef classes, by the discussion before Lemma \ref{lem_incr}, we deduce that 
		\begin{equation}
			d_p(P_{\hat{\theta}(g)}(\hat{\phi_0})^*, P_{\hat{\theta}(g)}(\hat{\phi_1})^*)
			=
			\lim_{k \to \infty}
			d_p(P_{\hat{\theta}(g)}[b_k \psi_E](\hat{\phi_0})^*, P_{\hat{\theta}(g)}[b_k \psi_E](\hat{\phi_1})^*).
		\end{equation}
		We now fix a sufficiently large $k_0 \in \nat$, so that for any $k \geq k_0$, we have
		\begin{equation}\label{eq_mab_isom_2111}
			\Big| 
			d_p(P_{\hat{\theta}(g)}(\hat{\phi_0})^*, P_{\hat{\theta}(g)}(\hat{\phi_1})^*)^p
			-
			d_p(P_{\hat{\theta}(g)}[b_k \psi_E](\hat{\phi_0})^*, P_{\hat{\theta}(g)}[b_k \psi_E](\hat{\phi_1})^*)^p
			\Big|
			\leq
			\epsilon.
		\end{equation}
		A combination of (\ref{eq_comp_law_env}), (\ref{eq_env_corresp_one_an}), (\ref{eq_vol_k_ass00}) and (\ref{eq_mab_isom_2111}), implies that 
		\begin{equation}\label{eq_mab_isom_2121}
			\Big| 
			d_p(P_{\theta}(\phi_0)^*, P_{\theta}(\phi_1)^*)^p
			-
			d_p(P_{\theta}[\psi + b \psi_E](\phi_0)^*, P_{\theta}[\psi + b \psi_E](\phi_1)^*)^p
			\Big|
			\leq
			2 \epsilon.
		\end{equation}
		\par 
		By Theorem \ref{thm_cont_begz}, as $b_k \psi_E$ converge, as $k \to \infty$, to $V_{\hat{\theta}(g)}$ almost everywhere, upon increasing $k_0 \in \nat$, we may further assume that for any $k \geq k_0$, for the Kähler form $\hat{\theta}(b_k) := \hat{\theta}(g) - b_k \cdot \beta$,
		\begin{equation}\label{eq_vol_k_ass2}
			\int \hat{\theta}(b_k)^n 
			\geq
			\int (\theta + \ddc \psi)^n 
			-
			\epsilon.
		\end{equation}
		We fix $k_0$ as above, and by an abuse of notation, we note $b := b_{k_0}$.		
		By the definition of the Mabuchi-Darvas distance on $\psh(X, \theta, \psi + b \psi_E)$, for $\phi_0(b) := \hat{\phi_0} - b \psi_E$, $\phi_1(b) := \hat{\phi_1} - b \psi_E$,
		\begin{equation}\label{eq_mab_isom_2}
			d_p(P_{\theta}[\psi + b \psi_E](\phi_0)^*, P_{\theta}[\psi + b \psi_E](\phi_1)^*)
			=
			d_p(P_{\hat{\theta}(b)}(\phi_0(b))^*, P_{\hat{\theta}(b)}(\phi_1(b))^*).
		\end{equation}
		This greatly simplifies the situation as the form $\hat{\theta}(b)$ is Kähler by what we described before, and so the right-hand side of (\ref{eq_mab_isom_2}) corresponds to the Mabuchi-Darvas distance in a Kähler class.
		Moreover, by Lemma \ref{lem_cont_sing_m}, the envelopes $P_{\hat{\theta}(b)}(\phi_0(b))^*$, $P_{\hat{\theta}(b)}(\phi_1(b))^*$ can be replaced with the envelopes of continuous metrics.
		\par 
		Our proof of Theorem \ref{thm_isom} will now proceed in two steps.
		First, we show that the convergence (\ref{eq_thm_isom}) holds if one considers $k \in \nat$ running over certain multiplicative subsequences in $\nat$.
		Then we will explain how to upgrade the argument to cover the general limits.
		\par 
		Following Construction 2 from Section \ref{sect_subr}, applied for $\psi := \psi + b \psi_E$, consider the $\mathbb{Q}$-line bundle $A := \pi^* L \otimes \mathscr{O}_X(- E(b) - E)$.
		Let $N \in \nat^*$ be such that $A^{\otimes N}$ becomes a genuine line bundle.
		As described in Section \ref{sect_subr}, for any $k \in \nat$, $N | k$, for $V_{k} := H^0(\hat{X}, A^{\otimes k})$, we then have a natural inclusion $\iota_{k} : V_{k} \to H^0(X, L^{\otimes k})$.
		\par 
		From Theorem \ref{thm_bouck_vol}, (\ref{eq_choice_psi_cl2}) and (\ref{eq_vol_k_ass2}), there is $k_0 \in \nat$ such that for any $k \geq k_0$, $N | k$, we have
		\begin{equation}\label{eq_vol_k_ass3}
			\frac{\dim V_{k}}{H^0(X, L^{\otimes k})} \geq 1 - \epsilon.
		\end{equation}
		In particular, for $C > 0$, so that we have $h^L_1 \cdot \exp(-C) \leq h^L_0 \leq h^L_1 \cdot \exp(C)$, by (\ref{eq_comp_dist_two}), for any $k \geq k_0$, $N | k$, we obtain the following bound
		\begin{multline}\label{eq_vol_k_ass4}
			\Big| d_p(\ban^{\infty}_{k}(h^L_0), \ban^{\infty}_{k}(h^L_1))^p -  d_p(\iota_{k}^* \ban^{\infty}_k(h^L_0), \iota_{k}^* \ban^{\infty}_k(h^L_1))^p \Big|
			\\
			\leq
			30 \epsilon C^p \cdot k^p + 20 (1 + \log n_k) p \cdot C^{p - 1} \cdot k^{p - 1}.
		\end{multline}
		By Proposition \ref{prop_pull_back}, if we denote by $h_{\rm{sing}}$ the canonical singular metric on the $\mathbb{Q}$-line bundle $\otimes \mathscr{O}_X(E(b) + E)$, for any $k \in \nat$, $N | k$, we have 
		\begin{equation}
			\iota_{k}^* \ban^{\infty}_k(h^L_i) := \ban^{\infty}_k(P(h^L_i / h_{\rm{sing}})).
		\end{equation}
		In particular, using the ampleness of $A$, by Lemma \ref{lem_cont_sing_m} and Proposition \ref{thm_isom_ample_ban}, we have
		\begin{equation}\label{eq_vol_k_ass5}
			\lim_{k \to \infty, N | k}
			\frac{d_p(\iota_{k}^* \ban^{\infty}_k(h^L_0), \iota_{k}^* \ban^{\infty}_k(h^L_1))}{k}
			=
			d_p(P(h^L_0 / h_{\rm{sing}})^*, P(h^L_1 / h_{\rm{sing}})^*).
		\end{equation}
		Note, however, that by (\ref{eq_metr_pot_dist_corresp}), we have
		\begin{equation}\label{eq_vol_k_ass500}
			d_p(P(h^L_0 / h_{\rm{sing}})^*, P(h^L_1 / h_{\rm{sing}})^*)
			=
			d_p(P_{\theta}[\psi + b \psi_E](\phi_0)^*, P_{\theta}[\psi + b \psi_E](\phi_1)^*).
		\end{equation}
		All in all, the estimates (\ref{eq_mab_isom_2121}), (\ref{eq_vol_k_ass4}), (\ref{eq_vol_k_ass5}) and (\ref{eq_vol_k_ass500}) imply that for any $k \geq k_0$, $N | k$, we have
		\begin{equation}\label{eq_vol_k_ass6}
			\Big|
			d_p(P(h^L_0), P(h^L_1))^p
			-
			\frac{d_p(\ban^{\infty}_{k}(h^L_0), \ban^{\infty}_{k}(h^L_1))^p}{k^p}
			\Big|
			\leq
			\epsilon (3 + 30 C^p).
		\end{equation}
		From (\ref{eq_vol_k_ass6}), it is immediate to see that (\ref{eq_thm_isom}) holds if instead of the limits over all natural numbers, one considers limits over certain multiplicative subsequences.
		\par 
		We now proceed to show that Theorem \ref{thm_isom} remains valid in full generality.
		For this, we assume now that $r \in \nat$ is sufficiently large so that there is a nonzero $s_r \in H^0(X, L^{\otimes r})$.
		For any $k \in \nat$, consider the monomorphism
		\begin{equation}\label{eq_jk_morph_mult_triv}
			j_k : H^0(X, L^{\otimes k}) \mapsto H^0(X, L^{\otimes k + r}), \qquad s \mapsto s \cdot s_{r}.
		\end{equation}
		Note that the norm $j_k^* \ban^{\infty}_{k + r}(h^L)$ can be described as the sup-norm $\ban^{\infty}_1((h^L)^k \cdot \rho)$, where $\rho(x) := |s_r(x)|_{(h^L)^r}$, $x \in X$.
		Immediately from Proposition \ref{prop_bm_volume}, we deduce that for any $\epsilon > 0$ and any continuous metric $h^L$ on $L$, there is $k_0 > 0$ such that for any $k \geq k_0$, we have
		\begin{equation}
			\ban^{\infty}_k(h^L) \cdot \exp(- \epsilon k) \leq j_k^* \ban^{\infty}_{k + r}(h^L) \leq \ban^{\infty}_k(h^L).
		\end{equation}
		We conclude that for any continuous metrics $h^L_0$, $h^L_1$ on $L$ and any $\epsilon > 0$, there is $k_0 > 0$ such that for any $k \geq k_0$, we have
		\begin{equation}\label{eq_jk_comp}
			\Big|
				d_p(j_k^* \ban^{\infty}_{k + r}(h^L_0), j_k^* \ban^{\infty}_{k + r}(h^L_1))
				-
				d_p(\ban^{\infty}_k(h^L_0), \ban^{\infty}_k(h^L_1))
			\Big|
			\leq
			\epsilon k.
		\end{equation}
		Note, however, that by the fact the $\limsup$ in the definition of the volume of a line bundle is actually a limit, we see that for any $\epsilon > 0$, there is $k_0 \in \nat$ such that for any $k \geq k_0$, we have
		\begin{equation}\label{eq_dim_jk000}
			\frac{\dim \Im j_k}{\dim H^0(X, L^{\otimes (k + r)})} \geq 1 - \epsilon.
		\end{equation}
		From this, (\ref{eq_comp_dist_two}), (\ref{eq_jk_comp}) and (\ref{eq_dim_jk000}), we deduce that for any $\epsilon > 0$, there is $k_0 \in \nat$ such that for any $k \geq k_0$, we have
		\begin{equation}\label{eq_jk_comp2}
			\Big|
				d_p(\ban^{\infty}_{k + r}(h^L_0), \ban^{\infty}_{k + r}(h^L_1))
				-
				d_p(\ban^{\infty}_k(h^L_0), \ban^{\infty}_k(h^L_1))
			\Big|
			\leq
			\epsilon k.
		\end{equation}
		\par 
		Now, we fix a sufficiently large $k_0 \in \nat$, so that for any $i = 0, \ldots, N - 1$, there is a non-zero $s_i \in H^0(X, L^{\otimes (k_0 N + i)})$.
		Then the analogue of the estimate (\ref{eq_jk_comp2}) can be made for $r$ giving all the residues modulo $k$.
		By this and (\ref{eq_vol_k_ass6}), we deduce that for any $\epsilon > 0$, there is $k_1 \in \nat$ such that (\ref{eq_vol_k_ass6}) holds for any $k \geq k_1$.
		This obviously concludes the proof of Theorem \ref{thm_isom}.
	\end{proof}
	\par 
	Let us now explain the relation between Theorem \ref{thm_isom} and the calculation of the volumes of unit balls of the sup-norms, highlighted in Remark \ref{rem_thm_isom}.
	In the following lemma, which justifies (\ref{eq_rem_thm_isom}), we borrow the notations from Lemma \ref{lem_vol_balls_d1}.
	\begin{lem}
		For any two Hermitian norms $N_0$, $N_1$ on $V$, we have
		\begin{equation}\label{eq_vol_balls_d_nord}
			d_1(N_0, N_1)
			=
			\frac{1}{v} \log
			\Big(
			\frac{\vol(B_0) \cdot \vol(B_1)}{\vol(B_{0 \vee 1})^2}
			\Big),
		\end{equation}
		where $B_i \subset V$, $i = 0, 1$, are the unit balls of the norms $N_i$, and $B_{0 \vee 1}$ is the unit ball of the norm $N_0 \vee N_1$.
		If the norms $N_0$, $N_1$ are not necessarily Hermitian, then we have
		\begin{equation}\label{eq_vol_balls_d_nord2}
			\Big|
			d_1(N_0, N_1)
			-
			\frac{1}{v} \log
			\Big(
			\frac{\vol(B_0) \cdot \vol(B_1)}{\vol(B_0 \cap B_1)^2}
			\Big)
			\Big|
			\leq
			8 \log v.
		\end{equation}
	\end{lem}
	\begin{proof}
		To prove (\ref{eq_vol_balls_d_nord}), we take a sum of the two equations corresponding to (\ref{eq_vol_balls_d1}), applied for $N_0, N_0 \vee N_1$ and $N_1, N_0 \vee N_1$, and use (\ref{eq_wedge}).
		The proof of (\ref{eq_vol_balls_d_nord2}) reduces to (\ref{eq_vol_balls_d_nord}) through the use of (\ref{eq_john_ellips}) and (\ref{eq_max_vee_comp}); one should only note that $B_0 \cap B_1$ corresponds to the unit ball of the norm $\max\{N_0, N_1 \}$.
		We leave the details to the reader.
	\end{proof}
	To conclude this section, we formulate a version of Theorem \ref{thm_isom_ample} for big line bundles.
	For this, let us fix a smooth volume form $dV_X$ on $X$.
	\begin{cor}\label{thm_isom_big_l2}
		For any continuous metrics $h^L_0$, $h^L_1$ on $L$, and any $p \in [1, +\infty[$,
		\begin{equation}\label{eq_thm_bigl2}
			d_p(P(h^L_0), P(h^L_1))
			=
			\lim_{k \to \infty}
			\frac{d_p(\hilb_k(h^L_0, dV_X), \hilb_k(h^L_1, dV_X))}{k}.
		\end{equation}
	\end{cor}
	\begin{proof}
		Analogous to the proof of Proposition \ref{thm_isom_ample_ban}.
	\end{proof}

	\section{Characterization of submultiplicative norms}\label{sect_4}
	The main goal of this section is to characterize the submultiplicative norms as sup-norms, i.e., to show Theorem \ref{thm_char}.
	More specifically, in Section \ref{sect_sm_sing}, we establish a version of Theorem \ref{thm_char_ample} valid for singular varieties.
	Then in Section \ref{sect_reg_metr}, we study the sup-norms associated with metrics whose potentials have minimal singularities.
	Finally, in Section \ref{sec_sm_norm_big_pf}, we establish Theorem \ref{thm_char}.
	
	\subsection{Submultiplicative norms on section rings of singular spaces}\label{sect_sm_sing}
	The main purpose of this section is to investigate submultiplicative norms on section rings of ample line bundles over singular spaces, thereby extending our previous work \cite{FinNarSim} to the singular setting. 
	This will be essential for the proof of Theorem \ref{thm_char}, where singular varieties naturally arise. 
	\par 
	Our argument will largely parallel the one in \cite{FinNarSim}, with one notable difference: in \cite[\S 3.1]{FinNarSim} we made use of the results from \cite{FinSecRing}, which concern $L^2$-norms rather than sup-norms. 
	In the singular case, however, $L^2$-norms are significantly harder to control due to the absence of the appropriate version of the Ohsawa-Takegoshi extension theorem, and it is not clear whether the results of \cite{FinSecRing} extend fully to singular varieties. 
	The novelty of our approach lies in avoiding the use of \cite{FinSecRing} altogether, and working instead directly with sup-norms.
	This is feasible because, for sup-norms, Bost \cite{BostDwork} and Randriambololona \cite{RandriamCrelle} established a version of the Ohsawa-Takegoshi extension theorem that remains valid in the singular setting; see Theorem~\ref{thm_bost_ext}.
	\par 
	We fix a reduced compact complex space $(Y, \mathscr{O}_Y)$ endowed with an ample line bundle $A$, see \cite[\S II.5]{DemCompl} for the basic introduction to analysis on complex spaces.
	Below, whenever we refer to smooth functions on $Y$ or smooth metrics on $A$, this is always understood as the restriction of a smooth function or metric defined in a local analytic chart.
	The same applies for the positivity notion of a metric on a line bundle, and it can be seen that the resulting definition doesn't depend on the choice of the analytic chart, see \cite[Lemma 4]{NarasimhLevi}.
	\par 
	By Fekete's lemma, cf. Lemma \ref{lem_fs_sm}, for any submultiplicative norm $N$ on $R(Y, A)$, the sequence of metrics $FS(N_k)^{\frac{1}{k}}$ on $A$ converges, as $k \to \infty$, to a (possibly only bounded from above and even null) metric on $A$, which we denote by $FS(N)$. 
	Throughout this section we assume that $FS(N)$ is continuous.
	In particular, it is a nowhere zero metric.
	Immediately from (\ref{eq_fek_sup}) and (\ref{eq_nk_fs_lw_bnd}), we see that this implies that $\ban^{\infty}(FS(N)) \leq N$, i.e., one side of the inequality required in the definition of a bounded norm is satisfied.
	Another side of the inequality follows from the finite generation of the section ring; moreover, the following stronger statement holds.
	\begin{thm}\label{thm_char2}
		Assume that a graded norm $N = (N_k)_{k = 0}^{\infty}$ over the section ring $R(Y, A)$ of an ample line bundle $A$ is submultiplicative and the metric $FS(N)$ on $A$ is continuous.
		Then for any $\epsilon > 0$, there is $k_0 \in \nat^*$ such that for any $k \geq k_0$, we have
		\begin{equation}\label{eq_char1}
			\ban^{\infty}_k(FS(N)) \leq N_k \leq \exp(\epsilon k) \cdot  \ban^{\infty}_k(FS(N)).
		\end{equation}
	\end{thm}
	In the smooth case, Theorem \ref{thm_char2} recovers one of the main results of \cite{FinNarSim}, see Theorem \ref{thm_char_ample}.
	Similarly to \cite{FinNarSim}, projective and injective tensor norms play a central role in our proof of Theorem \ref{thm_char2}.
	We therefore begin by recalling the definitions.
	\par 
	Let $V_1, V_2$ be two finite dimensional vector spaces endowed with norms $N_i = \norm{\cdot}_i$, $i = 1, 2$.
	The \textit{projective tensor norm} $N_1 \otimes_{\pi} N_2 = \norm{\cdot}_{\otimes_{\pi}}$ on $V_1 \otimes V_2$ is defined for $f \in V_1 \otimes V_2$ as 
	\begin{equation}\label{eq_defn_proj_norm}
		\norm{f}_{ \otimes_{\pi} }
		=
		\inf
		\Big\{
			\sum \| x_i \|_1 \cdot  \| y_i \|_2
			:
			\quad
			f = \sum x_i \otimes y_i
		\Big\},
	\end{equation}
	where the infimum is taken over different ways of partitioning $f$ into a sum of decomposable terms.
	The \textit{injective tensor norm} $N_1 \otimes_{\epsilon} N_2 = \norm{\cdot}_{ \otimes_{\epsilon} }$ on $V_1 \otimes V_2$ is defined as 
	\begin{equation}\label{eq_defn_inf_norm}
		\norm{f}_{ \otimes_{\epsilon} }
		=
		\sup
		\Big\{
			\big|
				(\phi \otimes \psi)(f)
			\big|
			:
			\quad
			\phi \in V_1^*, \psi \in V_2^*, \| \phi \|_{1}^* = \| \psi \|_{2}^* = 1
		\Big\}
	\end{equation}
	where $\| \cdot \|_{i}^*$, $i = 1, 2$, are the dual norms associated with $\| \cdot \|_{i}$.
	Lemma below compares injective and projective tensor norms, see \cite[Proposition 6.1]{RyanTensProd}, \cite[Theorem 21]{SzarekComp} for a proof.
	\begin{lem}\label{lem_inj_proj_bnd_nntr}
		The following inequality between the norms on $V_1 \otimes V_2$ holds
		\begin{equation}\label{eq_bnd_proj_inf}
			 N_1 \otimes_{\epsilon} N_2
			 \leq
			 N_1 \otimes_{\pi} N_2
			 \leq
			 N_1 \otimes_{\epsilon} N_2
			 \cdot
			 \min \{ \dim V_1, \dim V_2 \}.
		\end{equation}
	\end{lem}
	\par 
	One of the key technical insights in \cite{FinNarSim} is that, although the injective and projective tensor norms on the full tensor algebra of a fixed finitely dimensional vector space may differ substantially, their restrictions to the symmetric subalgebra are in fact equivalent, up to a subexponential factor. 
	We now recall this statement.
	\par 
	We first recall some constructions.
 	A norm $N_V = \| \cdot \|_V$ on a finite-dimensional vector space $V$ naturally induces the norm $\| \cdot \|_Q$ on any quotient $Q$, $\pi : V \to Q$ of $V$ through
	\begin{equation}\label{eq_defn_quot_norm}
		\| f \|_Q
		:=
		\inf \big \{
		 \| g \|_V
		 ;
		 \quad
		 g \in V, 
		 \pi(g) = f
		\},
		\qquad f \in Q.
	\end{equation}
	At times we denote the quotient norm by $[ N_V ]$, leaving the underlying quotient space implicit.
	\par 
	Recall that for any $k \in \nat^*$, we have the \textit{symmetrization} map $\sym : V^{\otimes k} \to \sym^k(V)$.
	Consider two norms $\sym^k_{\epsilon}(N_V) := \| \cdot \|^{\sym, \epsilon}_{N_V, k}$ and $\sym^k_{\pi}(N_V) := \| \cdot \|^{\sym, \pi}_{N_V, k}$ on symmetric tensors $\sym^k(V)$, induced by the symmetrization map by the quotient construction, and the norms $N_V \otimes_{\epsilon} \cdots \otimes_{\epsilon} N_V$, $N_V \otimes_{\pi} \cdots \otimes_{\pi} N_V$ on $V^{\otimes k}$.
	Define the norm $\sym^k_{{\rm{ev}}}(N_V) := \| \cdot \|^{{\rm{ev}}}_{N_V, k}$ on $\sym^k(V)$ as
	\begin{equation}\label{eq_sup_norm_polll}
		\| P \|^{{\rm{ev}}}_{N_V, k}
		:=
		\sup_{\substack{v \in V^* \\ \| v \|_V^* \leq 1}} |P(v)|, 
		\qquad
		P \in \sym^k(V).
	\end{equation}
	\begin{thm}[{ \cite[Theorem 3.13]{FinNarSim} }]\label{thm_sym_equiv}
		For any $\epsilon > 0$, there is $k_0 \in \nat^*$ such that for any $k \geq k_0$, the following chain of inequalities is satisfied
		\begin{equation}
			\exp(- \epsilon k) \cdot \sym_{\pi}^k(N_V) \leq \sym_{{\rm{ev}}}^k(N_V) \leq \sym_{\epsilon}^k(N_V) \leq \sym_{\pi}^k(N_V).
		\end{equation}
	\end{thm} 
 	\par 
 	We shall also extensively rely on the following result due to Bost \cite[Theorem A.1]{BostDwork} and Randriambololona \cite[Théorème B]{RandriamCrelle}, which can be seen as an asymptotic version of the sup-version of Ohsawa-Takegoshi extension theorem, cf. \cite{OhsTak1}, \cite{DemBookAnMet}.
 	See also \cite[Theorem 1.10]{FinOTAs} for a more precise result in the smooth setting.
 	\par 
 	\begin{thm}\label{thm_bost_ext}
 		Let $(Y,\mathscr{O}_Y)$ be a reduced compact complex space equipped with an ample line bundle $A$ carrying a smooth positive metric $h^A$, and let $E$ be a vector bundle on $Y$ endowed with a smooth metric $h^E$. 
 		For any reduced compact complex analytic subspace $\iota\colon (Z,\mathscr{O}_Z)\hookrightarrow (Y,\mathscr{O}_Y)$, the following statement holds.
 		For any $\epsilon > 0$, there is $k_0 \in \nat^*$ such that for any $k \geq k_0$, $s \in H^0(Z, \iota^*A^{\otimes k} \otimes \iota^* E)$, there is $\tilde{s} \in H^0(Y, A^{\otimes k} \otimes E)$ such that $\iota^* \tilde{s} = s$ and
 		\begin{equation}
 			\sup_{y \in Y} |\tilde{s}(y)|_{(h^A)^{k} \cdot h^E}
 			\leq
 			\sup_{z \in Z} |s(z)|_{(h^A)^{k} \cdot h^E}
 			\cdot
 			\exp(\epsilon k).
 		\end{equation}
 	\end{thm}
 	The preceding result admits a reformulation in the framework of quotient norms, a perspective that will be important for the arguments below.
	To formulate it, for any $k \in \nat$, consider the restriction map
	\begin{equation}\label{eq_res_y_mor}
	 	\res_Z : H^0(Y, A^{\otimes k} \otimes E) \to H^0(Z, \iota^*A^{\otimes k} \otimes \iota^* E).
	 \end{equation}
	One easily sees that for sufficiently large $k$, the map (\ref{eq_res_y_mor}) is surjective.
	Indeed it follows immediately by Serre's vanishing theorem and the long exact sequence in cohomology 
	\begin{equation}
		\cdots
		\to
		H^0(Y, A^{\otimes k} \otimes E) 
		\to 
		H^0(Z, \iota^*A^{\otimes k} \otimes \iota^* E)
		\to
		H^1(Y, A^{\otimes k} \otimes E \otimes \mathcal{J}_Z) 
		\to 
		\cdots,
	\end{equation}
	associated with the short exact sequence of sheaves 
	\begin{equation}
		0 \to \mathcal{J}_Z \to \mathscr{O}_Y \to \iota_* \mathscr{O}_Z \to 0,
	\end{equation}
	where $\mathcal{J}_Z$ denotes the sheaf of holomorphic functions on $Y$ vanishing along $Z$.
	Hence, a norm on $H^0(Y, A^{\otimes k} \otimes E)$ induces by (\ref{eq_res_y_mor}) a norm on $H^0(Z, \iota^*A^{\otimes k} \otimes \iota^* E)$.
	We denote by $[\ban^{\infty}_1(Y, (h^A)^k \cdot h^E)]$ the quotient norm on $H^0(Z, A^{\otimes k} \otimes E)$ associated with the norm $\ban^{\infty}_1(Y, (h^A)^k \cdot h^E)$ on $H^0(Y, A^{\otimes k} \otimes E)$.
	Theorem \ref{thm_bost_ext} then reformulates as follows: for any $\epsilon > 0$, there is $k_0 \in \nat^*$ such that for any $k \geq k_0$, we have
 	\begin{equation}\label{eq_bost_ref_quot}
 		[\ban^{\infty}_1(Y, (h^A)^k \cdot h^E)]
 		\leq
 		\ban^{\infty}_1(Z, (h^A)^k \cdot h^E)
 		\cdot
 		\exp(\epsilon k).
 	\end{equation}
 	Note that the following inverse inequality is immediate
 	\begin{equation}\label{eq_quot_trivial}
 		\ban^{\infty}_1(Z, (h^A)^k \cdot h^E)
 		\leq
 		[\ban^{\infty}_1(Y, (h^A)^k \cdot h^E)],
 	\end{equation}
 	as it simply says that the sup-norm of the restriction of a section is no bigger than the sup-norm of the section.
 	\par 
 	As a first application of the result above, we extend Theorem \ref{thm_bost_ext} from the case of the symmetric algebra to the more general setting of the section ring associated with an arbitrary ample line bundle.
 	For any $r \in \nat^*$, $k; k_1, \ldots, k_r \in \nat$, $k_1 + \cdots + k_r = k$, we define the multiplication map
	\begin{equation}\label{eq_mult_map}
		{\rm{Mult}}_{k_1, \cdots, k_r} : H^0(Y, A^{k_1}) \otimes \cdots \otimes H^0(Y, A^{k_r}) 
		\to
		H^0(Y, A^{\otimes k}),
	\end{equation}	
	as $f_1 \otimes \cdots \otimes f_r \mapsto f_1 \cdots f_r$.
	As $A$ is assumed to be ample, it is standard that there is $k_0 \in \nat^*$ such that for any $k_1, \cdots, k_r \geq k_0$, the map ${\rm{Mult}}_{k_1, \cdots, k_r}$ is surjective, cf. \cite[Example 1.2.22]{LazarBookI}.
	\par 
	Assume now that $k, l \in \mathbb{N}^*$ are sufficiently large so that the multiplication map ${\rm Mult}_{k,l}$ is surjective. 
	A straightforward verification shows that, using the notations of (\ref{eq_defn_proj_norm}) and (\ref{eq_defn_quot_norm}), the submultiplicativity condition can be reformulated in terms of inequalities between the norms on $H^0(Y, A^{\otimes (k+l)})$ as follows:
	\begin{equation}\label{eq_reform_subm_cond}
		N_{k + l} \leq [N_k \otimes_{\pi} N_l],
	\end{equation}
	where the bracket $[\cdot]$ above is interpreted with respect to ${\rm Mult}_{k,l}$.
	\par 
 	For simplicity, we further assume that $A$ is very ample and all the multiplication maps are surjective.
 	By the finite generation of the ring $R(Y, A)$, the above can always be satisfied upon considering $R(Y, A^N)$ a sufficiently large $N \in \nat^*$.
 	\par 
 	Let $N_1$ be a norm on $H^0(Y, A)$.
 	By the surjectivity of the multiplication maps, we endow $H^0(Y, A^{\otimes k})$ with the norms $N_k^{\pi} = [N_1 \otimes_{\pi} \cdots \otimes_{\pi} N_1]$ and $N_k^{\epsilon} = [N_1 \otimes_{\epsilon} \cdots \otimes_{\epsilon} N_1]$, where the tensor powers are repeated $k$ times.
 	According to (\ref{eq_reform_subm_cond}), the norm $N^{\pi} = (N_k^{\pi})_{k = 0}^{\infty}$ is the biggest submultiplicative norm on $R(Y, A)$, extending $N_1$.
 	We now prove the following result, which may be regarded as an instance of Theorem \ref{thm_char2}, specialized to $N^{\pi}$.
 	\begin{prop}\label{prop_induc}
 		For any $\epsilon > 0$, there is $k_0 \in \nat^*$ such that for any $k \geq k_0$, the following chain of inequalities is satisfied
		\begin{equation}\label{eq_prop_induc}
			\exp(- \epsilon k) \cdot N_k^{\pi} \leq \ban^{\infty}_k(FS(N_1)) \leq N_k^{\epsilon} \leq N_k^{\pi}.
		\end{equation}
 	\end{prop}
 	\begin{proof}
 		The proof follows that of \cite[Theorem 3.1]{FinNarSim}, which is done in the non-singular setting with essentially one modification: instead of the semiclassical Ohsawa-Takegoshi extension theorem from \cite{FinOTAs}, we apply Theorem \ref{thm_bost_ext}.
 		For the convenience of the reader, we reproduce the argument below with minor simplifications.
 		First of all, the inequality $N^{\epsilon} \leq N^{\pi}$ follows from the respective inequality on the level of tensor norms from Lemma \ref{lem_inj_proj_bnd_nntr}.
 		The inequality $\ban^{\infty}(FS(N_1)) \leq N^{\epsilon}$ is also rather immediate.
 		Indeed, by (\ref{eq_nk_fs_lw_bnd}), we have $\ban^{\infty}_1(FS(N_1)) \leq N_1$.
 		Hence, by the order-preserving nature of the tensor norms, we deduce 
 		\begin{equation}\label{eq_sym_ev_vs_ban00}
 			[\ban^{\infty}_1(FS(N_1)) \otimes_{\epsilon} \cdots \otimes_{\epsilon} \ban^{\infty}_1(FS(N_1))] \leq N_k^{\epsilon}.
 		\end{equation}
 		Note that in the notations of Theorem \ref{thm_sym_equiv}, over $\sym^k H^0(Y, A)$, we have 
 		\begin{equation}\label{eq_sym_ev_vs_ban0}
 			\ban^{\infty}_1(FS(N_1)) \otimes_{\epsilon} \cdots \otimes_{\epsilon} \ban^{\infty}_1(FS(N_1)) = \sym_{\epsilon}^k(\ban^{\infty}_1(FS(N_1))).
 		\end{equation}
 		Immediately from the definitions, we see that for any $k \in \nat$, we have 
 		\begin{equation}\label{eq_sym_ev_vs_ban}
 			\sym_{{\rm{ev}}}^k(\ban^{\infty}_1(FS(N_1))) = \ban^{\infty}_k(\mathbb{P}(H^0(Y, A)^*), FS(N_1))),
 		\end{equation}
 		where above by an abuse of notation, we denoted the metric on $\mathscr{O}(1)$ induced by $N_1$ as $FS(N_1)$.
 		\par 
 		Let us consider the Kodaira embedding ${\rm{Kod}}_1: Y \hookrightarrow \mathbb{P}(H^0(Y, A)^*)$.
		We denote by $\res_{{\rm{Kod}}} : R(\mathbb{P}(H^0(Y, A)^*), \mathscr{O}(1)) \to R(Y, A)$ the associated restriction operator, and by $\res_{{\rm{Kod}}, k}$, $k \in \nat^*$, the restriction operators on the associated graded pieces.
		The multiplication operator ${\rm{Mult}}_{1, \cdots, 1}$ from (\ref{eq_mult_map}) factorizes through symmetrization and restriction as
		\begin{equation}\label{eq_kod_map_comm_d}
		\begin{tikzcd}
			H^0(Y, A)^{\otimes k} \arrow[swap, rrdd, "{\rm{Mult}}_{1, \cdots, 1}"] \arrow[r, "\sym"] & {\rm{Sym}}^k(H^0(Y, A)) \arrow[rd, equal] &  \\
			&  & H^0(\mathbb{P}(H^0(Y, A)^*), \mathscr{O}(k)) \arrow[d, "\res_{{\rm{Kod}}, k}"]  \\
 	 		&  & H^0(Y, A^{\otimes k}).
    	\end{tikzcd}
		\end{equation}
		In particular, the quotient norm $[\sym_{{\rm{ev}}}^k(\ban^{\infty}_1(FS(N_1)))]$ can be seen as the quotient norm calculated not with respect to the multiplication operator, but with respect to the restriction operator.
		Note, however, that it is standard that the construction of the Fubini-Study metric $FS(N_1)$ can alternatively be described as the pull-back of the metric on $\mathscr{O}(1)$ associated with the Kodaira embedding and the isomorphism between ${\rm{Kod}}_1^* \mathscr{O}(1)$ and $A$, see (\ref{eq_kod}).
		The inequality 
		\begin{equation}\label{eq_sym_ev_vs_ban1}
			[\ban^{\infty}(\mathbb{P}(H^0(Y, A)^*), FS(N_1)))] \geq \ban^{\infty}(FS(N_1))
		\end{equation}
		is then simply a version of (\ref{eq_quot_trivial}).
 		The inequality $\ban^{\infty}(FS(N_1)) \leq N^{\epsilon}$ then follows immediately from the second inequality from Theorem \ref{thm_sym_equiv}, (\ref{eq_sym_ev_vs_ban00}), (\ref{eq_sym_ev_vs_ban0}), (\ref{eq_sym_ev_vs_ban}) and (\ref{eq_sym_ev_vs_ban1}).
 		\par 
 		Let us now establish the only non-trivial part of Proposition \ref{prop_induc} which is the inequality $\exp(- \epsilon k) \cdot N_k^{\pi} \leq \ban^{\infty}_k(FS(N_1))$. 
 		Immediately from Theorem \ref{thm_sym_equiv}, we deduce that in its notations, there is $k_0 \in \nat^*$ such that for any $k \geq k_0$, we have $\exp(- \epsilon k / 2) \cdot \sym_{\pi}^k(N_1) \leq \sym_{{\rm{ev}}}^k(N_1)$.
		Since by definition, we have $[\sym_{\pi}^k(N_1)] = N_k^{\pi}$, we see that it suffices to show that there is $k_0 \in \nat^*$ such that for any $k \geq k_0$, we have
		\begin{equation}
			\exp(- \epsilon k / 2) \cdot [\sym_{{\rm{ev}}}^k(N_1)] \leq \ban^{\infty}_k(FS(N_1)).
		\end{equation}
		But the above inequality is just a restatement of (\ref{eq_bost_ref_quot}) taking into account the identifications explained after (\ref{eq_kod_map_comm_d}).
 	\end{proof}
 	\begin{proof}[Proof of Theorem \ref{thm_char2}.]
		Let us fix $\epsilon > 0$.
		By our assumption on the continuity of $FS(N)$ and a version of Dini's theorem for submultiplicative functions, cf. \cite[\S A.2]{FinSecRing}, there is $k_0 \in \nat$ such that for any $k \geq k_0$, we have
		\begin{equation}\label{eq_fs_unif}
			FS(N_k)^{\frac{1}{k}}
			\leq
			\exp(\epsilon) \cdot FS(N).
		\end{equation}
		By the submultiplicativity of $N$ and (\ref{eq_reform_subm_cond}), in the notations of Theorem \ref{thm_sym_equiv}, for any $q = k, k+1$, $l \in \nat$, we have $N_{q l} \leq [\sym_{\pi}^l(N_q)]$.
		Directly from Proposition \ref{prop_induc}, we then obtain that for any $q = k, k+1$, there is $l_0 \in \nat$ such that for $l \geq l_0$, we have
		\begin{equation}\label{eq_pf_thmch21}
			N_{q l} \leq \ban^{\infty}_{l}(FS(N_q)) \cdot \exp(\epsilon ql).
		\end{equation}
		\par 
		Let us now establish that (\ref{eq_pf_thmch21}) implies Theorem \ref{thm_char2}.
		For this, we closely follow the idea of \cite[\S 3.1]{FinSecRing}, interpreting the multiplication as a restriction map from the product manifold to the diagonal.
		Let us consider the product manifold $Y \times  Y$ and the diagonal submanifold in it, given by $\{(y, y) : y \in Y \} =: \Delta \hookrightarrow Y \times Y$.
		Clearly, we have a natural isomorphism $\Delta \to y$, $(y, y) \mapsto y$ and for any $k, l \in \nat$, we have an isomorphism $A^k \boxtimes A^l|_{\Delta} \to A^{k + l}$.
		\par 
		For any $i, j \in \nat$, we denote $E_{i, j} := A^{k i} \boxtimes A^{(k + 1) j}$.
		Künneth theorem allows us to interpret the tensor product of cohomologies as the cohomology of the product manifold.
		In particular, for any $k, l, i, j \in \nat$, we have the following commutative diagram 
		\begin{equation}\label{eq_comm_diag}
			\begin{CD}
				H^0(Y \times  Y, (A^k \boxtimes A^{k + 1})^{\otimes l} \otimes E_{i, j})  @> {\rm{Res}}_{\Delta} >> H^0(\Delta, (A^k \boxtimes A^{k + 1})^{\otimes l} \otimes E_{i, j})|_{\Delta}) 
				\\
				@VV {} V @VV {} V
				\\
				H^0(Y, A^{k(l+i)}) \otimes H^0(Y, A^{(k+1)(l + j)})  @> {\rm{Mult}} >> H^0(Y, A^{k(l+i) + (k+1)(l + j)}).
			\end{CD}
		\end{equation}
		By applying \eqref{eq_bost_ref_quot} to the line bundle $A^k \boxtimes A^{k+1}$ equipped with the metric $FS(N_k) \boxtimes FS(N_{k+1})$ and to the vector bundle $E_{i,j}$ equipped with the metric $FS(N_k)^i \boxtimes FS(N_{k+1})^j$, and combining this with the commutative diagram \eqref{eq_comm_diag}, we deduce that there is $l_0 \in \nat$ such that for any $i = 0, \ldots, k$, $j = 0, \ldots, k - 1$, for $l \geq l_0$, we have
		\begin{multline}\label{eq_comm111}
			[\ban^{\infty}_1(Y \times Y, FS(N_k)^{l+i} \boxtimes FS(N_{k+1})^{l + j})] 
			\\
			\leq
			\ban^{\infty}_1(FS(N_k)^{l+i} \cdot FS(N_{k + 1})^{l+j})
			\cdot
			\exp(\epsilon (k(l+i) + (k+1)(l + j))).
		\end{multline}
		Note, however, that as explained in \cite[Theorem 3.6]{FinSecRing}, we have
		\begin{equation}\label{eq_comm112}
			\ban^{\infty}_1(Y \times Y, FS(N_k)^{l+i} \boxtimes FS(N_{k + 1})^{l + j})
			=
			\ban^{\infty}_{l+i}(FS(N_k)) \otimes_{\epsilon} \ban^{\infty}_{l + j}(FS(N_{k+1})).
		\end{equation}
		\par 
		By submultiplicativity and (\ref{eq_reform_subm_cond}), in the notations of Theorem \ref{thm_sym_equiv}, for any $q = k, k+1$, $l \in \nat$, we have $N_{k(l+i) + (k+1)(l + j)} \leq [\sym_{\pi}^{l+i}(N_k) \otimes_{\pi} \sym_{\pi}^{l + j}(N_{k + 1})]$.
		Now, we apply (\ref{eq_pf_thmch21}) to obtain
		\begin{multline}\label{eq_comm113}
			N_{k(l+i) + (k+1)(l + j)} 
			\leq [\ban^{\infty}_{l+i}(FS(N_k)) \otimes_{\pi} \ban^{\infty}_{l + j}(FS(N_{k+1}))] \cdot 
			\\
			\cdot \exp(\epsilon (k(l+i) + (k+1)(l + j))).
		\end{multline}
		Then Lemma \ref{lem_inj_proj_bnd_nntr} and (\ref{eq_comm111}), (\ref{eq_comm112}), (\ref{eq_comm113}), imply that 
		\begin{equation}\label{eq_comm114}
			N_{k(l+i) + (k+1)(l + j)} 
			\leq \ban^{\infty}_1(FS(N_k)^{l+i} \cdot FS(N_{k+1})^{l + j}) 
			\cdot \exp(3 \epsilon (k(l+i) + (k+1)(l + j))).
		\end{equation}
		From (\ref{eq_fs_unif}) and (\ref{eq_comm114}), we conclude that the upper bound from (\ref{eq_char1}) holds for $k := k(l+i) + (k+1)(l + j)$ for any $l \geq l_0$ and $i = 0, \ldots, k$, $j = 0, \ldots, k - 1$, with $4 \epsilon$ in place of $\epsilon$.
		Note, however, that since $k$ and $k + 1$ are relatively prime, the set of such numbers fully covers the set of sufficiently large natural numbers.
		This finishes the proof.
	\end{proof}
	As a final observation, we emphasize that the derivation of (\ref{eq_comm114}) does not rely on the continuity of $FS(N)$. 
	In particular, the upper bound in (\ref{eq_comm114}) implies that every submultiplicative norm on $R(Y,A)$ is bounded from above in the sense of (\ref{eq_bound_metr}). 
	This phenomenon is closely tied to the fact that \(R(Y,A)\) is finitely generated.
	In contrast, if one considers a section ring that is not finitely generated, such an upper bound may fail to hold, as illustrated by the following example.
	\par 
	Consider a compact complex manifold $X$ and a line bundle $L$ over it such that the section ring $R(X, L)$ is not finitely generated. 
	We will show that, in this case, the upper bound in (\ref{eq_bound_metr}) need not hold.
	To see this, consider a sequence of elements $s_i \in H^0(X, L^{\otimes k_i})$ for some increasing $k_i \in \nat$, so that each $s_i$ doesn't lie in the algebra induced by $\oplus_{k = 0}^{k_i - 1} H^0(X, L^{\otimes k})$. 
	We fix a metric $h^L$ on $L$ and consider a norm $N_k$ on $H^0(X, L^{\otimes k})$ which coincides $\ban^{\infty}_k(h^L)$ if there is no $i \in \nat$ such that $k = k_i$, and otherwise, if $k = k_i$ for some $i \in \nat$ such that it is no smaller than $\ban^{\infty}_k(h^L)$, coincides with it over the intersection of algebra induced by $\oplus_{k = 0}^{k_i - 1} H^0(X, L^{\otimes k})$ and $H^0(X, L^{\otimes k})$, and such that $\| s_i \|_k := \exp(k_i^2) \cdot \| s_i \|_{\ban^{\infty}_k(h^L)}$.
	Clearly, the resulting norm is not bounded, but an easy verification shows that it is submultiplicative: for any $s \in H^0(X, L^{\otimes k}), s' \in H^0(X, L^{\otimes l})$, we have $\| s \cdot s' \|_{k + l} = \| s \cdot s' \|_{\ban^{\infty}_{k + l}(h^L)} = \| s \|_{\ban^{\infty}_k(h^L)} \cdot \| s' \|_{\ban^{\infty}_l(h^L)} \leq \| s \|_k \cdot \| s' \|_l$.
	
	\subsection{Regularizable from above psh metrics and sup-norms}\label{sect_reg_metr}
	The primary aim of this section is to investigate the sup-norms associated with general metrics having potentials with minimal singularities instead of those given by the envelopes of continuous metrics. 
	The reason for presenting this study here is that the Fubini-Study metric corresponding to an arbitrary bounded norm is not in general given by an envelope of a continuous metric, cf. \cite[Proposition 3.8]{FinNarSim} for an example. 
	Consequently, the methods from Section \ref{sect_sm_sing} cannot be used as-is to prove Theorem \ref{thm_char}. 
	The main goal of the section is to show that the asymptotic study of sup-norms associated with general metrics having potentials with minimal singularities reduces to the study of the sup-norms associated with the so-called \textit{regularizable from above psh metrics}.
	\par 
	For ample line bundles, this phenomenon was already observed by Boucksom-Eriksson \cite{BouckErik21} and later developed by the author \cite{FinNarSim}. 
	Accordingly, the present section can be viewed as a generalization of those results to the big case. 
	\par
	The main technical challenge in extending the aforementioned works to the big setting arises from the need to identify a suitable substitute for the class of continuous metrics with plurisubharmonic potentials that is stable under taking minima. 
	This is a subtle issue, as due to singularities it requires working on the birational models of a given manifold rather than on the manifold itself -- a difficulty that does not occur in the ample case.
	\par 	
	Throughout this section, we fix a complex projective manifold $X$ and a big line bundle $L$ over $X$.
	We say that a function $\psi: X \to [-\infty, +\infty[$ has \textit{weakly neat analytic singularities} if for a certain birational model $\pi: \hat{X} \to X$, locally, $\psi \circ \pi$ can be written as in (\ref{eq_phi_alg_sing}), where $g$ is continuous.
	We say that a function $\psi: X \to [-\infty, +\infty[$ has \textit{weakly neat algebraic singularities} if in addition to the above condition, $c$ is rational.
	\begin{prop}\label{prop_max_alg_sing_weak_neat}
		Assume that $\psi_i : X \to [-\infty, +\infty[$, $i = 0, 1$, have weakly neat algebraic singularities. 
		Then $\max\{\psi_0, \psi_1\}$ also has weakly neat algebraic singularities.
	\end{prop}
	\begin{rem}\label{rem_prop_max_alg_sing_weak_neat}
		a) Since the maximum of two smooth functions is, in general, not smooth, it follows that the maximum of two potentials with neat algebraic singularities will not, in general, have neat algebraic singularities. 
		More strikingly, while the maximum of two continuous functions of course remains continuous, if, instead of weakly neat analytic singularities, one considers the smaller class of potentials for which the function $g$ in (\ref{eq_phi_alg_sing}) is continuous on $X$ itself (rather than on a birational model), then this class is also not closed under taking maxima.
		Indeed, as we know from analysis on Section \ref{sect_min_sing}, on $\comp^2$, the function $\max\{ \log |z_1|^2, \log |z_2|^2 \}$ writes as $\log (|z_1|^2 + |z_2|^2) + g$ for some bounded function $g$.
		But the reader will easily check that such $g$ is not continuous on $\comp^2$.
		Proposition \ref{prop_max_alg_sing_weak_neat} says, however, that it becomes continuous on some birational model of $\comp^2$, and this can be easily verified directly by considering the blow-up of $\comp^2$ at the origin, see (\ref{eq_cont_ext}).
		\par 
		b) Since any two birational models can be dominated by a third one, it is immediate that a positive rational linear combination of a finite number of potentials with weakly neat algebraic singularities has weakly neat algebraic singularities.
	\end{rem}
	\begin{proof}
		Following the same line of thought as in the proof of Proposition \ref{prop_max_alg_sing}, we see that it is enough to prove that for any $p_0, p_1 \in \nat^*$, neighborhood $U$ of $0 \in \comp^n$, holomorphic $f_{0, i}, f_{1, j} : U \to \comp$, $i = 1, \ldots, N$, $j = 1, \ldots, M$,  and continuous $g_0, g_1 : U \to \real$, the function $U \ni z \mapsto h_0(z)$, defined as
		\begin{multline}\label{eq_max_g3a}
		h_0 := \max \Big\{ p_0 \cdot \log \Big( \sum |f_{0, i}|^2 \Big) + g_0, p_1 \cdot \log \Big( \sum |f_{1, j}|^2 \Big) + g_1 \Big\}
		\\
		-
		\log \Big(
			\big( \sum |f_{0, i}|^2 \big)^{p_0} + \big( \sum |f_{1, j}|^2 \big)^{p_1}
		\Big)
		\end{multline}
		is continuous on the neighborhood of the preimage of $U$ on some birational model of $\comp^n$.
		\par 
		To prove this, we proceed in several steps.
		First, let us check that for any $N, M, R \in \nat$, and any continuous functions $G_0, G_1 : \comp^N \times \comp^M \times \comp^R \to \real$, the function that sends $X := (x_i)_{i = 1}^N$, $Y := (y_j)_{j = 1}^M$, $Z := (z_r)_{r = 1}^R$ $x_i, y_j, z_r \in \comp$, to 
		\begin{multline}\label{eq_max_g3a1}
		h_1(X, Y, Z)
		=
		\max \Big\{ p_0 \cdot \log \Big( \sum |x_i|^2 \Big) + G_0(X, Y, Z), p_1 \cdot \log \Big( \sum |y_j|^2 \Big) + G_1(X, Y, Z) \Big\}
		\\
		-
		\log \Big(
			\big( \sum |x_i|^2 \big)^{p_1} + \big( \sum |y_j|^2 \big)^{p_2}
		\Big)
		\end{multline}
		extends continuously to some birational model over $\comp^N \times \comp^M \times \comp^R$.
		\par 
		When $p_0 = p_1 = 1$, the result is immediate; indeed, the function above extends continuously to the blow-up of  $\comp^N \times \comp^M \times \comp^R$ along $0 \times 0 \times \comp^R$, by taking the value 
		\begin{multline}\label{eq_cont_ext}
		h_1([X:Y], Z) :=
		\max \Big\{ \log \Big( \sum |x_i|^2 \Big) + G_0(0, 0, Z), \log \Big( \sum |y_j|^2 \Big) + G_1(0, 0, Z) \Big\}
		\\
		-
		\log \Big(
			\sum |x_i|^2 + \sum |y_j|^2
		\Big)
		\end{multline}
		at the point of the exceptional divisor with homogeneous coordinates $([X:Y], Z) \in \mathbb{P}^{N + M - 1} \times \comp^R$. 
		We claim that the statement for arbitrary $p_0, p_1 \in \nat^*$ reduces to this special one.
		\par 
		To see this reduction, we introduce additional variables, $\overline{X}_k$ (resp. $\overline{Y}_l$), where $k$ (resp. $l$) runs through the index set $I$ (resp. $J$) enumerating all possible monomials of degree $p_1$ in $N$ (resp. $p_2$ in $M$) variables.
		For arbitrary continuous functions $H_0, H_1 : \comp^{|I|} \times \comp^{|J|} \times \comp^R \to \real$, we then consider the function that sends $\overline{X} := (X_k)_{k = 1}^{|I|}$, $\overline{Y} := (Y_k)_{k = 1}^{|J|}$, $Z := (z_r)_{r = 1}^R$, $X_k, Y_l, z_r \in \comp$, to 
		\begin{multline}\label{eq_max_g3a2}
			h_2(\overline{X}, \overline{Y}, Z):
			=
			\max \Big\{ \log \Big( \sum \binom{p_0}{k} \cdot |X_k|^2 \Big) + H_0(\overline{X}, \overline{Y}, Z),
			\\
			\log \Big( \sum \binom{p_1}{l} \cdot  |Y_l|^2 \Big) + H_1(\overline{X}, \overline{Y}, Z) \Big\}
			\\
			-
			\log \Big(
				\sum \binom{p_0}{k} \cdot |X_k|^2  + \sum \binom{p_1}{l} \cdot |Y_l|^2
			\Big),
		\end{multline}
		where above $\binom{p_0}{k}$ and $\binom{p_1}{l}$ represent the coefficients in the binomial theorem corresponding to the monomials associated with the indices $k$ and $l$.
		For simplicity, we denote $C_{N, M, R} := \comp^N \times \comp^M \times \comp^R$ and $B := \comp^{|I|} \times \comp^{|J|}  \times \comp^R$.
		A trivial modification of the above argument for $p_0, p_1 = 1$ shows that $h_2$ extends continuously to the blow-up at $0 \times 0 \times \comp^R$ of $\comp^{|I|} \times \comp^{|J|} \times \comp^R$, which we denote by $\hat{B}$.
		\par 
		\begin{sloppypar}
		Consider the natural map $b : C_{N, M, R} \to B$, which sends $((x_i)_{i = 1}^N, (y_j)_{j = 1}^M, (z_r)_{r = 1}^R)$ to $((x^k)_{k \in I}, (y^l)_{l \in J}, (z_r)_{r = 1}^R)$, where $x^k$ and $y^l$ represent the monomials associated with $k$ and $l$ respectively.
		As $b$ is a closed embedding, by Tietze-Urysohn-Brouwer extension theorem one can find continuous functions $H_0, H_1$ on $B$ which are related to $G_0, G_1$ as $G_0 = b^* H_0$ and $G_1 = b^* H_1$.
		Then immediate application of the binomial theorem says that $h_1 = b^* h_2$.
		\end{sloppypar}
		We resolve the indeterminacies of the rational map from $C_{N, M, R}$ to $\hat{B}$ as in the diagram
		\begin{equation}
			\begin{tikzcd}
			\hat{C}_{N, M, R} \arrow[r, "\hat{b}"] \arrow[d, "\pi_1"] & \hat{B} \arrow[d, "\pi_2"]  \\
			C_{N, M, R} \arrow[r,"b"] & B.
		\end{tikzcd}
		\end{equation}
		Since $h_1 = b^* h_2$, we have $\pi_1^* h_1 = \hat{b}^* \pi_2^* h_2$. Moreover, as $\pi_2^* h_2$ is continuous, $\pi_1^* h_1$ is continuous as well, finishing the proof of the fact that $h_1$ is continuous on a birational model.
		\par 
		Let us finally treat the general case. 
		We denote $R := n + m$.
		The functions $f_{0, i}$, $i = 1, \ldots, N$ and $f_{1, j}$, $j = 1, \ldots, M$, form a holomorphic map $a: \comp^n \times \comp^m \to C_{N, M, R}$, which sends $(X, Y) \in \comp^n \times \comp^m$ to $((f_{0, i}(X))_{i = 1}^{N}, (f_{1, j}(Y))_{j = 1}^{N}, X, Y)$.
		We resolve the indeterminacies of the induced rational map from $\comp^n \times \comp^n$ to $\hat{C}_{N, M, R}$ as in the following diagram
		\begin{equation}
			\begin{tikzcd}
			D \arrow[r, "\hat{a}"] \arrow[d, "\pi_0"] & \hat{C}_{N, M, R}  \arrow[d, "\pi_1"]  \\
			\comp^n \times \comp^m \arrow[r,"a"] & C_{N, M, R}.
		\end{tikzcd}
		\end{equation}
		Note that the map $a$ is a closed embedding, and so again by Tietze-Urysohn-Brouwer extension theorem, we can find continuous functions $H_0, H_1 : C_{N, M, R} \to \real$ such that $g_0 = a^* H_0$, $g_1 = a^* H_1$.
		Then immediately from the definitions, we have $h_0 = a^* h_1$, which implies $\pi_0^* h_0 = \hat{a}^* \pi_1^* h_1$.
		This finishes the proof, because we already established that $\pi_1^* h_1$ is continuous.
	\end{proof}	
	\par
	We can now proceed to the main definition of this section.
	\begin{defn}\label{defn_regul}
		We say that a metric $h^L$ on $L$ with a potential of minimal singularities is regularizable from above if there is a decreasing sequence of metrics $h^L_k$ on $L$, the potentials of which are psh and have weakly neat algebraic singularities, and which converge, as $k \to \infty$, towards a metric $h^{L, 0}$, so that $h^{L, 0} \geq h^L \geq h^{L, 0}_*$. 
		If $h^L$ moreover has a psh potential, we say that $h^L$ is a regularizable from above psh metric, and the above condition is equivalent to requiring that $h^L_k$ converge, as $k \to \infty$, towards $h^L$ almost everywhere.
	\end{defn}
	\par 
	We note that in the above definition, no upper semicontinuity assumption is imposed on the potential of $h^L$.
	Nevertheless, from \cite[(I.4.23)]{DemCompl}, it follows that whenever $h^L$ is regularizable from above, its lower semicontinuous regularization $h^L_*$ is a regularizable from above psh metric.
	Moreover, by \cite{BedfordTaylor}, $h^L = h^L_*$ outside of a pluripolar subset.
	\par 
	From Proposition \ref{prop_max_alg_sing_weak_neat} and Remark \ref{rem_prop_max_alg_sing_weak_neat}b), it is immediate that a class of regularizable from above metrics is closed under the formation of the minima and products.
	We also point out that while every metric with a psh potential can be approximated from below by metrics with neat algebraic singularities by Theorem \ref{thm_dem_appr}, not every metric with a psh potential is regularizable from above.
	Indeed, the discontinuity set of a regularizable from above metric is necessarily pluripolar by the Bedford-Taylor characterization of negligible sets \cite[\S 5]{BedfordTaylor}; see also \cite{BedfordRegulBelow} for a converse statement.
	\par
	While the above definition is, a priori, more general than the one used in the ample case, see \cite[Definition 1.5]{FinNarSim}, the following proposition shows that they are equivalent.
	\begin{prop}
		A metric $h^L$ with a bounded potential on an ample line bundle $L$ is regularizable from above if and only if there is a decreasing sequence of metrics $h^L_k$, $k \in \nat$, on $L$, which have continuous psh potentials and which converge towards $h^L$ almost everywhere. 
	\end{prop}
	\begin{proof}
		One direction is immediate, and we only have to establish that if there is a decreasing sequence of metrics $h^L_k$ on $L$, which have weakly neat algebraic singularities, then one can construct $h^L_{k, 0}$ which will be in addition continuous.
		To see this, fix an arbitrary metric $h^L_{-1}$ on $L$ which has a continuous psh potential, and which verifies $h^L_{-1} \geq h^L$.
		We claim that the sequence $h^L_{k, 0} := \min\{ h^L_{-1}, h^L_k \}$ verifies the above claim.
		Indeed, $h^L_{k, 0}$ has a psh potential, as the maximum of two psh functions is psh.
		The sequence $h^L_{k, 0}$, $k \in \nat$, also converges towards $h^L$ almost everywhere, as $h^L_k$ does so. 
		Finally, $h^L_{k, 0}$ is continuous.
		Indeed, the potential of $h^L_{-1}$ is continuous, and by assumption the potential of $h^L_k$ is continuous outside its polar set. The maximum of two such potentials is obviously continuous.
	\end{proof}
	\par 
	In what follows, for any metric $h^L$ with a potential of minimal singularities, we shall associate a certain regularizable from above metric.
	This will be done through an envelope construction, replicating what was done in the ample case in \cite[\S 7.5]{BouckErik21}.
	For this, we define
	\begin{multline}
		Q(h^L)
		:=
		\inf \Big\{
			h^L_0 \text{ with a psh potential}
		\\
			\text{of weakly neat algebraic singularities verifying } h^L \leq h^L_0
		\Big\}.
	\end{multline}
	Note that by Theorem \ref{thm_dem_appr}, there are metrics with neat algebraic singularities on $L$, and since $h^L$ has a potential with minimal singularities, the above infimum is taken over a nonempty set.
	\par 
	\begin{prop}\label{prop_regul_env}
		For any metric $h^L$ with a potential of minimal singularities, the metric $Q(h^L)$ is regularizable from above. 
	\end{prop}
	\begin{proof}
		Choquet's lemma, cf. \cite[(I.4.23)]{DemCompl}, implies that there is a sequence $h^L_{i, 0}$, $i \in \nat$, of metrics on $L$ with psh potentials of weakly neat algebraic singularities such that $h^L_{i, 0} \geq h^L$ and
		\begin{equation}\label{eq_0prop_regul_env}
			Q(h^L)_*
			:=
			\big\{ \inf_{i \in \nat} h^L_{i, 0} \big\}_*.
		\end{equation}
		By Proposition \ref{prop_max_alg_sing_weak_neat}, the sequence 
		\begin{equation}
			h^L_i := \min \{ h^L_{0, 0}, h^L_{1, 0}, \ldots, h^L_{i, 0} \},
		\end{equation}				
		provides a decreasing sequence of metrics on $L$ with psh potentials of weakly neat algebraic singularities.
		We denote by $h^{L, 0}$ the pointwise limit of $h^L_i$, as $i \to \infty$.
		Note that since $h^L_i \geq h^L$, we have $h^L_i \geq Q(h^L)$, and as a consequence $h^{L, 0} \geq Q(h^L)$.
		However, by (\ref{eq_0prop_regul_env}), we have $h^{L, 0}_* = Q(h^L)_*$, and so we have $h^{L, 0} \geq Q(h^L) \geq h^{L, 0}_*$, which finishes the proof.
	\end{proof}
	
	\begin{prop}
		For any metric $h^L$ with a potential of minimal singularities, we have $Q(Q(h^L)) = Q(P(h^L)) = Q(h^L)$.
	\end{prop}
	\begin{proof}
		Note that we have the following inequalities 
		\begin{equation}\label{eq_q_env_ord}
			Q(h^L) \geq P(h^L) \geq h^L.
		\end{equation}
		Since forming an envelope is a monotone procedure, it is only left to establish the inequality $Q(Q(h^L)) \leq Q(h^L)$.
		For this, we consider a metric $h^L_0$ with a psh potential of weakly neat algebraic singularities verifying $h^L_0 \geq h^L$.
		By the definition of $Q(h^L)$, we then have $h^L_0 \geq Q(h^L)$.
		Then we also have $h^L_0 \geq Q(Q(h^L))$.
		The inequality $Q(h^L) \geq Q(Q(h^L))$ follows by taking the infimum over all possible $h^L_0$.
	\end{proof}
	
	The following result provides a version of \cite[Corollary 7.27]{BouckErik21} for big line bundles and, in view of (\ref{eq_q_env_ord}), yields a refinement of (\ref{eq_max_prin}).
	\begin{prop}\label{thm_regul_blw}
		For any metric $h^L$ with a potential of minimal singularities and any $p \in [1, +\infty[$, 
		\begin{equation}
			\ban^{\infty}(h^L) = \ban^{\infty}(Q(h^L)).
		\end{equation}
	\end{prop}
	\begin{proof}
		The result follows immediately from the maximum principle in the spirit of (\ref{eq_max_prin}), and the proof presented in the course of the proof of Proposition \ref{prop_pull_back} repeats verbatim; one only has to change the envelope $P_{\theta}(\psi)$ in the proof by the potential associated with $Q(h^L)$.
	\end{proof}
	
	The following result gives a version of Theorem \ref{thm_isom}.
	Although it is not immediately clear that the theorem below implies Theorem \ref{thm_isom}, it will follow from Proposition \ref{prop_env_coinc}.
	
	\begin{thm}\label{thm_isom_reg}
		For any regularizable from above metrics $h^L_0$, $h^L_1$ on $L$, and any $p \in [1, +\infty[$,
		\begin{equation}
			d_p(h^L_{0 *}, h^L_{1 *})
			=
			\lim_{k \to \infty}
			\frac{d_p(\ban^{\infty}_k(h^L_0), \ban^{\infty}_k(h^L_1))}{k}.
		\end{equation}
	\end{thm}
	The proof of Theorem \ref{thm_isom_reg} relies on the following result.
	\begin{lem}\label{lem_isom_reg}
		For any regularizable from above metric $h^L$ on $L$, and any continuous metric $h^L_0$ on $L$, verifying $h^L_0 \leq h^L$, we have
		\begin{equation}\label{eqlem_isom_reg}
			\limsup_{k \to \infty}
			\frac{d_p(\ban^{\infty}_k(h^L), \ban^{\infty}_k(h^L_0))}{k}
			\leq 
			d_p(h^L_*, P(h^L_0)).
		\end{equation}
	\end{lem}
	\begin{proof}
		Let us consider a decreasing sequence of metrics $h^L_k$, $k \in \nat^*$ on $L$, which have weakly neat algebraic singularities and which converge towards a metric $h^{L, 0}$, so that $h^{L, 0} \geq h^L \geq h^{L, 0}_*$.
		Then if one replaces $h^L$ by $h^{L, 0}$ the right-hand side of (\ref{eqlem_isom_reg}) remains intact, but since we have $\ban^{\infty}_k(h^{L, 0}) \geq \ban^{\infty}_k(h^L) \geq \ban^{\infty}_k(h^L_0)$, the left-hand side of (\ref{eqlem_isom_reg}) increases by (\ref{eq_log_rel_spec2}). 
		In particular, it suffices to establish Lemma \ref{lem_isom_reg} for $h^L := h^{L, 0}$, which we assume from now on.
		We assume hence that the sequence $h^L_k$ decreases towards $h^L$.
		\par 
		For convenience, we fix a smooth metric $h^{L, 1}$ and denote by $\psi_k$ (resp. $\psi$) the potentials of $h^L_k$ (resp. $h^L$) with respect to $h^{L, 1}$.
		We will now prove that for any $\epsilon > 0$, there are $m, N \in \nat^*$ such that for any $k \geq m$, $N | k$, there is a vector subspace $\iota_k : V_k \to H^0(X, L^{\otimes k})$, verifying
		\begin{equation}\label{eq_mab_appl_isom_ress00}
			\frac{\dim V_k}{H^0(X, L^{\otimes k})} \geq 1 - \epsilon,
		\end{equation}
		and such that we have
		\begin{equation}\label{eq_mab_appl_isom_ress01}
			\limsup_{k \to \infty, N | k}
			\frac{d_p(\iota_k^*  \ban^{\infty}_k(h^L), \iota_k^* \ban^{\infty}_k(h^L_0))}{k}
			\leq 
			d_p(P[\psi_m](\psi^*)^*, P[\psi_m](\psi_0)).
		\end{equation}
		\par
		To do so, by Theorems \ref{thm_bouck_vol} and \ref{thm_cont_begz}, we pick $m \in \nat$, so that we have
		\begin{equation}\label{eq_psi_k_approx}
			\int (\theta +  \mathrm{d} \mathrm{d}^c \psi_m)^n \geq {\rm{vol}}(L) - \epsilon.
		\end{equation}
		We, moreover, assume that $\epsilon > 0$ is small enough so that ${\rm{vol}}(L) - \epsilon > 0$.
		We will now construct the subspaces $V_k$ using a version of the second construction in Section \ref{sect_subr}.
		We consider the birational model $\pi: \hat{X}_m \to X$, and the $\mathbb{Q}$-divisor $E_m$ as in (\ref{eq_resol_psi_sing}), i.e., such that
		\begin{equation}\label{eq_resol_psi_sing_k}
			\psi_m \circ \pi(x)
			=
			\log |s(x)|_{h^{E_m}} + g_m, 
		\end{equation}	 
		where $s_m$ is canonical holomorphic sections of $\mathscr{O}(E_m)$, $h^{E_m}$ is a smooth metric on the $\mathbb{Q}$-line bundle $\mathscr{O}(E_m)$, and $g_m$ is a bounded function.
		Note that after performing additional blow-ups of $\hat{X}_m$, we may further assume that $g_m$ is continuous, as we have assumed that $\psi_m$ has weakly neat algebraic singularities.
		We make this additional assumption from now on.
		\par 
		Over $\hat{X}_m$, consider the $\mathbb{Q}$-line bundle $F_m := \pi^* L \otimes \mathscr{O}(-E_m)$.
		Then by Theorem \ref{thm_bouck_vol}, (\ref{eq_resol_psi_sing_k}) and the boundedness of $g_m$, we have
		\begin{equation}\label{eq_mab_appl_isom_ress2}
			{\rm{vol}}_{\hat{X}_m}(F_m)
			=
			\int (\theta +  \mathrm{d} \mathrm{d}^c \psi_m)^n.
		\end{equation}
		By (\ref{eq_psi_k_approx}), we see that $F_m$ is big.
		Choose $N \in \nat^*$ such that $N \cdot E_m$ defines a $\mathbb{Z}$-divisor.
		For any $k \in \nat$, $N | k$, consider the vector space $V_k := H^0(\hat{X}_m, F_m^{\otimes k})$.
		As described in the end of Construction 2 from Section \ref{sect_subr}, we have a natural injection $\iota_k : V_k \to H^0(X, L^{\otimes k})$.
		Clearly, it verifies (\ref{eq_mab_appl_isom_ress00}) by (\ref{eq_psi_k_approx}).
		Let us verify that (\ref{eq_mab_appl_isom_ress01}) holds as well.
		For this, we shall later prove that for any $l \geq m$, we have
		\begin{equation}\label{eq_mab_appl_isom_ress01aa}
			\limsup_{k \to \infty, N | k}
			\frac{d_p(\iota_k^*  \ban^{\infty}_k(h^L), \iota_k^* \ban^{\infty}_k(h^L_0))}{k}
			\leq 
			d_p(P[\psi_m](\psi_l)^*, P[\psi_m](\psi_0)^*).
		\end{equation}
		Note that since $\psi_m$ has analytic singularities, the right-hand side of (\ref{eq_mab_appl_isom_ress01aa}) is well-defined by Proposition \ref{prop_model_pot}.
		From Lemma \ref{lem_incr}, we see that (\ref{eq_mab_appl_isom_ress01aa}) implies (\ref{eq_mab_appl_isom_ress01}).
		\par 
		Let us now establish (\ref{eq_mab_appl_isom_ress01aa}).
		Note first that while $\ban^{\infty}_k(h^L_l)$ is not a well-defined norm on $H^0(X, L^{\otimes k})$ (it might take $+\infty$ values), for any $l \geq m$, $N | k$, the norm $\iota_k^*  \ban^{\infty}_k(h^L_l)$ is well-defined on $V_k$.
		By our monotonicity assumption, we have $\iota_k^* \ban^{\infty}_k(h^L_0) \leq \iota_k^* \ban^{\infty}_k(h^L) \leq \iota_k^* \ban^{\infty}_k(h^L_l)$.
		Hence, by (\ref{eq_log_rel_spec2}), we obtain
		\begin{equation}
			 d_p(\iota_k^*  \ban^{\infty}_k(h^L), \iota_k^* \ban^{\infty}_k(h^L_0))
			 \leq
			 d_p(\iota_k^*  \ban^{\infty}_k(h^L_l), \iota_k^* \ban^{\infty}_k(h^L_0)).
		\end{equation}
		We denote by $h^{E_m}_{{\rm{sing}}}$ the canonical singular metric on the line bundle $\mathscr{O}(E_m)$.
		By Proposition \ref{prop_pull_back} and the fact that the sup-norms behave predictably under the pull-back, we deduce that 
		\begin{equation}
			\iota_k^*  \ban^{\infty}_k(h^L_l)
			=
			\ban^{\infty}_k(P(\pi^* h^L_l/h^{E_m}_{{\rm{sing}}})),
			\qquad
			\iota_k^* \ban^{\infty}_k(h^L_0)
			=
			\ban^{\infty}_k(P(\pi^* h^L_0/h^{E_m}_{{\rm{sing}}})).
		\end{equation}
		According to Lemma \ref{lem_cont_sing_m} (which applies to the line bundle $F_m$ as it is big), there is a continuous metric $h^{L, 0}_0$ on $F_k$ such that 
		\begin{equation}
			P(\pi^* h^L_0/h^{E_m}_{{\rm{sing}}})
			=
			P(h^{L, 0}_0)
		\end{equation}
		Let us establish that the same holds for $P(\pi^* h^L_l/h^{E_m}_{{\rm{sing}}})$ modulo a birational modification.
		To see this, note that since $h^L_l$ has weakly neat algebraic singularities and $h^L_l \leq h^L_m$, there is a birational model $\pi_l : \hat{X}'_m \to \hat{X}_m$ such that the potential $\psi_l^0$ of the metric $\pi_l^*(\pi^* h^L_l/h^{E_m}_{{\rm{sing}}})$ is continuous outside of its $+\infty$-set.
		As in the proof of Lemma \ref{lem_cont_sing_m}, we then see that there is a continuous metric $h^{L, 0}_l$ on $\pi_l^* F_k$ such that 
		\begin{equation}
			P(\pi_l^*(\pi^* h^L_l/h^{E_m}_{{\rm{sing}}}))
			=
			P(h^{L, 0}_l)
		\end{equation}
		Now, as $F_m$ is big and $\pi_l$ is a birational modification, $\pi_l^* F_m$ is big as well.
		Moreover, under the isomorphism $H^0(\hat{X}_m, F_m^{\otimes k}) \simeq H^0(\hat{X}'_m, \pi_m^* F_m^{\otimes k})$, the sup-norms transform into the sup-norms.
		Hence, by Theorem \ref{thm_isom}, applied to continuous metrics $\pi_l^* h^{L, 0}_0$ and $h^{L, 0}_l$ on $\pi_l^* F_m$, using (\ref{eq_bir_eq_mab}), we see that for any $l \geq m$, we have
		\begin{equation}\label{eq_mab_appl_isom_ress01aaeee}
			\lim_{k \to \infty, N | k}
			\frac{d_p(\iota_k^*  \ban^{\infty}_k(h^L_l), \iota_k^* \ban^{\infty}_k(h^L_0))}{k}
			=
			d_p(P(h^L_l/h^{E_m}_{{\rm{sing}}})_*, P(h^L_0/h^{E_m}_{{\rm{sing}}})_*).
		\end{equation}
		By (\ref{eq_metr_pot_dist_corresp}) and (\ref{eq_mab_appl_isom_ress01aaeee}), we deduce (\ref{eq_mab_appl_isom_ress01aa}).
		\par 
		Now, by (\ref{eq_contraction_prop}) and (\ref{eq_metr_pot_dist_corresp}), we obtain
		\begin{equation}\label{eq_reg_final_qqq}
			d_p(P[\psi_m](\psi^*)^*, P[\psi_m](\psi_0)^*)
			\leq
			d_p(h^L_*, P(h^L_0))
		\end{equation}
		A combination of (\ref{eq_comp_dist_two}), (\ref{eq_mab_appl_isom_ress00}), (\ref{eq_mab_appl_isom_ress01}) and (\ref{eq_reg_final_qqq}) shows that
		\begin{equation}\label{eq_reg_final_qqqq}
			\limsup_{k \to \infty, N | k} \frac{d_p(\ban^{\infty}_k(h^L), \ban^{\infty}_k(h^L_0))^p}{k^p}
			\leq 
			d_p(h^L_*, P(h^L_0))^p 
			+
			3 \epsilon \cdot \sup |\log(h^L / P(h^L_0))|^p.
		\end{equation}
		\par 
		By repeating the same argument as in (\ref{eq_jk_morph_mult_triv}), we deduce that the above bound continues to hold even if $\limsup_{k \to \infty, N | k}$ is replaced by $\limsup_{k \to \infty}$ in (\ref{eq_reg_final_qqqq}).
		As $\epsilon > 0$ can be chosen arbitrary small, this finishes the proof.
	\end{proof}
	Let us show that Lemma \ref{lem_isom_reg} implies Theorem \ref{thm_isom_reg}.
	\begin{proof}[Proof of Theorem \ref{thm_isom_reg}.]
		Since the potentials of $h^L_j$, $j = 0, 1$, are upper semicontinuous, we can find increasing sequences $h^L_{j, i}$, $i \in \nat$, of continuous metrics converging pointwise towards $h^L_{j *}$.
		Note that we then have $h^L_j \geq h^L_{j *} \geq P(h^L_{j, i}) \geq h^L_{j, i}$ for any $j = 0, 1$, $i \in \nat$, and so $P(h^L_{j, i})$ converge towards $h^L_{j *}$, as $i \to \infty$.
		By the definition of the Mabuchi-Darvas distance, (\ref{eq_ext_dp_gennn}),
		\begin{equation}\label{eq_mab_appl_isom_ress}
			d_p(h^L_{0 *}, h^L_{1 *})
			=
			\lim_{i \to \infty}
			d_p(P(h^L_{0, i}), P(h^L_{1, i})),
			\qquad
			\lim_{i \to \infty}
			d_p(h^L_{j *}, P(h^L_{j, i}))
			=
			0.
		\end{equation}
		By the second limit from (\ref{eq_mab_appl_isom_ress}), Lemma \ref{lem_isom_reg} and the quasi-metric properties of the distance, see (\ref{eq_gen_triangle}), we conclude that for any $\epsilon > 0$, there is $i \in \nat$ such that 
		\begin{equation}\label{eq_mab_appl_isom_ress1}
			\limsup_{k \to \infty}
			\Big|
			\frac{d_p(\ban^{\infty}_k(h^L_0), \ban^{\infty}_k(h^L_1)) - d_p(\ban^{\infty}_k(P(h^L_{0, i})), \ban^{\infty}_k(P(h^L_{1, i})))}{k}
			\Big|
			\leq
			\epsilon.
		\end{equation}
		However, by Theorem \ref{thm_isom}, we have
		\begin{equation}\label{eq_mab_appl_isom_ress2}
			\lim_{k \to \infty}
			\frac{d_p(\ban^{\infty}_k(P(h^L_{0, i})), \ban^{\infty}_k(P(h^L_{1, i})))}{k}
			=
			d_p(P(h^L_{0, i}), P(h^L_{1, i})).
		\end{equation}
		From (\ref{eq_mab_appl_isom_ress1}), (\ref{eq_mab_appl_isom_ress2}) and the first limit from (\ref{eq_mab_appl_isom_ress}), we finish the proof.
	\end{proof}
	As an application of the above analysis, we establish the following result. 
	It essentially shows that studying sup-norms associated with metrics having potentials with minimal singularities can be reduced to studying sup-norms associated with regularizable from above psh metrics.
	\begin{prop}\label{prop_eq_ban_q_weak}
		For any metric $h^L$ with a potential of minimal singularities, and any $p \in [1, + \infty[$, we have
		\begin{equation}\label{eqprop_eq_ban_q_weak}
			\ban^{\infty}(h^L) \sim_p \ban^{\infty}(Q(h^L)_*).
		\end{equation}
	\end{prop}
	\begin{proof}
		Remark first that if $h^L$ has an upper semicontinuous potential, then by taking the lower semicontinuous regularization of the inequality $Q(h^L) \geq h^L$ from (\ref{eq_q_env_ord}), we obtain  $Q(h^L) \geq Q(h^L)_* \geq h^L_* = h^L$, showing in light of Proposition \ref{thm_regul_blw} that we even have $\ban^{\infty}(h^L) = \ban^{\infty}(Q(h^L)_*)$.
		In general, however, such a strong statement doesn't hold, as the reader can easily verify by considering an arbitrary smooth metric and increasing its value at one given point. 
		Let us hence proceed to the proof of (\ref{eqprop_eq_ban_q_weak}).
		\par 
		By Proposition \ref{thm_regul_blw}, it suffices to establish that 
		\begin{equation}
			\lim_{k \to \infty} \frac{d_p(\ban^{\infty}_k(Q(h^L)), \ban^{\infty}_k(Q(h^L)_*))}{k}
			=
			0.
		\end{equation}
		This, however, follows immediately from Theorem \ref{thm_isom_reg}, as both $Q(h^L)$ and $Q(h^L)_*$ are regularizable from above by Proposition \ref{prop_regul_env} and the discussion after Definition \ref{defn_regul}.
	\end{proof}
	Let us now state a result which refines Lemma \ref{lem_fs_sm}.
	\begin{prop}\label{prop_fs_regul_blw}
		For any submultiplicative bounded graded norm $N = (N_k)_{k = 0}^{\infty}$, the metric $FS(N)$ is regularizable from above.
	\end{prop}	
	\begin{proof}
		Note that by Lemma \ref{lem_fs_sm}, we have $FS(N) = \lim_{k \to \infty} FS(N_{2^k})^{\frac{1}{2^k}}$, and the sequence of metrics $FS(N_{2^k})^{\frac{1}{2^k}}$ is decreasing.
		Hence, it suffices to show that for an arbitrary norm $N_k$ on $H^0(X, L^{\otimes k})$, the metric $FS(N_k)$ has a potential with weakly neat algebraic singularities as long as $H^0(X, L^{\otimes k}) \neq \{0\}$.
		For Hermitian norm $N_k$, $FS(N_k)$ even has a potential with neat algebraic singularities as we described in Section \ref{sect_fs_oper}.
		The general case follows immediately from the alternative description of $FS(N_k)$ from (\ref{eq_fs_n_k_a_k_rel}), as the metric on the hyperplane bundle of $\mathbb{P}(H^0(X, L^{\otimes k})^*)$ induced by an arbitrary norm $N_k$ is continuous.
	\end{proof}
	As an application of the above results, we deduce the following statement, which in the ample case follows from \cite[Theorem 7.26]{BouckErik21}.
	\begin{cor}\label{cor_env_fs_eq}
		For any metric $h^L$ with a potential of minimal singularities, we have 
		\begin{equation}\label{eq_cor_env_fs_eq}
			FS(\ban^{\infty}(h^L))_* = Q(h^L)_*.
		\end{equation}
	\end{cor}
	\begin{proof}[Proof of Corollary \ref{cor_env_fs_eq} admitting Theorem \ref{thm_char}]
		We denote $N := \ban^{\infty}(h^L)$.
		Of course, $N$ is submultiplicative. 
		Note that since $h^L$ has a potential with minimal singularities, by (\ref{eq_bound_norm_min_sing}), the norm $N$ is bounded.
		Hence, we can apply Theorem \ref{thm_char} to get $N \sim_p \ban^{\infty}(FS(N)_*)$ for any $p \in [1, +\infty[$.
		By this, Theorem \ref{thm_isom_reg} and Propositions \ref{prop_regul_env}, \ref{prop_eq_ban_q_weak}, \ref{prop_fs_regul_blw}, we deduce that $d_p(FS(N)_*, Q(h^L)_*) = 0$.
		However, Gupta in \cite{GuptaPrakhar} established that $d_p$ is a distance, and in particular it separates the points, which implies (\ref{eq_cor_env_fs_eq}).
	\end{proof}
	Theorem \ref{thm_isom_reg} refines Theorem \ref{thm_isom}, as the following result suggests.
	\begin{prop}\label{prop_env_coinc}
		For any continuous metric $h^L$ on $L$, we have $P(h^L) = Q(h^L)_*$.
		In particular, $P(h^L)$ is regularizable from above.
		Moreover, we even have $P(h^L) = Q(h^L)$ outside of $\mathbb{B}_+(L)$, defined in (\ref{eq_aug_bl}).
	\end{prop}
	\begin{proof}
		The proof of the first part is a direct combination of Theorem \ref{thm_conv_fs} and Corollary \ref{cor_env_fs_eq}.
		Let us now provide an independent proof of the second part, which of course refines the first part.
		First, it is immediate that $P(h^L) \leq Q(h^L)$.
		Hence, it suffices to establish that $P(h^L) \geq Q(h^L)$ over $X \setminus \mathbb{B}_+(L)$.
		For this, it suffices to construct a sequence of metrics $h^L_k$ with potentials of weakly neat algebraic singularities, so that they converge towards $P(h^L)$ pointwise on $X \setminus \mathbb{B}_+(L)$. 
		We claim $h^L_k := FS(\ban^{\infty}_k(h^L))^{\frac{1}{k}}$ provides such an example.
		Indeed, by Theorem \ref{thm_conv_fs}, $h^L_k$ converges towards $P(h^L)$ pointwise on $X \setminus \mathbb{B}_+(L)$.
		Moreover, by the proof of Proposition \ref{prop_fs_regul_blw}, $h^L_k$ has a potential with weakly neat algebraic singularities.
	\end{proof}
	\par 
	In conclusion of this section, let us establish an analogue of Proposition \ref{prop_bm_volume} for the regularizable from above metrics.
	We now fix an arbitrary smooth volume form $dV_X$ on $X$.
	\begin{prop}\label{prop_bm_volume_reg}
		For any $\rho \in L^{\infty}(X)$, $\rho \geq 0$ such that ${\rm{ess\,supp}}(\rho) = X$, the following holds.
		For any $p \in [1, +\infty[$ and any regularizable from above metric $h^L$ on $L$, we have
		\begin{equation}\label{eq_bm_volume_reg}
			\ban^{\infty}(h^L) \sim_p \big(\ban^{\infty}((h^L)^k \cdot \rho ) \big)_{k = 0}^{\infty} \sim_p \hilb(h^L, \rho \cdot d V_X).
		\end{equation}
	\end{prop}
	\begin{proof}
		For simplicity, without loosing the generality, we may assume $\sup_{x \in X} \rho(x) = 1$.
		Remark first that for an arbitrary continuous metric $h^L_0$ verifying $h^L \geq h^L_0$, we have $\hilb_k(h^L, \rho \cdot d V_X) \geq \hilb_k(h^L_0, \rho \cdot d V_X)$.
		Hence, by applying Proposition \ref{prop_bm_volume} for $h^L := h^L_0$, we see that for any $\epsilon > 0$, there is $k_0 \geq 0$ such that for any $k \geq k_0$, we have 
		\begin{equation}\label{eq_bm_volume_reg1}
			\hilb_k(h^L, \rho \cdot d V_X)
			\geq
			\ban^{\infty}_k(h^L_0)
			\cdot
			\exp(-\epsilon k).
		\end{equation}
		From this and (\ref{eq_bound_norm_min_sing}), we conclude that there is $C > 0$ such that for any $k \in \nat$, we have
		\begin{equation}\label{eq_bm_volume_reg112312312}
			\ban^{\infty}_k(h^L)
			\geq
			\ban^{\infty}((h^L)^k \cdot \rho )
			\geq
			\hilb_k(h^L, \rho \cdot d V_X)
			\geq
			\ban^{\infty}_k(h^L)
			\cdot
			\exp(- C k).
		\end{equation}
		By this and \cite[Proposition 5.4]{FinEigToepl}, we deduce that it suffices to establish the equivalence between the first norm from (\ref{eq_bm_volume_reg}) and the last one for $p = 1$.
		However, by (\ref{eq_lidski_gen}) and (\ref{eq_bm_volume_reg1}), for any $\epsilon > 0$, there is $k_0 > 0$ such that for any $k \geq k_0$, we have 
		\begin{equation}
			d_1\Big(\ban^{\infty}_k(h^L), \hilb_k(h^L, \rho \cdot d V_X) \Big)
			\leq
			d_1\Big(\ban^{\infty}_k(h^L), \ban^{\infty}_k(h^L_0) \Big)
			+
			\epsilon k
			+
			6 \log n_k.
		\end{equation}
		By this and Theorem \ref{thm_isom_reg}, we deduce that
		\begin{equation}\label{eq_bm_volume_reg1121}
			\limsup_{k \to \infty}
			\frac{1}{k} d_1\Big(\ban^{\infty}_k(h^L), \hilb_k(h^L, \rho \cdot d V_X) \Big)
			\leq
			d_1 (h^L_*, P(h^L_0))
			+
			\epsilon.
		\end{equation}
		Now, we consider a sequence of continuous metrics $h^L_i$, $i \in \nat$, increasing pointwise towards $h^L$ (since the potential of $h^L$ is upper semicontinuous, such a sequence always exists).
		Then by applying (\ref{eq_bm_volume_reg1121}) for $h^L_0 := h^L_i$, and taking the limit $i \to \infty$, since $\epsilon > 0$ can be chosen to be arbitrarily small, we deduce that 
		\begin{equation}\label{eq_bm_volume_reg11121}
			\limsup_{k \to \infty}
			\frac{1}{k} d_1\Big(\ban^{\infty}_k(h^L), \hilb_k(h^L, \rho \cdot d V_X) \Big)
			\leq
			\lim_{i \to \infty}
			d_1 (h^L_*, P(h^L_i)).
		\end{equation}
		However, by (\ref{eq_ext_dp_gennn}), we have $\lim_{i \to \infty} d_1 (h^L_*, P(h^L_i)) = 0$, and (\ref{eq_bm_volume_reg11121}) finishes the proof.
	\end{proof}
	
	\subsection{Submultiplicative norms on section rings of big line bundles}\label{sec_sm_norm_big_pf}
	The main goal of this section is to prove Theorem \ref{thm_char}.
	For this, we shall rely on the results of both Sections \ref{sect_sm_sing} and \ref{sect_reg_metr}.
	\par 
	\begin{proof}[Proof of Theorem \ref{thm_char}]
		\begin{sloppypar}
		For the moment, let us admit that for any $\epsilon > 0$, there are $N \in \nat^*$, $k_0 \in \nat$ such that for any $k \geq k_0$, $N | k_0$, we have
		\begin{equation}\label{eq_v_k_ban_d1}
			\limsup_{k \to \infty, N | k}\frac{1}{k} d_p\big( N_k, \ban^{\infty}_k(FS(N)_*) \big) \leq \epsilon.
		\end{equation}
		Let us now show that (\ref{eq_v_k_ban_d1}) continues to hold even when $\limsup_{k \to \infty, N | k}$ is replaced by $\limsup_{k \to \infty}$.
		Once this is established, the proof of Theorem \ref{thm_char} is complete, as $\epsilon > 0$ can be chosen arbitrarily small.
		\par 
		Our argument will be similar to the one described after (\ref{eq_jk_morph_mult_triv}).
		We assume now that $r \in \nat$ is sufficiently large so that there is a nonzero $s_r \in H^0(X, L^{\otimes r})$.
		For any $k \in \nat$, consider the monomorphism $j_k$ from (\ref{eq_jk_morph_mult_triv}).
		As in (\ref{eq_jk_comp2}), by the boundedness of $N$, we then deduce
		\begin{equation}\label{eq_char_pf1}
			\limsup_{k \to \infty}
			\frac{1}{k}
			\Big|
				d_p(N_{k + r}, \ban^{\infty}_{k + r}(FS(N)_*))
				-
				d_p(j_k^* N_{k + r}, j_k^* \ban^{\infty}_{k + r}(FS(N)_*))
			\Big|
			=
			0.
		\end{equation}
		Note, however, that the norm $j_k^* \ban^{\infty}_{k + r}(FS(N)_*)$ can be described as the sup-norm $\ban^{\infty}_1(FS(N)_*^k \cdot \rho)$, where $\rho(x) := |s_r(x)|_{FS(N)_*^r}$, $x \in X$.
		On the other hand, we have $\| s \cdot s_{r} \|_{N_{k + r}} \leq C \cdot \| s \|_{N_{k}}$, where $C := \| s_{r} \|_{N_r}$.
		We thus have
		\begin{equation}\label{eq_char_pf2}
			\ban^{\infty}_1((FS(N)_*)^k \cdot \rho)
			=
			j_k^* \ban^{\infty}_{k + r}(FS(N)_*)
			\leq
			j_k^* N_{k + r}
			\leq
			C N_k.
		\end{equation}
		Together with Proposition \ref{prop_bm_volume_reg}, (\ref{eq_log_rel_spec2}), (\ref{eq_gen_triangle}) and (\ref{eq_char_pf2}), we obtain
		\begin{equation}\label{eq_char_pf3}
			\limsup_{k \to \infty}
			\frac{1}{k}
			d_p(j_k^* N_{k + r}, j_k^* \ban^{\infty}_{k + r}(FS(N)_*))
			\leq
			\limsup_{k \to \infty}
			\frac{1}{k}
			d_p(N_k, \ban^{\infty}_k(FS(N)_*)).
		\end{equation}
		A combination of (\ref{eq_char_pf1}) and (\ref{eq_char_pf3}) implies
		\begin{equation}
			\limsup_{k \to \infty}
			\frac{1}{k}
			d_p(N_{k + r}, \ban^{\infty}_{k + r}(FS(N)_*))
			\leq 
			\limsup_{k \to \infty}
			\frac{1}{k}
			d_p(N_k, \ban^{\infty}_k(FS(N)_*)).
		\end{equation}
		By choosing $s_i \in H^0(X, L^{\otimes i})$ for a set of indices $i \in I$ which covers all the residues modulo $N$ and using the same argument as after (\ref{eq_jk_comp2}), we now see that if (\ref{eq_v_k_ban_d1}) holds, then it  continues to hold even if $\limsup_{k \to \infty, N | k}$ is replaced by $\limsup_{k \to \infty}$. 
		\end{sloppypar}
		\par 
		Now, it is only left to establish (\ref{eq_v_k_ban_d1}).
		Our first claim is that for a given $\epsilon > 0$, $m \in \nat$, there are $k_0 \in \nat$, $N \in \nat^*$ such that for any $k \geq k_0$, $N | k$, we have
		\begin{equation}\label{eq_restr_sym_bnd}
			N_{km}|_{[\sym^k H^0(X, L^{\otimes m})]}
			\leq
			\ban^{\infty}_k(FS(N_m))|_{[\sym^k H^0(X, L^{\otimes m})]}
			\cdot
			\exp(\epsilon k).
		\end{equation}
		Our proof will rely on the Construction 1 from Section \ref{sect_subr}, from which we borrow the notations for $Y_m$, $A_m$, $E_m$, $X_m$.
		Then as described in Construction 1 from Section \ref{sect_subr}, there are $k_0 \in \nat$, $N \in \nat^*$ such that for any $k \geq k_0$, $N | k$, there is a natural inclusion
		\begin{equation}
			\iota_k : H^0(Y_m, A_m^{\otimes k}) \to H^0(X, L^{\otimes km}),
		\end{equation}
		with the identification
		\begin{equation}\label{eq_ident_am_sym}
			 \iota_k (H^0(Y_m, A_m^{\otimes k})) = [\sym^k H^0(X, L^{\otimes m})].
		\end{equation}
		\par 
		Clearly, the norms $(\iota_k^* N_{km})_{k = 0}^{+\infty}$ form a submultiplicative norm on $R(Y_m, A_m)$.
		Immediately from Proposition \ref{prop_induc} and (\ref{eq_reform_subm_cond}), we deduce that for any $\epsilon > 0$, there is $k_0 \in \nat$ such that for any $k \geq k_0$, $N | k$, we have 
		\begin{equation}\label{eq_ident_am_sy123m}
			\iota_k^* N_{km}
			\leq
			\ban^{\infty}_k(Y_m, FS(\iota_1^* N_m))
			\cdot
			\exp(\epsilon k).
		\end{equation}
		Furthermore, by using the identification as in (\ref{eq_fs_n_k_a_k_rel}) and considering the straightforward behavior of the sup-norms under pull-backs, we observe that
		\begin{equation}
			\ban^{\infty}_k(Y_m, FS(\iota_1^* N_m))
			=
			\iota_k^*
			\ban^{\infty}_k(X, FS(N_m)).
		\end{equation}
		This shows that (\ref{eq_ident_am_sy123m}) is just a reformulation of (\ref{eq_restr_sym_bnd}) by (\ref{eq_ident_am_sym}).
		\par 
		\begin{sloppypar}
		Applying (\ref{eq_restr_sym_bnd}) with $m := m^l$ and using the obvious inclusion $[\sym^{k  m^{l - 1}} H^0(X, L^{\otimes m})] \subset [\sym^{k} H^0(X, L^{\otimes m^l})]$, $k, l \in \nat$, we deduce that for $h^L_i := FS(N_{m^i})^{\frac{1}{m^i}}$, the following holds: for any $\epsilon > 0$, $i \in \nat$, there are $k_0 \in \nat$, $N \in \nat^*$ such that for any $k \geq k_0$, $N | k$, we have
		\begin{equation}\label{eq_restr_sym_bnd1}
			N_{km}|_{[\sym^k H^0(X, L^{\otimes m})]}
			\leq
			\ban^{\infty}_k(h^L_i)|_{[\sym^k H^0(X, L^{\otimes m})]}
			\cdot
			\exp(\epsilon k/2)
		\end{equation}
		In particular, we have
		\begin{multline}\label{eq_restr_sym_bnd101}
			\limsup_{k \to \infty, mN | k}\frac{1}{k} d_p\big( N_k, \ban^{\infty}_k(FS(N)_*) \big) 
			\\
			\leq 
			\epsilon/2 
			+ 
			\limsup_{k \to \infty, mN | k}\frac{1}{k} d_p\big( \iota_k^* \ban^{\infty}_k(h^L_i), \iota_k^* \ban^{\infty}_k(FS(N)_*) \big).
		\end{multline}
		Note, however, that the sequence of metrics $h^L_i$ are decreasing by Lemma \ref{lem_fs_sm}.
		Moreover, by the proof of Proposition \ref{prop_fs_regul_blw}, $h^L_i$ have potential with weakly neat algebraic singularities.
		From this and the fact that $h^L_i$ converge towards $FS(N)_*$ almost everywhere, by Lemma \ref{lem_incr} and (\ref{eq_mab_appl_isom_ress01aaeee}), we see that there are $i \in \nat$, $N \in \nat^*$ such that 
		\begin{equation}\label{eq_restr_sym_bnd102}
			\limsup_{k \to \infty, mN | k}\frac{1}{k} d_p\big( \iota_k^* \ban^{\infty}_k(h^L_i), \iota_k^* \ban^{\infty}_k(FS(N)_*) \big)
			\leq
			\epsilon/2.
		\end{equation}
		A combination of (\ref{eq_restr_sym_bnd101}) and (\ref{eq_restr_sym_bnd102}) implies (\ref{eq_v_k_ban_d1}), finishing the proof.
		\end{sloppypar}
	\end{proof}

	\section{Applications}\label{sect_5}
	The main goal of this section is to present several applications of our results and techniques. 
	More specifically, in Section \ref{sect_quant_geod}, we study the quantization of geodesics and rooftop norms; in particular, we show Theorem \ref{cor_appr_geod}. 
	In Section \ref{sect_geod_ray}, we associate a geodesic ray of singular metrics to any bounded submultiplicative filtration and establish Theorem \ref{thm_filt}.
	
	\subsection{Quantization of geodesics and rooftop norms}\label{sect_quant_geod}
	The main goal of this section is to establish Theorem \ref{cor_appr_geod} and discuss the rooftop norms, thereby refining and partially generalizing a result of Darvas-Lu-Rubinstein \cite[Theorem 1.3]{DarvLuRub}
	\par 
	In the proof of Theorem \ref{cor_appr_geod}, the following result will play a crucial role.
	For the definition of complex interpolation, we refer the reader to (\ref{eq_defn_cx_inter}).
	\begin{lem}\label{lem_cx_inter_sm}
		Assume that $R = \bigoplus R_k$ is a graded algebra such that $\dim R_k$ is finite.
		Let $N_0 = (N_{k, 0})_{k = 0}^{\infty}$ and $N_1 = (N_{k, 1})_{k = 0}^{\infty}$ be two submultiplicative norms on $R$.
		Then, for any $t \in [0, 1]$, the graded norm $N_t := (N_{k, t})_{k = 0}^{\infty}$, composed from complex interpolations $N_{k, t}$ between $N_{k, 0}$ and $N_{k, 1}$ is submultiplicative.
	\end{lem}
	\begin{proof}
		The result follows immediately from the definition (\ref{eq_defn_cx_inter}).
		For completeness, let us provide an argument.
		For given $k, l \in \nat$, we fix $s_1 \in R_k$, $s_2 \in R_l$.
		Then in the notations of (\ref{eq_defn_cx_inter}), for any $t \in ]0, 1[$, we have
		\begin{equation}\label{eq_sm_intrp}
		\| s_1 \cdot s_2 \|_{N_{k + l, t}}
		= \inf_{\substack{f \in \mathcal{F}(R_{k + l}) \\ f(t) = s_1 \cdot s_2}} \Big\{ 
		    \max \big( 
	   	     \sup_{y \in \mathbb{R}} \| f(iy) \|_{N_{k + l, 0}}, \;
	   	     \sup_{y \in \mathbb{R}} \| f(1+iy) \|_{N_{k + l, 1}}
	   	 \big)
		\Big\}.
		\end{equation}
		Since we have $\mathcal{F}(R_k) \cdot \mathcal{F}(R_l) \subset \mathcal{F}(R_{k + l})$, the right-hand side of (\ref{eq_sm_intrp}) increases if we replace the infimum over $f \in \mathcal{F}(R_{k + l})$ verifying $f(t) = s_1 \cdot s_2$ by the infimum over $f_1 \in \mathcal{F}(R_{k})$, $f_2 \in \mathcal{F}(R_{l})$ verifying $f_1(t) = s_1$, $f_2(t) = s_2$.
		Moreover, by the submultiplicativity of $N_0$, $N_1$, for any $f_1 \in \mathcal{F}(R_{k})$ and $f_2 \in \mathcal{F}(R_{l})$, the following bound holds
		\begin{multline}
			 \sup_{y \in \mathbb{R}} \max \big( \| f_1 \cdot f_2 (iy) \|_{N_{k + l, 0}}, \; \| f_1 \cdot f_2 (1 + iy) \|_{N_{k + l, 1}} \big)
			 \\
			\leq
			 \sup_{y \in \mathbb{R}}  \max \big( \| f_1(iy) \|_{N_{k, 0}}, \; \| f_1(1 + iy) \|_{N_{k, 1}} \big)
			 \\
			\cdot
			\sup_{y \in \mathbb{R}} \max \big( \| f_2(iy) \|_{N_{l, 0}}, \; \| f_2(1 + iy) \|_{N_{l, 1}} \big).
		\end{multline}
		Combining all of the above bounds yields the bound $\| s_1 \cdot s_2 \|_{N_{k + l, t}} \leq \| s_1 \|_{N_{k, t}} \cdot \| s_2 \|_{N_{l, t}}$.
	\end{proof}
	\par 
	\begin{proof}[Proof of Theorem \ref{cor_appr_geod}]
		We conserve the notation from Theorem \ref{cor_appr_geod}.
		For any $t \in [0, 1]$, we introduce the graded norm $N_t := (N_{k, t})_{k = 0}^{\infty}$.
		By Lemma \ref{lem_cx_inter_sm}, $N_t$ is submultiplicative for any $t \in [0, 1]$, as it is submultiplicative for $t = 0, 1$.
		Moreover, we claim that the norm $N_t$ is bounded for any $t \in [0, 1]$.
		Indeed, we fix a continuous metric $h^L$ on $L$, and let $C > 0$ be such that for any $k \in \nat^*$, $t = 0, 1$, we have $\ban^{\infty}_k(h^L) \cdot \exp(-C k) \leq N_{k, t} \leq \ban^{\infty}_k(h^L) \cdot \exp(C k)$.
		Then by the monotonicity properties of complex interpolation explained after (\ref{eq_defn_cx_inter}), the above inequality continues to hold for any $t \in [0, 1]$.
		\par 
		\begin{sloppypar}
		Let us now establish that it suffices to show that for any $p \in [1, +\infty[$, $s, t \in [0, 1]$, we have 
		\begin{equation}\label{eq_mab_geod_metr_fs}
			d_p(FS(N_t)_*, FS(N_s)_*) = |t - s| \cdot d_p(FS(N_0)_*, FS(N_1)_*).
		\end{equation}
		Indeed, since Mabuchi geodesics are unique metric geodesics with respect to the Mabuchi-Darvas distances, $d_p$, $p \in ]1, +\infty[$, see Theorem \ref{thm_mab_geod}, and by Theorem \ref{thm_conv_fs}, $FS(N_0)_* = P(h^L_0)$, $FS(N_1)_* = P(h^L_1)$, we see that (\ref{eq_mab_geod_metr_fs}) implies that $FS(N_t)_*$, $t \in [0, 1]$, is the Mabuchi geodesic between $P(h^L_0)$ and $P(h^L_0)$.
		This clearly finishes the proof as by Lemma \ref{lem_fs_sm} and Remark \ref{rem_conv_out_pp}, $FS(N_{k, t})^{\frac{1}{k}}$ converge towards $FS(N_t)_*$, as $k \to \infty$, outside of a pluripolar subset.
		\end{sloppypar}
		\par 
		Now, let us establish (\ref{eq_mab_geod_metr_fs}).
		First of all, by the aforementioned submultiplicativity and boundedness, we can apply Theorem \ref{thm_char}, which gives $N_t \sim_p \ban^{\infty}(FS(N_t)_*)$ for any $t \in [0, 1]$, $p \in [1, +\infty[$.
		Then by Theorem \ref{thm_isom_reg} and Proposition \ref{prop_fs_regul_blw}, we deduce that 
		\begin{equation}\label{eq_dp_fs_geod_isom}
			\lim_{k \to \infty} \frac{d_p(N_{k, t}, N_{k, s})}{k}
			=
			d_p(FS(N_t)_*, FS(N_s)_*).
		\end{equation}
		Let us establish that for any $k \in \nat$, the path $[0, 1] \ni t \mapsto N_{k, t}$ is not very far from the metric geodesic in the sense that the following estimate holds
		\begin{equation}\label{eq_dp_fs_geod_isom2}
			\Big| 
				d_p(N_{k, t}, N_{k, s})
				-
				|t - s|
				\cdot
				d_p(N_{k, 0}, N_{k, 1})
			\Big|
			\leq
			30 \log n_k.
		\end{equation}
		Once (\ref{eq_dp_fs_geod_isom2}) is established, it would finish the proof of (\ref{eq_mab_geod_metr_fs}) by (\ref{eq_dp_fs_geod_isom}).
		\par 
		To establish (\ref{eq_dp_fs_geod_isom2}), we replace complex interpolation by the geodesics between Hermitian norms. 
		For this, let $H_{k, 0}$, $H_{k, 1}$ be Hermitian norms on $H^0(X, L^{\otimes k})$, associated by (\ref{eq_john_ellips}) to $N_{k, 0}$, $N_{k, 1}$, i.e., such that for $t = 0, 1$, we have
		\begin{equation}\label{eq_nkt_hkt_compl}
			H_{k, t} \leq N_{k, t} \leq H_{k, t} \cdot \sqrt{n_k}.
		\end{equation}
		We denote by $H_{k, t}$, $t \in [0, 1]$, the complex interpolation between $H_{k, 0}$ and $H_{k, 1}$.
		Then by (\ref{eq_nkt_hkt_compl}) and the monotonicity properties of complex interpolation explained after (\ref{eq_defn_cx_inter}), we deduce that (\ref{eq_nkt_hkt_compl}) continues to hold for any $t \in [0, 1]$.
		However, as recalled in Section \ref{sect_quas_metr}, for Hermitian norms, complex interpolation produces metric geodesics for $d_p$.
		Hence, we have
		\begin{equation}\label{eq_dp_fs_geod_isom3}
			d_p(H_{k, t}, H_{k, s})
			=
			|t - s|
			\cdot
			d_p(H_{k, 0}, H_{k, 1}).
		\end{equation}
		The estimate (\ref{eq_dp_fs_geod_isom2}) now follows from (\ref{eq_gen_triangle}), (\ref{eq_nkt_hkt_compl}) and (\ref{eq_dp_fs_geod_isom3}).
	\end{proof}
	\par 
	Let us now formulate an $L^2$-version of Theorem \ref{cor_appr_geod} in the spirit of Theorem \ref{thm_appr_geod_ample}.
	We still follow the notation from Theorem \ref{cor_appr_geod}.
	For any $k \in \nat$, $t \in [0, 1]$, we consider the geodesic $H_{k, t}$ between $\hilb_k(h^L_0, dV_X)$ and $\hilb_k(h^L_1, dV_X)$, as described after (\ref{dist_norm_fins_expl}).
	\begin{prop}\label{prop_appr_geod_big_l2}
		For any $t \in [0, 1]$, the sequence of metrics $FS(H_{k, t})^{\frac{1}{k}}$ on $L$ converges outside of a pluripolar subset towards $h^L_t$, as $k \to \infty$.
	\end{prop}
	\begin{proof}
		By Proposition \ref{prop_bm_volume} and the reasoning as in (\ref{eq_nkt_hkt_compl}), we obtain that for any $\epsilon > 0$, there is $k_0 \in \nat$, so that for any $t \in [0, 1]$, we have
		\begin{equation}
			H_{k, t} \cdot \exp(-\epsilon k) \leq N_{k, t} \leq H_{k, t} \cdot \exp(\epsilon k).
		\end{equation}
		The result now follows from the above estimate, Theorem \ref{cor_appr_geod} and (\ref{eq_fs_mono}).
	\end{proof}
	\par
	We conclude this section by discussing certain consequences concerning the quatization of the rooftop norms introduced in Section \ref{sect_comp_max_norm}. 
	We borrow the notations from Section \ref{sect_comp_max_norm} below.
	\begin{prop}\label{prop_rftop_quant}
		For any continuous metrics $h^L_0$, $h^L_1$ on $L$, and any smooth volume form $dV_X$ on $X$, the following holds. For any $\epsilon > 0$, there is $k_0 \in \nat$ such that for any $k \geq k_0$, we have
		\begin{multline}\label{eq_prop_rftop_quant1}
			\hilb_k(\max\{ h^L_0, h^L_1 \}, dV_X) 
			\cdot
			\exp(-\epsilon k)
			\\
			\leq
			\hilb_k(h^L_0, dV_X) \vee \hilb_k(h^L_1, dV_X)
			\\
			\leq
			\hilb_k(\max\{ h^L_0, h^L_1 \}, dV_X) 
			\cdot
			\exp(\epsilon k).
		\end{multline}
		In particular, uniformly over compact subsets of $X \setminus \mathbb{B}_+(L)$, we have
		\begin{equation}\label{eq_prop_rftop_quant2}
			\lim_{k \to \infty} FS(\hilb_k(h^L_0, dV_X) \vee \hilb_k(h^L_1, dV_X))^{\frac{1}{k}}
			=
			P(\max\{ h^L_0, h^L_1 \}).
		\end{equation}
	\end{prop}	 
	\begin{rem}
		For more singular $h^L_0, h^L_1$ with psh potentials on ample $L$, a weaker form of the above result was established by Darvas-Lu-Rubinstein \cite[Theorem 1.3]{DarvLuRub} using different techniques.
	\end{rem}
	\begin{proof}
		First of all, we note that for any $k \in \nat$, we have the following obvious identity
		\begin{equation}
			\max \Big\{ 
				\ban^{\infty}_k(h^L_0), \ban^{\infty}_k(h^L_1)
			\Big\}
			=
			\ban^{\infty}_k(\max\{ h^L_0, h^L_1 \}).
		\end{equation}
		By this, Proposition \ref{prop_bm_volume} applied for $h^L_0$, $h^L_1$, and (\ref{eq_max_vee_comp}), we then establish
		\begin{multline}
			\ban^{\infty}_k(\max\{ h^L_0, h^L_1 \}) 
			\cdot
			\exp(-\epsilon k)
			\\
			\leq
			\hilb_k(h^L_0, dV_X) \vee \hilb_k(h^L_1, dV_X)
			\\
			\leq
			\ban^{\infty}_k(\max\{ h^L_0, h^L_1 \}) 
			\cdot
			\sqrt{2}.
		\end{multline}
		Then (\ref{eq_prop_rftop_quant1}) follows immediately from an additional application of Proposition \ref{prop_bm_volume}, but now for the metric $\max\{ h^L_0, h^L_1 \}$.
		The second result of Proposition \ref{prop_rftop_quant} is an immediate consequence of the first one and the result of Berman \cite[Theorem 1.4]{BermanEnvProj}, cf. the proof of Theorem \ref{thm_conv_fs}.
	\end{proof}
	
	\subsection{Geodesic rays and submultiplicative filtrations}\label{sect_geod_ray}
	The main goal of this section is twofold: first, to extend the construction of the geodesic ray associated with a filtration from the ample case to the big case; and second, to show that the speed of this geodesic ray captures the statistical invariants of the filtration.
	\par 
	We begin with some necessary preliminaries from linear algebra: we discuss two different constructions of rays of norms associated with a filtration.	
	\par 
	We fix a finitely-dimensional vector space $V$, $\dim V = v$, and a filtration $\mathcal{F}$ of $V$.
	We also fix a \textit{Hermitian} norm $H_0$ on $V$. 
	Consider an orthonormal basis $s_1, \ldots, s_v$, of $V$, adapted to the filtration $\mathcal{F}$, i.e., verifying $s_i \in \mathcal{F}^{e_{\mathcal{F}}(i)} V$, where $e_{\mathcal{F}}(i)$ are defined as in (\ref{eq_defn_jump_numb}).
	We define the ray of Hermitian norms $H_t$, $t \in [0, +\infty[$, on $V$ by declaring the basis 
	\begin{equation}\label{eq_bas_st}
		(s_1^t, \ldots, s_v^t) := \big( e^{t e_{\mathcal{F}}(1)} s_1, \ldots, e^{t e_{\mathcal{F}}(v)} s_v \big),
	\end{equation}
	to be orthonormal with respect to $H_t$.
	\par 
	It is immediate to verify that the resulting ray is geodesic with respect to the distances $d_p$, $p \in [1, +\infty[$, defined in Section \ref{sect_quas_metr}.
	More specifically, for any $t, s \in [0, +\infty[$, we have
	\begin{equation}\label{eq_dp_gr_e_jumppp}
		d_p (H_t, H_s)
		=
		|t - s| \cdot
		\sqrt[p]{ \frac{\sum_{i = 1}^{v} |e_{\mathcal{F}}(i)|^p}{v}}.
	\end{equation}
	\par
	We now introduce an alternative construction of the ray of norms, originating in \cite{FinNarSim}. 
	This construction works for an arbitrary (not necessarily Hermitian) norm $N_0$ on $V$. 
	We define the ray of norms $N_t$, $t \in [0, +\infty[$, emanating from $N_0$, as follows:
	\begin{equation}\label{eq_ray_norm_defn0}
		\| f \|_{N_t} :=
		\inf 
		\Big\{
			\sum
			e^{- t \mu_i} \cdot 
			\| f_i \|_{N_0}
			\,
			:
			\,
			f = \sum f_i, \, f_i \in \mathcal{F}^{\mu_i} V
		\Big\}.
	\end{equation}
	The following result compares the two constructions.
	\begin{lem}[{ \cite[Lemma 2.11]{FinNarSim} }]\label{lem_two_norms_comp0}
		For any (resp. Hermitian) norm $N_0$ (resp. $H_0$) on $V$ and any $t \in [0, +\infty[$, we have
		\begin{equation}
			d_{+ \infty}(H_t, N_t)
			\leq
			d_{+ \infty}(H_0, N_0)
			+
			\log v,
		\end{equation}
		where $d_{+ \infty}(N_0, N_1)$ for two norms $N_0$, $N_1$ on $V$ is defined as the minimal $C > 0$ such that $N_0 \cdot \exp(-C) \leq N_1 \leq N_0 \cdot \exp(C)$, and $H_t$ (resp. $N_t$) is the ray of norms emanating from $H_0$ (resp. $N_0$) and associated with $\mathcal{F}$ as in the construction (\ref{eq_bas_st}) (resp. (\ref{eq_ray_norm_defn0})).
	\end{lem}
	We have the following property, initially observed in \cite[\S5]{FinNarSim}.
	\begin{lem}\label{lem_sm_ray}
		Assume that $R = \bigoplus R_k$ is a graded algebra such that $\dim R_k$ is finite.
		Let $N = (N_k)_{k = 0}^{\infty}$ be a submultiplicative norm on $R$, and let $\mathcal{F} = (\mathcal{F}_k)_{k = 0}^{\infty}$ be a submultiplicative filtration on $R$.
		Then, for any $t \in [0, +\infty[$, the norm $N_{t} := (N_{k, t})_{k = 0}^{\infty}$, constructed from the rays $N_{k, t}$ emanating from $N_k$ and associated with the filtration $\mathcal{F}_k$ as in (\ref{eq_ray_norm_defn0}), is submultiplicative.
	\end{lem}
	\par 
	We fix a bounded submultiplicative filtration $\mathcal{F} = (\mathcal{F}_k)_{k = 0}^{\infty}$ and a continuous metric $h^L$.
	We denote $N_0 := (N_{k, 0})_{k = 0}^{\infty} = \ban^{\infty}(h^L)$, and for $t \in [0, +\infty[$, define $N_t := (N_{k, t})_{k = 0}^{\infty}$ such that $N_{k, t}$ is the ray of norms emanating from $N_{k, 0}$ and associated with $\mathcal{F}_k$ as in the construction (\ref{eq_ray_norm_defn0}).
	It is immediate to see that $N_t$ is a bounded metric.
	Indeed, for $C > 0$ as in (\ref{eq_bnd_filt}), immediately from (\ref{eq_ray_norm_defn0}), we have
	\begin{equation}\label{eq_bound_ray_ccc}
		\ban^{\infty}_k(h^L) \cdot \exp(-C k)
		\leq
		N_{k, t} 
		\leq 
		\ban^{\infty}_k(h^L) \cdot \exp(C k).	
	\end{equation}
	In particular, by Lemmas \ref{lem_fs_sm} and \ref{lem_sm_ray}, the metric $FS(N_{t})$ is well-defined.
	\par 
	\begin{thm}\label{thm_const_geod_ray_sm}
		The ray of metrics $h^L_t := FS(N_t)_*$, $t \in [0, +\infty[$, consists of regularizable from above psh metrics, it emanates from $P(h^L)$, and it constitutes a geodesic ray with respect to any of the $d_p$-distances, $p \in [1, +\infty[$.
	\end{thm}
	\begin{proof}
		The results follows from Theorem \ref{thm_conv_fs} and Proposition \ref{prop_fs_regul_blw}, only the part concerning the geodesic ray needs clarification.
		The proof is somewhat similar to the proof of Theorem \ref{cor_appr_geod}.
		\par 
		We need to show that for any $p \in [1, +\infty[$, $s, t \in [0, +\infty[$, under our new notations, we have (\ref{eq_mab_geod_metr_fs}).
		First of all, by Theorem \ref{thm_char}, Lemma \ref{lem_sm_ray} and (\ref{eq_bound_ray_ccc}), we have $N_t \sim \ban^{\infty}(FS(N_t)_*)$.
		Then by Theorem \ref{thm_isom_reg} and Proposition \ref{prop_fs_regul_blw}, we deduce that, under our new notations, (\ref{eq_dp_fs_geod_isom}) holds for any $s, t \in [0, +\infty[$.
		Let us establish that for any $k \in \nat$, the path $[0, +\infty[ \ni t \mapsto N_{k, t}$ resembles a metric geodesic ray in the sense that the analogue of (\ref{eq_dp_fs_geod_isom2}) holds under our new notations.
		Once this is established, it would finish the proof exactly as in the proof of Theorem \ref{cor_appr_geod}.
		\par 
		For this, let $H_{k, 0}$, be a Hermitian norm on $H^0(X, L^{\otimes k})$, associated by (\ref{eq_john_ellips}) to $N_{k, 0}$, i.e., such that for $t = 0$, we have (\ref{eq_nkt_hkt_compl}) under our new notations.
		We denote by $H_{k, t}$, $t \in [0, 1]$, the ray of norms emanating from $H_{k, 0}$ and associated with $\mathcal{F}_k$ as in (\ref{eq_bas_st}).
		Then by Lemma \ref{lem_two_norms_comp0}, we deduce 
		\begin{equation}\label{eq_nkt_hkt_comp_rays}
			H_{k, t}
			\cdot
			n_k^{-1}
			\leq
			N_{k, t}
			\leq
			H_{k, t}
			\cdot
			n_k^{7}.
		\end{equation}
		However, immediately from the definitions, we see that $H_{k, t}$, $t \in [0, +\infty[$, is a metric geodesic for $d_p$, i.e., (\ref{eq_dp_fs_geod_isom3}) holds under our new notations.
		Using the weak triangle inequality (\ref{eq_gen_triangle}) and the above estimate, we deduce 
		\begin{equation}\label{eq_dp_nkt_hkt}
			\Big| 
				d_p(N_{k, t}, N_{k, s})
				-
				d_p(H_{k, t}, H_{k, s})
			\Big|
			\leq
			50 \log n_k.
		\end{equation}
		The proof is now finished in exactly the same manner as in Theorem \ref{cor_appr_geod}.
	\end{proof}	 
	\par 
	Recall that in the ample setting, for finitely generated filtrations, Phong-Sturm \cite{PhongSturmTestGeodK} proposed an alternative construction of geodesic rays using the geometric quantization of $L^2$-norms.
	Later, still within the context of finitely generated filtrations, they \cite{PhongSturmRegul} presented a construction based on solutions to homogeneous Monge-Ampère equations on the resolution of a test configuration, and demonstrated that this approach agrees with their earlier construction.
	\par 
	For arbitrary bounded submultiplicative filtrations in the ample setting, Ross-Witt Nyström \cite{RossNystAnalTConf} constructed geodesic rays using the language of test curves, and verified that for finitely generated filtrations, it recovers the construction from \cite{PhongSturmTestGeodK}. 
	Boucksom-Jonsson \cite{BouckJohn21} later offered an alternative construction based on valuations and non-Archimedean geometry.
	\par 
	The construction given in Theorem \ref{thm_const_geod_ray_sm} in the ample setting was carried out by the author in \cite{FinNarSim}, where it was also verified that it coincides with the constructions of Boucksom-Jonsson and Ross-Witt Nyström.
	The equivalence between the latter two follows also from Darvas-Xia \cite{DarvXiaTest}.
	\par 
	One may reasonably expect that, in the big setting, in addition to the construction from Theorem \ref{thm_const_geod_ray_sm}, all of the aforementioned constructions can also be performed, and that a similar equivalence statement holds. 
	We discuss one part of this story in Proposition \ref{prop_two_rays_ccccc}, but for brevity, we leave a formal statement and verification of the full claim to future work.
	\par 
	\begin{proof}[Proof of Theorem \ref{thm_filt}]
		The proof repeats verbatim the one given in the ample setting from \cite{FinNarSim}, the only difference is that one has to use the big versions of Proposition \ref{thm_isom_ample} and Theorem \ref{thm_char_ample}.
		For the convenience of the reader, we reproduce the argument below.
		\par
		Let us first remark that it is enough to establish Theorem \ref{thm_filt} in the special case when the filtration $\mathcal{F}$ satisfies the additional assumption 
		\begin{equation}\label{eq_filt_bnd_zero}
			\mathcal{F}^0 R(X, L) = \{0\}.
		\end{equation}
		To see this, remark that since $\mathcal{F}$ is bounded, there is $C > 0$, verifying $\mathcal{F}^{C k} H^0(X, L^{\otimes k}) = \{0\}$.
		Consider now another filtration $\mathcal{F}_0$ on $R(X, L)$, defined for any $k \in \nat$, $\lambda \in \real$, as follows $\mathcal{F}_0^{\lambda} H^0(X, L^{\otimes k}) = \mathcal{F}^{\lambda + Ck} H^0(X, L^{\otimes k})$.
		Clearly, $\mathcal{F}_0$ is submultiplicative and bounded whenever $\mathcal{F}$ is submultiplicative and bounded.
		An easy verification shows that establishing Theorem \ref{thm_filt} for $\mathcal{F}_0$ and $\mathcal{F}$ is equivalent. We, hence, assume from now on that $\mathcal{F}$ satisfies (\ref{eq_filt_bnd_zero}).
		\par 
		By (\ref{eq_bas_st}), (\ref{eq_dp_nkt_hkt}) and (\ref{eq_filt_bnd_zero}), we deduce that
		\begin{equation}\label{eq_dp_jump_norm2}
			\Big| 
			d_p(N_{k, 0}, N_{k, 1})
			-
			k
			\cdot
			\sqrt[p]{
			\int_{x \in \real} (-x)^p  d \mu_{\mathcal{F}, k}(x)}
			\Big|
			\leq
			50 \log n_k.
		\end{equation}
		As in (\ref{eq_mab_geod_metr_fs}), from (\ref{eq_dp_gr_e_jumppp}) and (\ref{eq_dp_jump_norm2}), we then deduce 
		\begin{equation}\label{eq_dp_jump_norm3}
			d_p(FS(N_0)_*, FS(N_1)_*)
			=
			\lim_{k \to \infty}
			\sqrt[p]{\int_{x \in \real} (-x)^p  d \mu_{\mathcal{F}, k}(x)}.
		\end{equation}
		Since $N_{k, t} \geq N_{k, 0}$, we have $h^L_t \geq h^L_0$, and so $\dot{h}^L_0 \leq 0$.
		From Theorem \ref{thm_mab_geod} and (\ref{eq_dp_jump_norm3}), for any $p \in [1, +\infty[$, we then obtain
		\begin{equation}\label{eq_dp_jump_norm4}
			\int_{x \in \real} x^p d \mu_{\mathcal{F}}(x)
			=
			\lim_{k \to \infty}
			\int_{x \in \real} x^p  d\mu_{\mathcal{F}, k}(x).
		\end{equation}
		Note that the measures $\mu_{\mathcal{F}, k}$, $k \in \nat$, all have supports contained in the same compact subset.
		As measures of compact support are characterized by their moments, (\ref{eq_dp_jump_norm4}) implies that $\mu_{\mathcal{F}, k}$ converge weakly, as $k \in \nat$, towards $\mu_{\mathcal{F}}$, which finishes the proof.
	\end{proof}		
	
	We conclude the paper by giving an alternative definition of $h^L_t$ in the spirit of Phong-Sturm \cite{PhongSturmTestGeodK} and Ross-Witt Nyström \cite{RossNystAnalTConf}.
	For this, for a continuous metric $h^L$ on $L$ and a smooth volume form $dV_X$ on $X$, we denote by $H_{k, t}$, $t \in [0, 1]$, the ray of Hermitian norms emanating from $\hilb_k(h^L, dV_X)$ and associated with $\mathcal{F}_k$ as in (\ref{eq_bas_st}).
	\begin{prop}\label{prop_two_rays_ccccc}
		For any $t \in [0, +\infty[$, the sequence of metrics $FS(H_{k, t})^{\frac{1}{k}}$ converges, as $k \to \infty$, towards a singular metric $h^L_{t, 0}$, so that for the ray $h^L_t$ from Theorem \ref{thm_const_geod_ray_sm}, we have $h^L_t = (h^L_{t, 0})_*$.
	\end{prop}
	\begin{proof}
		By Proposition \ref{prop_bm_volume} and the reasoning as in (\ref{eq_nkt_hkt_comp_rays}), we obtain that for any $\epsilon > 0$, there is $k_0 \in \nat$, so that for any $t \in [0, 1]$, we have
		\begin{equation}
			H_{k, t} \cdot \exp(-\epsilon k) \leq N_{k, t} \leq H_{k, t} \cdot \exp(\epsilon k).
		\end{equation}
		The result now follows from the above estimate, Lemma \ref{lem_fs_sm} and (\ref{eq_fs_mono}).
	\end{proof}

\bibliography{bibliography}

		\bibliographystyle{abbrv}

\Addresses

\end{document}